\documentclass[a4paper,11pt,reqno,noindent]{amsart}
\usepackage[centertags]{amsmath}
\usepackage{amsfonts,amssymb,amsthm} 
\usepackage[dvips]{graphicx} 
\usepackage{psfrag}
\usepackage[english]{babel}
\usepackage{newlfont}
\usepackage{color}
\usepackage[body={15cm,21.5cm},centering]{geometry} 
\usepackage{fancyhdr}
\usepackage[active]{srcltx}
\usepackage{esint}
\usepackage{hyperref}
\newtheorem{theo}{Theorem}
\newtheorem{lemma}{Lemma}
\newtheorem{prop}{Proposition}
\newtheorem{corollary}{Corollary}

\theoremstyle{definition}
\newtheorem{rem}{Remark}
\newtheorem{defi}{Definition}
\newtheorem{hypo}{Hypothesis}

\numberwithin{equation}{section}

\newcommand{\sym}{\mathrm{sym}}
\newcommand{\loc}{\mathrm{loc}}

\newcommand{\ho}{\mathrm{hom}}
\newcommand{\per}{\mathrm{per}}

\newcommand{\R}{\mathbb{R}}
\newcommand{\Z}{\mathbb{Z}}
\newcommand{\N}{\mathbb{N}}
\newcommand{\Md}{\mathcal{M}_d(\R)}
\newcommand{\Id}{\text{Id}}

\newcommand{\e}{\varepsilon}

\newcommand{\ee}{\mathbf{e}}

\newcommand{\calT}{\mathcal{T}}
\newcommand{\calP}{\mathcal{P}}
\newcommand{\calM}{\mathcal{M}}
\newcommand{\calL}{\mathcal{L}}
\newcommand{\calH}{\mathcal{H}}
\newcommand{\Mab}{\mathcal{M}_{\alpha\beta}}

\newcommand{\Mabs}{\mathcal{M}_{\alpha\beta}^\sym}
\newcommand{\DD}{\mathrm{D}}

\newcommand{\Rhodrho}{Q_{\delta\rho}\cap T_{-x}D}
\renewcommand{\P}{\mathbb{P}}

\mathchardef\emptyset="001F

\newcommand{\dig}[1]{\mathrm{diag}\left[ #1\right]}
\newcommand{\var}[1]{\mathrm{var}\left[#1\right]}

\newcommand{\osc}[2]{\mathrm{osc}_{#2}\,#1}
\newcommand{\expec}[1]{\left\langle #1 \right\rangle}

\newcommand{\step}[1]{\noindent \textit{Step} #1.}

\title[Reduction of the resonance error~II]
{Reduction of the resonance error in numerical homogenization~II: correctors and extrapolation}
\author[A. Gloria]{Antoine Gloria}
\address[Antoine Gloria]{Universit\'e Libre de Bruxelles (ULB), Brussels, Belgium \&
  Project-team MEPHYSTO, Inria Lille - Nord Europe, Villeneuve d'Ascq,
  France}
\email{agloria@ulb.ac.be}
\author[Z. Habibi]{Zakaria Habibi}
\date{\today}
\address[Zakaria Habibi]{Project-team MEPHYSTO, Inria Lille - Nord Europe, Villeneuve d'Ascq,  France}
\email{zakaria.habibi@inria.fr}

\begin{document}
%\graphicspath{{}}
\maketitle
%%----------------------------------------------------------------------------------------------------------------------

\begin{center}
\begin{minipage}{13cm}
\small{
\noindent {\bf Abstract.} 
This paper is the follow-up of \cite{Gloria-09}.
One common drawback among numerical homogenization methods is the presence of the so-called resonance error,
which roughly speaking is a function of the ratio $\frac{\e}{\rho}$, where $\rho$ is a typical macroscopic lengthscale
and $\e$ is the typical size of the heterogeneities.
In the present work, we make a systematic use of regularization and extrapolation to reduce this resonance error at the level of the approximation
of homogenized coefficients and correctors for general non-necessarily symmetric stationary ergodic coefficients.
We quantify this reduction for the class of periodic coefficients, for the Kozlov subclass of almost periodic coefficients, and for the subclass of random coefficients that satisfy a spectral gap estimate (e.g. Poisson random inclusions).
We also report on a systematic numerical study in dimension 2, which
demonstrates the efficiency of the method and the sharpness of the analysis.
Last, we combine this approach to numerical homogenization methods, prove
the asymptotic consistency in the case of locally stationary ergodic coefficients
and give quantitative estimates in the case of periodic coefficients.
}

\vspace{10pt}
\noindent {\bf Keywords:} 
numerical homogenization, resonance error, effective coefficients, correctors, periodic, almost periodic, random.
%
%\vspace{6pt}
%\noindent {\bf 2000 Mathematics Subject Classification:} 35B27, 60F99.}
% id : homogenization, limit theorems

\end{minipage}
\end{center}

\bigskip

%%%%%%%%%
\tableofcontents
\section{Introduction}

\noindent Numerical homogenization methods
are designed to solve partial differential equations for which the operator
is heterogeneous spatially at a small scale. 
Such problems arise in many applications such as diffusion in porous media or composite materials.
By numerical homogenization, we mean that we compute not only the mean field solution
of the heterogeneous problem, but also the local fluctuations, which may be important 
in many applications.
In the recent years, two main types of methods were introduced: methods which rely on homogenization to reduce the complexity of the problem (see \cite{Arbogast-00}, \cite{Hou-97,Hou-99}, \cite{E-05,E-07}, and \cite{Gloria-06b} e.g.), and methods which rely on the structure of the solution space independently of homogenization structures (see \cite{Owhadi-Zhang-07,Owhadi-Zhang-12,Owhadi-Zhang-Berlyand-13}, \cite{Henning-Peterseim-13}, \cite{Babuska-Lipton-11}).
Both approaches are nonlinear in the sense that the vector space in which the solution is approximated depends on the operator itself. 

\medskip

\noindent
In this article we focus on methods based on the homogenization theory.
More precisely, we shall address the following prototypical scalar linear elliptic equation: for some $1\gg \e>0$, 
find $u_\e\in H^1_0(D)$ such that
\begin{equation}\label{eq:Aeps}
-\nabla \cdot A_\e \nabla u_\e \,=\,f_\e \qquad \text{in }D,
\end{equation}
on a domain $D$ with $f_\e\in H^{-1}(D)$.
Here, the spatial dependence of the operator is encoded in the function $A_\e$, whose frequencies are of order $\e^{-1}$.
Academic examples are of the form: $A_\e(x)=A(\frac{x}{\e})$, with $A$ periodic,
almost periodic or stationary in an ergodic stochastic setting.
More realistic models can be of the form: $A_\e(x)=A(x,\frac{x}{\e})$, where $A(x,\cdot)$
may be periodic, almost periodic or stationary for all $x\in D$, provided some suitable cross-regularity
holds (see \cite{Allaire-92,Bourgeat-Mikelic-Wright-94} e.g.).
In all these examples, $\e$ refers to the \emph{actual} lengthscale of the heterogeneities.
Numerical homogenization methods suffer from the so-called resonance error, which is roughly speaking related to scale separation.
Since this error often dominates the error due to discretization (see for instance \cite[Section~4]{Abdulle-05} for a related discussion), its reduction is one of the biggest challenges in numerical homogenization of linear elliptic PDEs.
For general coefficients this was successfully achieved in  \cite{Owhadi-Zhang-Berlyand-13} and \cite{Henning-Peterseim-13}. Yet the price to pay is the ``oversampling" size (see below for details). In the present contribution we take advantage of homogenization structures (such as stationary ergodic coefficients) to reduce the resonance error without increasing significantly the oversampling size.
The aim of this paper is to introduce a general method which is consistent with homogenization and successfully reduces the resonance error in three main academic examples of heterogeneous media: the class of periodic coefficients, the Kozlov subclass of almost periodic coefficients, and the class of random coefficients that satisfy a spectral gap inequality (e.g. random Poisson inclusions).

\medskip
\noindent To this end we shall improve the approach proposed in \cite{Gloria-09} (and
further developed in \cite{Gloria-Mourrat-10}) to reduce the resonance error, and extend it
to the approximation of correctors, which are the central objects of the homogenization theory.
This will allow us to combine this method of reduction of the resonance error with standard numerical homogenization methods. 
Some of the results were announced in \cite{Gloria-12b}.
The starting point in \cite{Gloria-09} is the effective computation of 
\begin{equation}\label{eq:asymp-fo}
\xi'\cdot A_\ho\xi\,=\,\mathcal{M}\Big((\xi'+\nabla \phi')\cdot A(\xi+\nabla \phi)\Big),
\end{equation}
where $\xi,\xi'\in \R^d$, $\phi,\phi'$ are solutions (in a suitable sense) of
\begin{equation}\label{eq:corr}
\begin{array}{rcl}
-\nabla \cdot A(\xi+\nabla \phi)&=&0\qquad \text{in }\R^d,\\
-\nabla \cdot A^*(\xi'+\nabla \phi')&=&0\qquad \text{in }\R^d,
\end{array}
\end{equation}
and $A^*$ denotes the transpose of $A$, and 
where  $\mathcal{M}(\cdot)$ is the average operator on $\R^d$:
\begin{equation}\label{eq:average}
\mathcal{M}(\mathcal{E})\,:=\,\liminf_{R \uparrow\infty} \fint_{Q_R}\mathcal{E}(x)dx,
\end{equation}
with $Q_R:=(-R,R)^d$.
Such quantities are well-defined in periodic, almost periodic and stochastic homogenization, in which case the $\liminf$ in \eqref{eq:average} is a limit (see \cite{JKO-94} e.g.).
The naive approach to approximate \eqref{eq:asymp-fo} consists in replacing $\phi,\phi'$ by $\phi_R,\phi_R'$,  weak solutions in $H^1_0(Q_R)$ of
\begin{equation*}
\begin{array}{rcl}
-\nabla \cdot A(\xi+\nabla \phi_R)&=&0\qquad \text{in }Q_R,\\
-\nabla \cdot A^*(\xi'+\nabla \phi_R')&=&0\qquad \text{in }Q_R,
\end{array}
\end{equation*}
and the average on $\R^d$ by an average on $Q_R$, which yields the approximation $A_R$ of $A_{\hom}$:
\begin{equation}\label{eq:intro-approx1}
\xi'\cdot A_R\xi\,:=\, \fint_{Q_R} (\xi'+\nabla \phi_R'(x))\cdot A(x)(\xi+\nabla \phi_R(x))dx.
\end{equation}
In the case of periodic coefficients $A$, it is elementary to show that
\begin{equation}\label{eq:intro-conv1}
|A_R-A_{\hom}|\,\lesssim \, R^{-1}.
\end{equation}
This is what we (slightly abusively) called resonance error in \cite{Gloria-09}, in line with the classical resonance error originally introduced 
in \cite{Hou-97} and that concerns the approximation of $\nabla \phi,\nabla \phi'$ by $\nabla \phi_R,\nabla \phi_R'$,
see below for more details.
To reduce this error, we have introduced the following alternative to approximate $A_{\hom}$.
Let $\phi_{T,R},\phi_{T,R}'$ be the unique weak solutions in $H^1_0(Q_R)$ of
\begin{equation*}
\begin{array}{rcl}
T^{-1}\phi_{T,R}-\nabla \cdot A(\xi+\nabla \phi_{T,R})&=&0\qquad \text{in }Q_R,\\
T^{-1}\phi_{T,R}'-\nabla \cdot A^*(\xi'+\nabla \phi_{T,R}')&=&0\qquad \text{in }Q_R,
\end{array}
\end{equation*}
where $T>0$ controls the importance of the zero-order term and $R>0$ is the size of the finite domain $Q_R$. 
The homogenized coefficients are approximated by taking the filtered average
\begin{equation}\label{eq:homTRL}
\xi\cdot A_{T,R,L}\xi\,:=\,\int_{Q_R} (\xi'+\nabla \phi_{T,R}'(x))\cdot A(x) (\xi+\nabla \phi_{T,R}(x))\mu_L(x)dx,
\end{equation}
where $\mu_L$ is a smooth averaging function (whose properties will be detailed later on, see Definition~\ref{def:mask}) compactly supported in $Q_L=(-L,L)^d$, $L\leq R$, and of total mass 1.
Then, as showed in \cite[Theorem~1]{Gloria-09} for symmetric coefficients (the general case will be treated here), for
\begin{itemize}
\item $T=L^2 (\log L)^{-4}$,
\item $R=\frac{3}{2}L$,
\end{itemize}
and an averaging function of order at least 3,  we have for periodic coefficients
\begin{equation}\label{eq:opt-error-per}
|A_{T,R,L}-A_\ho|\,\lesssim\,R^{-4}(\log R)^8,
\end{equation}
which is much better than \eqref{eq:intro-conv1}.
This approach also yields improvements over the naive approach in the case of random independent and identically distributed coefficients for discrete elliptic equations, see 
\cite{Gloria-10}. 
Yet, it is not clear in general whether the approximation \eqref{eq:homTRL} remains uniformly elliptic for all values of the parameters, which is an issue of importance in practice.
In this contribution we shall show that a variant of \eqref{eq:homTRL} ensures the a priori uniform ellipticity of the approximation when the coefficients are symmetric (without changing in a quantitative way the convergence properties) and that Richardson extrapolations of $\nabla \phi_{T,R},\nabla \phi_{T,R}'$ with respect to $T$ further reduce the resonance error.

\medskip
\noindent Let us now go back to the original problem \eqref{eq:Aeps}, and show how the method introduced in \cite{Gloria-09}
can be used to reduce the ``classical" resonance error in the numerical solution of \eqref{eq:Aeps}.
We first assume that $A_\e$ satisfies a homogenization property,
that is there exist uniformly elliptic (not necessarily constant) coefficients $A_{\hom}$ such that the solution 
$u_\e$ of \eqref{eq:Aeps} and its flux $A_\e \nabla u_\e$ converge weakly in $H^1(D)$ and in $L^2(D)$, respectively, to the unique weak solution $u_{\hom}$ in $H^1_0(D)$ of
\begin{equation}\label{eq:intro2}
-\nabla \cdot A_\ho \nabla u_\ho\,=\, f\quad \text{ in }D
\end{equation}
and to the homogenized flux $A_\ho \nabla u_\ho$, respectively,
for all suitable $f$.
Such a homogenization property typically takes place when $A_\e$ is the combination
of a smooth function and an oscillating part at scale $\e>0$.
Unfortunately, $A_\ho$ is not explicit in general.
Its approximation is one part of numerical homogenization --- the other part being the approximation
of the local fluctuations of $\nabla u_\e$.

\medskip
\noindent A very general approach consists in averaging out the equation at a ``mesoscale" $\rho \gg \e$. 
More precisely, for all $\rho\geq \e>0$, let
$A_{\rho,\e}$ be defined by
\begin{equation}\label{eq:intro3}
\xi'\cdot A_{\rho,\e}(x)\xi\,:=\, \fint_{D\cap B(x,\rho)}(\xi'+\nabla_y \tilde \phi_{\rho,\e}'(x,y))\cdot A_\e(y)(\xi+\nabla_y \tilde \phi_{\rho,\e}(x,y))dy,
\end{equation}
where for all $x\in D$, $B(x,\rho)$ is the ball centered at $x$ and of radius $\rho$, and 
$y\mapsto \tilde \phi_{\rho,\e}(x,y),\tilde \phi_{\rho,\e}'(x,y)$ are the weak solutions in $H^1_0(D\cap B(x,\rho))$
of 
\begin{equation}\label{eq:intro3.1}
\begin{array}{rcl}
-\nabla_y \cdot A_\e(y)(\xi+\nabla_y \tilde \phi_{\rho,\e}(x,y))&=&0\qquad \text{in }D\cap B(x,\rho),\\
-\nabla_y \cdot A_\e^*(y)(\xi'+\nabla_y \tilde \phi_{\rho,\e}'(x,y)))&=&0\qquad \text{in }D\cap B(x,\rho).
\end{array}
\end{equation}
An approximation of $u_\ho$ is then given by the weak solution $u_{\rho,\e}$ in $H^1_0(D)$ of
\begin{equation}\label{eq:intro4}
-\nabla \cdot A_{\rho,\e} \nabla u_{\rho,\e}\,=\,f \quad \text{ in }D.
\end{equation}
In particular (see \cite{Gloria-12b} for the linear case, and \cite{Gloria-06b} for nonlinear operators), we have
\begin{equation}\label{eq:intro6}
\lim_{\rho\downarrow 0}\lim_{\e \downarrow 0}\|u_{\rho,\e}-u_\ho\|_{H^1(D)}\,=\,0,
\end{equation}
thus ensuring the consistency of the approach in the framework of the homogenization theory (the only assumption is that $A_\e$  H-converges towards $A_\ho$).
We may also approximate the local fluctuations of $\nabla u_\e$ on $D\cap B(x,\rho)$ by the
function $y\mapsto C_{\e,\rho}(x,y):=\sum_{i=1}^d \Big(\fint_{D\cap B(x,\rho)}\nabla_i u_{\rho,\e}(y')dy'\Big) \nabla_y \tilde \phi_{\rho,\e,i}(x,y)$,
where $\tilde \phi_{\rho,\e,i}$ denotes the solution of \eqref{eq:intro3.1} for $\xi=e_i$ (elements of the canonical basis of $\R^d$). We then have the following result (see \cite[Theorem~2]{Gloria-06b}, and \cite{Gloria-10} in the linear case):
$$
\lim_{\rho\downarrow 0}\lim_{\e \downarrow 0} \Big(\fint_{D\cap B(x,\rho)}|\nabla u_\e(y)-C_{\e,\rho}(x,y)|^2dy\Big)^{\frac{1}{2}}\,=\,0.
$$
For periodic coefficients, the argument of the limit above is of order $\sqrt{\frac{\e}{\rho}}$: this is the resonance error originally introduced in  \cite{Hou-97}.
By partitioning $D$ into subdomains of size $\rho$, this allows one to reconstruct a consistent approximation of the fluctutations of $\nabla u_\e$. 
In the case of periodic coefficients, for both the approximation of $u_{\e}$ by $u_{\e,\rho}$ and the approximation of $\nabla u_\e$ by $C_{\e,\rho}$, the error is controlled by the discrepancy between $\nabla \tilde \phi_{\rho,\e}$ and $\nabla \phi(\frac{\cdot}{\e})$ (where $\phi$ is the classical periodic corrector), which is why we use the term ``resonance error" for both quantities. 
The main difference between \eqref{eq:intro4} and \eqref{eq:Aeps} is that
$A_{\rho,\e}$ does not oscillate at scale $\e$, which is a big advantage for the
numerical practice. 
In order to make use of the strategy of \cite{Gloria-09} recalled above for the computation of
homogenized coefficients (and for approximations of local fluctuations), we change variables $x\leadsto \e^{-1}x$
to make the oscillations of $A_\e$ be of order 1.
Formula~\eqref{eq:intro3} now turns into
\begin{equation}\label{eq:intro3b}
\xi'\cdot A_{\rho,\e}(x)\xi\,:=\, \fint_{\frac{D-x}{\e}\cap B\left(0,\frac{\rho}{\e}\right)}(\xi'+\nabla_y \phi_{\rho,\e}'(x,y))\cdot A(y)(\xi+\nabla_y \phi_{\rho,\e}(x,y))dy
\end{equation}
where $A(y):=A_\e(x+\e y)$, and $y\mapsto \phi_{\rho,\e}(x,y),\phi_{\rho,\e}'(x,y)\in H^1_0\left(\frac{D-x}{\e}\cap B\left(0,\frac{\rho}{\e}\right)\right)$ are the weak solutions of 
\begin{equation}\label{eq:intro3.1b}
\begin{array}{rcl}
-\nabla_y \cdot A(y)(\xi+\nabla_y \phi_{\rho,\e}(x,y))&=&0\qquad \text{in }\frac{D-x}{\e}\cap B\left(0,\frac{\rho}{\e}\right),\\
-\nabla_y \cdot A^*(y)(\xi'+\nabla_y \phi_{\rho,\e}'(x,y)))&=&0\qquad \text{in }\frac{D-x}{\e}\cap B\left(0,\frac{\rho}{\e}\right).
\end{array}
\end{equation}
It is then clear that \eqref{eq:intro3b} can be seen as the approximation \eqref{eq:intro-approx1} of some homogenized coefficients $A_{\hom}$.
In particular, if $A$ is periodic, the error $|A_{\rho,\e}-A_\ho|$ scales as $\frac{\e}{\rho}$,
and $\Big(\fint_{\frac{D-x}{\e}\cap B\left(0,\frac{\rho}{\e}\right)}|\nabla_y \phi_{\rho,\e}(x,y)-\nabla_y \phi(y)|^2\Big)^{\frac{1}{2}}$ scales as $\sqrt{\frac{\e}{\rho}}$ (with $\phi$ the periodic corrector).
To reduce these resonance errors, techniques such as oversampling have been introduced (see e.g. \cite{Efendiev-00}, \cite{E-05}).
In particular, a variant of \eqref{eq:intro3b} is
\begin{equation}\label{eq:intro3c0}
\xi'\cdot {A}_{\rho,\e}(x)\xi\,:=\,\frac{1}{\int_{\frac{D-x}{\e}\cap B\left(0,\frac{\rho}{\e}\right)}\mu_{\rho,\e}(y)dy} \int_{\frac{D-x}{\e}\cap B\left(0,\frac{\rho}{\e}\right)}(\xi'+\nabla \phi_{\rho,\e}')
\cdot A(y)(\xi+\nabla \phi_{\rho,\e}) \mu_{\rho,\e}(y)dy,
\end{equation}
where $\mu_{\rho,\e}$ is an averaging function with support compactly included in $B\left(0,\frac{\rho}{\e}\right)$ and of total mass 1 (we still divide by the mass of $\mu_{\rho,\e}$
because we could have $B\left(0,\frac{\rho}{\e}\right))\neq \frac{D-x}{\e}\cap B\left(0,\frac{\rho}{\e}\right)$). Although this method does not change the scaling $\frac{\e}{\rho}$ of $|A_{\rho,\e}-A_\ho|$ in the periodic setting, it may reduce the prefactor of the error (the multiplicative constant in front of $\frac{\e}{\rho}$). It also improves the approximation of the corrector itself, 
and $\Big(\fint_{\frac{D-x}{\e}\cap B\left(0,\frac{\rho}{\e}\right)}|\nabla_y \phi_{\rho,\e}(x,y)-\nabla_y \phi(y)|^2\Big)^{\frac{1}{2}}$ scales as ${\frac{\e}{\rho}}$ provided $\mu_{\rho,\e}$ does not charge the boundary 
of $\frac{D-x}{\e}\cap B\left(0,\frac{\rho}{\e}\right)$; see \cite[Tables~3.1 \&~3.2]{Gloria-07b} for elementary numerical examples.
In general however, it is not clear under which conditions on $A_\e$ and $\mu_{\rho,\e}$ the coefficients $A_{\rho,\e}$ 
defined in \eqref{eq:intro3c0} are uniformly elliptic.
In order to circumvent this difficulty, reduce the resonance error, and obtain convergence rates valid for any coefficients $A_\e$,  Henning and Peterseim introduced in \cite{Henning-Peterseim-13} an original oversampling approach which amounts to solving a problem of type \eqref{eq:intro3.1} with an additional orthogonality constraint on a domain of size $\rho |\ln \rho|$ instead of $\rho$. One expects this method to be optimal in general in the sense that for some coefficients $A_\e$, the size $\rho |\ln \rho|$ of the ``oversampling region" cannot be reduced. It is however not clear at all in the analysis of \cite{Henning-Peterseim-13} whether the size can be reduced for specific coefficients which enjoy homogenization properties. A similar achievement is obtained in \cite{Owhadi-Zhang-Berlyand-13} by Owhadi, Zhang and Berlyand, and the same remark applies on the oversampling size. Progress in that direction is related as much to homogenization as to numerical analysis of linear elliptic PDEs.
The aim of this article is to introduce a numerical homogenization method for
the \emph{restricted class of locally stationary ergodic coefficients} with the following three properties:
\begin{enumerate}
\item the approximation of the homogenized coefficients is unconditionally uniformly elliptic for symmetric coefficients (there is a subtility for non symmetric coefficients);
\item the resonance error can be quantified and is reduced, both for the approximation of homogenized coefficients and for the approximation of fluctuations of the gradient of the solution;
\item the oversampling size remains of order $\rho$.
\end{enumerate}
We conclude this introduction by a quick presentation of our method.
Let $\mu_{\rho,\e}$  be a bounded averaging function.
For all $\psi \in L^1(\frac{D-x}{\e}\cap B\left(0,\frac{\rho}{\e}\right))$, set
$$
\expec{\psi}_{x,\mu_{\rho,\e}}\,:=\,\frac{1}{\int_{\frac{D-x}{\e}\cap B\left(0,\frac{\rho}{\e}\right)}
\mu_{\rho,\e}(y)dy}\int_{\frac{D-x}{\e}\cap B\left(0,\frac{\rho}{\e}\right)}
\psi(y)\mu_{\rho,\e}(y)dy.
$$
We replace \eqref{eq:intro3c0} by 
\begin{multline}\label{eq:intro3c}
\xi'\cdot {A}_{T,k,\rho,\e}(x)\xi\,\\
:=\,
\expec{(\xi'+\nabla \phi_{T,k,\rho,\e}'-\expec{\nabla \phi_{T,k,\rho,\e}'}_{x,\mu_{\rho,\e}})
\cdot A(y)(\xi+\nabla \phi_{T,k,\rho,\e}-\expec{\nabla \phi_{T,k,\rho,\e}}_{x,\mu_{\rho,\e}})}_{x,\mu_{\rho,\e}},
\end{multline}
where $T>0$ is a parameter, $k\in \N$, $\phi_{T,k,\rho,\e},\phi_{T,k,\rho,\e}'$ are Richardson extrapolations (with respect to $T$) of the unique weak solutions $\phi_{T,1,\rho,\e},\phi_{T,1,\rho,\e}'\in H^1_0\Big(\frac{D-x}{\e}\cap B\left(0,\frac{\rho}{\e}\right)\Big)$ of the regularized cell problems
\begin{equation}\label{eq:intro3.1c}
\begin{array}{rcl}
T^{-1} \phi_{T,1,\rho,\e}-\nabla \cdot A(\xi+\nabla \phi_{T,1,\rho,\e})&=&0\quad \text{in }\frac{D-x}{\e}\cap B\left(0,\frac{\rho}{\e}\right), \\
T^{-1} \phi_{T,1,\rho,\e}'-\nabla \cdot A^*(\xi'+\nabla \phi_{T,1,\rho,\e}')&=&0\quad \text{in }\frac{D-x}{\e}\cap B\left(0,\frac{\rho}{\e}\right),
\end{array}
\end{equation}
Then, if $A$ is uniformly elliptic, then ${A}_{T,k,\rho,\e}$ is uniformly elliptic as well, and we may approximate $u_\ho$ by the weak solution $u_{T,k,\rho,\e}$ in $H^1_0(D)$ of 
\begin{equation}\label{eq:intro4b}
-\nabla \cdot A_{T,k,\rho,\e} \nabla u_{T,k,\rho,\e}\,=\,f.
\end{equation}
Likewise, one may approximate $\nabla u_\e$ in $L^2({D}\cap B(0,\rho),\mu_{\rho,\e}dx)$ by
$y\mapsto C_{T,k,\e,\rho}(x,y):=\sum_{i=1}^d \expec{\nabla_i u_{T,k,\rho,\e}}_{x,\mu_{\rho,\e}} \nabla_y \phi_{T,k,\rho,\e,i}(x,\e y)$,
where $\phi_{T,k,\rho,\e,i}$ denotes the solution of \eqref{eq:intro3.1c} for $\xi=e_i$.

\medskip
\noindent From the analysis point of view, the first natural question concerns consistency: under which assumptions on $A_\e$ does the convergences
\begin{equation*}
\lim_{T\uparrow \infty,\rho\downarrow 0}\lim_{\e \downarrow 0}\|u_{T,k,\rho,\e}-u_{\ho}\|_{H^1(D)} \,=\,0,
\lim_{T\uparrow \infty,\rho\downarrow 0}\lim_{\e \downarrow 0} \expec{\nabla u_\e-C_{T,k,\e,\rho}(x,\cdot)}_{x,\mu_{\rho,\e}}^{\frac{1}{2}}\,=\,0
\end{equation*}
hold? Can we obtain convergence rates in function of $T$, $k$, $\rho$, $\e$?
The last question is related to the discretization of this approach and its relation with standard numerical homogenization methods.

\medskip
\noindent The article is organized as follows.
We recall in Section~\ref{sec:spectral} 
standard results on stochastic homogenization, and introduce 
a method based on regularization and extrapolation to approximate homogenized coefficients and correctors.
We provide two proofs of the consistency of the approximations in the stationary ergodic setting: one using spectral theory (limited to symmetric coefficients) and another one relying solely on PDE arguments (to treat nonsymmetric coefficients as well). We obtain in addition optimal quantitative estimates for the class of periodic coefficients, for the subclass of almost periodic coefficients satisfying a Poincar\'e inequality, and in the subclass of random stationary coefficients satisfying a spectral gap estimate (including random Poisson inclusions).
In Section~\ref{sec:tests} we devise computable proxies for these approximations, prove sharp 
estimates on the approximation error,  and display the results of numerical tests in dimension 2 on both symmetric and nonsymmetric 
periodic and almost periodic coefficients. The numerical tests confirm the interest of the method and the sharpness of the analysis. An interesting output of the analysis (confirmed by the numerical tests) is that \emph{correctors are easier to approximate in practice than homogenized coefficients} (essentially because there is no average to compute).
In Section~\ref{sec:reduc}, we finally show how the regularization and extrapolation method can be combined with standard numerical homogenization 
methods, and prove the asymptotic consistency for locally stationary ergodic coefficients.
If the coefficients are in addition periodic, we give a quantitative error analysis based on Section~\ref{sec:spectral} and Section~\ref{sec:tests}, which shows that the resonance error 
is drastically reduced. 

\medskip

\noindent We make use of the following notation:
\begin{itemize}
\item $d\geq 2$ is the dimension;
\item For all $D$ open subset of $\R^d$, $|D|$ denotes its $d$-dimensional Lebesgue measure, $\int_D$ the integral on $D$, and $\fint_D$ is a notation for $\frac{1}{|D|}\int_D$;
\item $Q:=(-1,1)^d$, $Q_R:=(-R,R)^d$, $B_R(x):=\{y\in \R^d:|x-y|<R\}$, $B_R:=B_R(0)$ for all $x\in \R^d$ and $R\in\R^+$;
\item $A^*$ denotes its transpose of a matrix $A$;
\item $H^1_\per(Q)$ denotes the closure in $H^1(Q)$ of smooth $Q$-periodic functions with zero mean;
\item $\lesssim$ and $\gtrsim$ stand for $\leq$ and $\geq$ up to a multiplicative constant which only depends on the dimension $d$ and the constants $\alpha,\beta$ (see Definition~\ref{def:alpha-beta} below) if not otherwise stated;
\item When both $\lesssim$ and $\gtrsim$ hold, we simply write $\sim$;
\item we use $\gg $ instead of $\gtrsim$ when the multiplicative constant is (much) larger than $1$.
\end{itemize}

%%%%%%%%%%%%%%%%%%%%%%%%%%%%%%%
%%%%%%%%%%%%%%%%%%%%%%%%%%%%%%%

\section{Correctors, homogenized coefficients, and Richardson extrapolation}\label{sec:spectral}

\subsection{Stochastic homogenization}

We first recall some standard qualitative results, which will be used in the core of the article and whose proofs can be found in \cite{Papanicolaou-Varadhan-79}.

\medskip

\noindent Let $(\Omega,\mathcal F,\mathbb{P})$ be a probability space (we denote by $\expec{\cdot}$ the associated expectation).
We shall say that the family of mappings 
$(\theta_z)_{z\in \R^d}$ from $\Omega$ to $\Omega$ is a strongly continuous measure-preserving ergodic translation group if:
\begin{itemize}
\item $(\theta_z)_{z\in \R^d}$ has the group property: $\theta_0=\mathrm{Id}$ (the identity mapping), and for all $x,y\in \R^d$, $\theta_{x+y}=\theta_x\circ\theta_y$;
\item $(\theta_z)_{z\in \R^d}$ preserves the measure: for all $x\in \R^d$, and every measurable set $F\in \mathcal F$, $\theta_xF$ is measurable
and $\mathbb{P}(\theta_xF)=\mathbb{P}(F)$;
\item $(\theta_z)_{z\in \R^d}$ is strongly continuous: for any measurable function $f$ on $\Omega$, the function $(\omega,x)\mapsto f(\theta_x\omega)$
defined on $\Omega\times \R^d$ is measurable (with the Lebesgue measure on $\R^d$);
\item $(\theta_z)_{z\in \R^d}$ is ergodic: for all $F\in \mathcal F$, if for all $x\in \R^d$, $\theta_xF \subset F$, then $\mathbb{P}(F)\in \{0,1\}$.
\end{itemize}

\medskip
\begin{defi}\label{def:alpha-beta} 
For all $d\geq 2$, $\beta\geq \alpha>0$, we define $\Mab$ as the set of $d\times d$ real matrices
such that for all $\xi\in \R^d$,
\begin{equation*}
\xi\cdot A\xi \,\geq \, \alpha |\xi|^2, \qquad
|A\xi|\,\leq \, \beta |\xi|,
\end{equation*}
and $\Mabs$the subset of symmetric matrices of $\Mab$.
\qed
\end{defi}
\noindent Let $0<\alpha\leq \beta <\infty$, and let $A \in L^\infty(\Omega,\Mab)$.
We define the stationary extension of $A$ (still denoted by $A$) on $\R^d\times \Omega$ as
follows: 
\begin{equation*}
A:(x,\omega) \mapsto A(x,\omega)\,:=\,A(\theta_x\omega).
\end{equation*}
Homogenization theory ensures that the solution operator associated
with $-\nabla \cdot A(\frac{\cdot}{\e},\omega)\nabla$ converges as $\e>0$ vanishes
to the solution operator of $-\nabla \cdot A_\ho \nabla$ for
$\P$-almost every $\omega$, where $A_\ho$ is a deterministic elliptic
matrix characterized as follows.
For all $\xi,\xi'\in \R^d$ with $|\xi'|=|\xi|=1$, and $\P$-almost every $\omega$,
\begin{eqnarray}\label{eq:hom-coeff-gen}
\xi'\cdot A_\ho \xi&=&\lim_{R\uparrow \infty} \fint_{Q_R} (\xi'+\nabla
\phi'(x,\omega))\cdot A(x,\omega)(\xi+\nabla
\phi(x,\omega))dx \nonumber \\
&=&\expec{(\xi'+\nabla
\phi')\cdot A(\xi+\nabla
\phi)(0)},
\end{eqnarray}
where $\phi:\R^d\times \Omega\to \R$ is the primal corrector in direction $\xi\in \R^d$ defined as the unique Borel measurable map such
that $\phi(0,\cdot)=0$ almost surely, $\nabla \phi$ is stationary, $\expec{\nabla \phi}=0$, and such that
$\phi(\cdot,\omega)\in H^1_\loc(\R^d)$ is almost surely a
distributional solution of the corrector equation
\begin{equation}\label{eq:corr-sto} 
-\nabla \cdot A(x,\omega)(\xi+\nabla \phi(x,\omega))\,=\,0 \mbox{
  in } \R^d.
\end{equation}
Likewise  $\phi'$ is the dual corrector in direction $\xi'\in \R^d$
defined as the unique Borel measurable map such
that $\phi'(0,\cdot)=0$ almost surely, $\nabla \phi'$ is stationary, $\expec{\nabla \phi'}=0$, and such that
$\phi'(\cdot,\omega)\in H^1_\loc(\R^d)$ is almost surely a
distributional solution of the dual corrector equation
\begin{equation}\label{eq:corr-sto-ad} 
-\nabla \cdot A^*(x,\omega)(\xi'+\nabla \phi'(x,\omega))\,=\,0 \mbox{
  in } \R^d.
\end{equation}

\noindent The proof of existence and uniqueness of these correctors is obtained
by regularization, and we consider for all $T>0$ the stationary
solutions $\phi_T,\phi_T'$ with zero expectation of the equations
\begin{equation}\label{eq:corr-reg}
\begin{array}{rcl}
T^{-1}\phi_T(x)-\nabla \cdot A(x)(\xi+\nabla \phi_T(x))&=&0 \mbox{
  in } \R^d,\\
T^{-1}\phi_T'(x)-\nabla \cdot A^*(x)(\xi'+\nabla \phi_T'(x))&=&0 \mbox{
  in } \R^d.
\end{array}
\end{equation}
Indeed, as proved in \cite[Lemma~2.7]{Gloria-Otto-10b},  for all $A\in L^\infty(\R^d,\Mab)$, the regularized corrector equations admit unique solutions $\phi_T,\phi_T'$ in the class of  functions $v\in H^1_\loc(\R^d)$
such that 
\begin{equation}\label{eq:class-1}
\sup_{x\in \R^d}\int_{B_{1}(x)}(v^2+|\nabla v|^2)dx'<\infty.
\end{equation}
Furthermore, the solutions $\phi_{T},\phi_T'$ satisfy the uniform a priori estimate
\begin{equation}\label{eq:class-2}
\begin{array}{l}
\sup_{x\in \R^d}\int_{B_{\sqrt{T}}(x)}(T^{-1}\phi_T^2+|\nabla \phi_T|^2)dx'\,\lesssim\,\sqrt{T}^d,\\
\sup_{x\in \R^d}\int_{B_{\sqrt{T}}(x)}(T^{-1}\phi_T'^2+|\nabla \phi_T'|^2)dx'\,\lesssim\,\sqrt{T}^d.
\end{array}
\end{equation}
Note that the existence of regularized correctors does not require any assumption on the coefficients $A$ besides uniform boundedness from below and above.
In the case when $A$ is a stationary field, the associated random fields $\phi_T,\phi_T'$ of solutions are stationary, and in addition the defining equations have an equivalent form in the probability space, to
which we can apply the Lax-Milgram theorem.
This formulation requires the definition of a differential calculus in the probability space: the differential operator $\nabla_i$ (for $i\in \{1,\dots,d\}$) has a counterpart in probability, denoted by $\DD_i$
and defined by
\begin{equation*}
\DD_i f(\omega)\,=\,\lim_{h\downarrow 0} \frac{f(\theta_{h\ee_i}\omega) -f(\omega)}{h}.
\end{equation*}
These are the infinitesimal generators of the $d$ one-parameter strongly continuous
unitary groups on $L^2(\Omega)$ defined by the translations in each of
the $d$ directions. These operators commute and are closed and densely defined on $L^2(\Omega)$.
We denote by $\calH(\Omega)$ the domain of $\DD=(\DD_1,\dots,\DD_d)$. 
This subset of $L^2(\Omega)$ is a Hilbert space
for the norm
\begin{equation*}
\|f\|_{\calH}^2\,=\,\expec{|\DD f|^2}+\expec{f^2}.
\end{equation*}
Since the groups are unitary, the operators are skew-dual so that
we have an integration by parts formula: for all $f,g\in \calH(\Omega)$
\begin{equation*}
\expec{f\DD_i g}\,=\,-\expec{g\DD_i f}.
\end{equation*}
The equivalent form of the regularized corrector equations for stationary coefficients is as
follows:
\begin{equation}\label{eq:mod-corr-sto} 
\begin{array}{rcl}
T^{-1}\phi_T(0,\cdot)-\DD \cdot A(0) (\xi+\DD \phi_T(0,\cdot))&=&0 ,\\
T^{-1}\phi_T'(0,\cdot)-\DD \cdot A^*(0) (\xi'+\DD \phi_T'(0,\cdot))&=&0 ,
\end{array}
\end{equation}
which admit unique weak solutions $\phi_T(0,\cdot),\phi_T'(0,\cdot)\in \calH(\Omega)$,
that are such that for all $\psi \in \calH(\Omega)$,
\begin{equation}\label{eq:mod-corr-sto-exp} 
\begin{array}{rcl}
\expec{T^{-1}\phi_T(0,\cdot)\psi+\DD \psi \cdot A(0) (\xi+\DD\phi_T(0,\cdot))}&=&0,\\
\expec{T^{-1}\phi_T'(0,\cdot)\psi+\DD \psi \cdot A^*(0) (\xi'+\DD\phi_T'(0,\cdot))}&=&0.
\end{array}
\end{equation}
One may prove using the integration by parts formula that $\DD
\phi_T(0,\cdot),\DD\phi_T'(0,\cdot)$ are bounded in $L^2(\Omega)$ and converge weakly in
$L^2(\Omega)$ to some potential fields $\Phi,\Phi'\in L^2(\Omega,\R^d)$.
Using the spectral representation of the translation group we may then
prove uniqueness of the corrector $\phi,\phi'$ (which are such that
$\nabla \phi(x,\omega)=\Phi(\theta_x \omega), \nabla \phi'(x,\omega)=\Phi'(\theta_x \omega)$ under the ergodicity assumption), see \cite{Papanicolaou-Varadhan-79}.
Note that whereas we have $\nabla \phi_T(x,\omega)\,=\,\DD \phi_T(0,\theta_x \omega),\nabla \phi_T'(x,\omega)\,=\,\DD \phi_T'(0,\theta_x \omega)$ for all $T>0$, $\phi,\phi'$ are not stationary fields and $\DD \phi(0,\cdot),\DD \phi'(0,\cdot)$ are not a well-defined quantities a priori (except typically in the periodic setting), only $\nabla \phi,\nabla \phi'$ are.

\medskip

\noindent We recall that the case of periodic and almost periodic coefficients can be recast in this stochastic framework up to randomizing the origin of the periodic cell. In particular, for $Q$-periodic coefficients, we then have
$\calH(\Omega)=H^1_\per(Q)$, and the expectation is replaced by the spatial average on $Q$.
We conclude by an example of random coefficients we shall more specifically consider in this article: Random Poisson inclusions.
By the ``Poisson ensemble'' we understand the following probability measure:
Let the configuration of points $\calP:=\{x_n\}_{n\in \N}$ on $\R^d$
be distributed according to the Poisson
point process with density one. This means the following
\begin{itemize}
\item 
For any two disjoint (Lebesgue measurable) subsets $D$ and $D'$ of $\R^d$ we have that
the configuration of points in $D$ and the configuration of points in $D'$ are independent.
\item For any (Lebesgue measurable) bounded subset $D$ of $\R^d$, 
the number of points in $D$ is Poisson distributed;
the expected number is given by the Lebesgue measure of $D$.
\end{itemize}
With any realization $\calP=\{x_n\}_{n\in \N}$ of the Poisson point process, 
we associate the coefficient field $A$ via
\begin{equation*}\label{eq:A-poisson}
A(x)=\left\{\begin{array}{ccc}
\alpha&\mbox{if}&x\in\bigcup_{n=1}^{\infty}B_{\frac{1}{2}}(x_n)\\
\beta&\mbox{else}\end{array}\right\}\Id.
\end{equation*}
This defines a probability measure $\expec{\cdot}$ on $\Omega$ by ``push-forward'' of Poisson measure.

\subsection{Approximation of homogenized coefficients and correctors}

\label{sec:App.ahom}

The starting point in \cite{Gloria-09} is the observation that the solutions 
$\phi_T,\phi_T' \in H^1_\loc(\R^d)$ of \eqref{eq:corr-reg} can be approximated 
 on a bounded domain $Q_L$ by the weak solutions $\phi_{T,R},\phi_{T,R}'\in H^1_0(Q_R)$
of 
\begin{equation}\label{eq:extra-approx}
\begin{array}{rcl}
T^{-1}\phi_{T,R}-\nabla \cdot A (\xi+\nabla \phi_{T,R})&=&0  \mbox{ in }Q_R,\\
T^{-1}\phi_{T,R}'-\nabla \cdot A^* (\xi'+\nabla \phi_{T,R}')&=&0  \mbox{ in }Q_R,
\end{array}
\end{equation}
up to an error which is of infinite order measured in units of $\frac{R-L}{\sqrt{T}}$.
This is in contrast to the correctors $\phi,\phi'$ for which the approximation using homogeneous
Dirichlet boundary conditions yields a surface term of order $\frac{1}{R}$.
Since we have easy access to a proxy for $\phi_T,\phi_T'$, a natural question is how to approximate
the gradients $\nabla \phi,\nabla \phi'$ of the correctors and the homogenized coefficients $A_\ho$ using $\phi_T,\phi_T'$ instead of $\phi,\phi'$.
In  \cite{Gloria-09}, the following approximations are considered in the case of symmetric coefficients $A$:
$$
\nabla \phi \leadsto \nabla \phi_T, \qquad \xi\cdot A_\ho\xi \leadsto \xi\cdot A_T\xi:=\expec{(\xi+\nabla \phi_T)\cdot A(\xi+\nabla \phi_T)}.
$$
In the case of symmetric periodic coefficients, it is shown that  $\expec{|\nabla \phi-\nabla \phi_T|^2}\sim T^{-2}$ and $|A_T-A_\ho|\sim T^{-2}$.
An explicit family of higher-order approximations $\{\tilde A_{T,k}\}_{k\in \N}$ has then been introduced in \cite{Gloria-Mourrat-10} based on spectral theory which satisfies
for symmetric periodic coefficients $|\tilde A_{T,k}-A_\ho|\sim T^{-2k}$.
More precisely, the first three approximations of $A_\ho$ are given by:
\begin{eqnarray*}
\xi\cdot \tilde A_{T,1}\xi&=& \expec{(\xi+\nabla \phi_T)\cdot A (\xi+\nabla \phi_T)},\\
\xi\cdot \tilde A_{T,2}\xi&=& \expec{(\xi+\nabla \phi_T)\cdot A (\xi+\nabla \phi_T)}-3T^{-1} \expec{\phi_T^2}-2T^{-1} \expec{\phi_{\frac{T}{2}}^2}\\
&&+5T^{-1} \expec{\phi_T\phi_{\frac{T}{2}}},\\
\xi\cdot \tilde A_{T,3}\xi&=& \expec{(\xi+\nabla \phi_T)\cdot A (\xi+\nabla \phi_T)}-\frac{55}{9}T^{-1}\expec{\phi_T^2}-8T^{-1}  \expec{\phi_{\frac{T}{2}}^2}\\
&&-\frac{4}{9}T^{-1} \expec{\phi_{\frac{T}{4}}^2} +\frac{41}{3}T^{-1} \expec{\phi_T\phi_{\frac{T}{2}}}-\frac{22}{9}T^{-1}\expec{\phi_T\phi_{\frac{T}{4}}}
\\
&&+\frac{10}{3}T^{-1}
 \expec{\phi_{\frac{T}{2}}\phi_{\frac{T}{4}}}.
\end{eqnarray*}
These results have two drawbacks:
\begin{itemize}
\item Approximation of homogenized coefficients: It is not clear how to extend the definition of $\tilde A_{T,k}$ to nonsymmetric coefficients;
\item Approximation of correctors: The approximation of $\nabla \phi$ by $\nabla \phi_T$ yields an $L^2(\Omega)$ error which saturates at $T^{-1}$.
\end{itemize}
The aim of this section is to propose approximations of $\nabla \phi$ and $A_\ho$ that reduce these drawbacks.
As noticed in \cite{Gloria-Neukamm-Otto-14}, the family $\{\tilde A_{T,k}\}_{k\in \N}$ can be seen as extrapolations of $A_T$ wrt $T$ in spectral space (noting that $T\mapsto A_T$ is a smooth function of $T$ for $T>0$).
The main idea of this section consists in approximating the corrector $\phi,\phi'$ by Richardson extrapolations of $\phi_T,\phi_T'$.
\begin{defi}[Richardson extrapolation of regularized correctors]\label{def:extra}
Let $A\in L^\infty(\R^d,\Mab)$ be some coefficients.
For all $T>0$, $\xi,\xi' \in \R^d$, we consider the regularized primal and dual correctors $\phi_T,\phi_T'\in H^1_\loc(\R^d)$ in directions $\xi,\xi'$, unique distributional solutions of \eqref{eq:corr-reg}
satisfying \eqref{eq:class-1}.
The families $\{\phi_{T,k}\}_{k\in \N},\{\phi_{T,k}'\}_{k\in \N}$ of Richardson extrapolations are defined by the following induction: $\phi_{T,1}\,:=\,\phi_{T},\phi_{T,1}'\,:=\,\phi_{T}'$, and
for all $k\in \N$,
\begin{equation}\label{eq:def:extrapo}
\phi_{T,k+1}\,:=\,\dfrac{1}{2^k-1}(2^k\phi_{2T,k}-\phi_{T,k}),
\qquad 
\phi'_{T,k+1}\,:=\,\dfrac{1}{2^k-1}(2^k\phi'_{2T,k}-\phi'_{T,k}).
\end{equation}
\qed
\end{defi}
\noindent We may define approximations of the homogenized coefficients associated with the Richardson extrapolations of the regularized correctors:
\begin{defi}[Approximation of homogenized coefficients]\label{def:extra-hom}
Let $A\in L^\infty(\R^d,\Mab)$ be some coefficients.
For all $T>0$, $\xi,\xi' \in \R^d$, we let $\{\phi_{T,k}\}_{k\in \N},\{\phi_{T,k}'\}_{k\in \N}$ be as in Definition~\ref{def:extra}, and we define the family $\{A_{T,k}\}_{k\in \N}$ of approximations of the homogenized coefficients by
\begin{equation*}
\xi'\cdot A_{T,k}\xi \,:=\,\calM\big( (\xi'+\nabla \phi_{T,k}')\cdot A(\xi+\nabla {\phi}_{T,k})\big).
\end{equation*}
\qed
\end{defi}
\noindent In the case of stationary ergodic coefficients $A$, Definition~\ref{def:extra-hom} makes sense and we have 
$$
\xi'\cdot A_{T,k}\xi \,=\,\expec{ (\xi'+\nabla \phi'_{T,k})\cdot A(\xi+\nabla {\phi}_{T,k})}.$$

\medskip

\noindent In the following paragraph we show, using elementary PDE arguments, that the approximations $\nabla \phi_{T,k},\nabla \phi'_{T,k}$ and $A_{T,k}$
are consistent with $\nabla \phi,\nabla \phi'$ and $A_\ho$ in the case of stationary ergodic coefficients, in the limit $T\uparrow +\infty$.

\subsection{Asymptotic consistency}
\begin{theo}\label{th:spec}
Let $A\in L^\infty(\Omega,\Mab)$ be stationary ergodic 
coefficients, $\phi,\phi'$ be the primal and dual correctors in directions $\xi,\xi'\in \R^d$, $|\xi'|=|\xi|=1$,
and $A_\ho$ be the homogenized coefficients.
For all $k\in \N$ and $T>0$, let $\phi_{T,k},\phi_{T,k}'$ and $A_{T,k}$ be as in Definitions~\ref{def:extra} and~\ref{def:extra-hom}.
Then,
\begin{equation*}
\lim_{T\uparrow \infty}\expec{|\nabla \phi_{T,k}-\nabla \phi|^2} \,=\,0,\qquad\lim_{T\uparrow \infty}\expec{|\nabla \phi'_{T,k}-\nabla \phi'|^2} \,=\,0,\qquad 
\lim_{T\uparrow \infty}|A_{T,k}-A_\ho| \,=\,0.
\end{equation*}
\qed
\end{theo}
\begin{proof}
As we shall see in Step~1, it is enough to prove that 
\begin{equation}\label{eq:asymp-conv-phiT}
\lim_{T\uparrow \infty} \expec{(\xi+\nabla\phi_T)\cdot A(\xi+\nabla\phi_T)}\,=\,
\expec{(\xi+\nabla \phi)\cdot A(\xi+\nabla \phi)}.
\end{equation}

\medskip

\step{1} Proof that \eqref{eq:asymp-conv-phiT} implies the claim.

\noindent We start with the approximation of the correctors. 
Because $\phi_{T,k},\phi'_{T,k}$ are defined by extrapolations, it is enough to prove
the claim for $\phi_T,\phi_T'$. We only address $\phi_T$. 
From the uniform ellipticity of $A$, the map $L^2(\Omega,\R^d)\ni \Psi \mapsto \expec{\Psi\cdot A(0)\Psi}^{\frac{1}{2}}$ is equivalent to the norm $\|\cdot\|_{L^2(\Omega,\R^d)}$.
Hence, since $\nabla \phi_T(0)$ converges weakly to $\nabla \phi(0)$ in $L^2(\Omega,\R^d)$
as $T \uparrow \infty$, it converges strongly to $\nabla \phi(0)$ 
in $L^2(\Omega,\R^d)$ if and only if 
$$
\lim_{T\uparrow \infty} \expec{\nabla \phi_T\cdot A\nabla \phi_T}^{\frac{1}{2}}\,=\,\expec{\nabla \phi\cdot A\nabla \phi}^{\frac{1}{2}}.
$$
The latter is a consequence of \eqref{eq:asymp-conv-phiT} and of the weak convergence 
of $\nabla \phi_T$ to $\nabla \phi$ as follows:
\begin{eqnarray*}
\expec{\nabla \phi_T\cdot A\nabla \phi_T}&=&\expec{(\xi+\nabla \phi_T)\cdot A(\xi+\nabla \phi_T)}-\expec{\nabla \phi_T\cdot A\xi}-\expec{\xi\cdot A\nabla \phi_T}-\expec{\xi\cdot A\xi}
\\
&\stackrel{T\uparrow \infty}{\longrightarrow}& \expec{(\xi+\nabla \phi)\cdot A(\xi+\nabla \phi)}-\expec{\nabla \phi\cdot A\xi}-\expec{\xi\cdot A\nabla \phi}-\expec{\xi\cdot A\xi}\\
&=&\expec{\nabla \phi\cdot A\nabla \phi}.
\end{eqnarray*}
The strong convergence in $L^2(\Omega,\R^d)$ of $\nabla \phi_{T,k}(0),\nabla\phi'_{T,k}(0)$
to $\nabla \phi(0),\nabla \phi'(0)$ trivially implies the convergence of $A_{T,k}$ to $A_\ho$.

\medskip

\step{2} Proof of \eqref{eq:asymp-conv-phiT}.

\noindent Since the coefficients are stationary, we can rewrite the regularized corrector equation for $\phi_T$ in the probability space, the weak formulation of which yields with test-function $\phi_T$:
\begin{equation}\label{eq:pr-asymp-conv-phiT-3}
T^{-1}\expec{\phi_T^2}+\expec{(\xi+\DD\phi_T)\cdot A(0)(\xi+\DD \phi_T)}\,=\,\expec{\xi\cdot A(0)(\xi+\DD \phi_T)}.
\end{equation}
On the one hand, since $\DD \phi_T \rightharpoonup \nabla \phi(0)$ weakly in $L^2(\Omega)^d$ (the convergence of the whole sequence and the uniqueness of the limit are consequences of ergodicity),
\begin{equation}\label{eq:pr-asymp-conv-phiT-1}
\lim_{T\uparrow \infty} \expec{\xi\cdot A(0)(\xi+\DD \phi_T)}\,=\,\xi\cdot A_\ho\xi.
\end{equation}
On the other hand, by the weak lower semicontinuity of quadratic functionals,
\begin{equation}\label{eq:pr-asymp-conv-phiT-2}
\liminf_{T\uparrow \infty} \expec{(\xi+\DD\phi_T)\cdot A(0)(\xi+\DD \phi_T)}\,\geq \, 
\expec{(\xi+\nabla\phi(0))\cdot A(0)(\xi+\nabla \phi(0))}=\xi\cdot A_\ho \xi.
\end{equation}
The combination of \eqref{eq:pr-asymp-conv-phiT-1},  \eqref{eq:pr-asymp-conv-phiT-2}, and \eqref{eq:pr-asymp-conv-phiT-3} then yields by non-negativity of $T^{-1}\expec{\phi_T^2}$
$$
\lim_{T\uparrow \infty} \expec{(\xi+\DD\phi_T)\cdot A(\xi+\DD \phi_T)}\,=\,\xi\cdot A_\ho \xi,
$$
as desired.
\end{proof}
\begin{rem}\label{rem:prop-spec3.1}
This elementary result seems to be new in the case of non-symmetric coefficients. For symmetric coefficients
it was proved by Papanicolaou and Varadhan in their seminal paper \cite{Papanicolaou-Varadhan-79} using spectral theory, see Subsection~\ref{sec:quant-spec}.
Note in particular that the proof of Theorem~\ref{th:spec} yields the following result:
\begin{equation*}
\lim_{T\uparrow \infty} T^{-1}\expec{\phi_T^2}\,=\,\lim_{T\uparrow \infty} T^{-1}\expec{\phi_T'^2}\,=\,0.
\end{equation*}
This convergence is much stronger than the a priori estimates 
$$
\expec{\phi_T^2}\,\lesssim \,T, \quad \expec{\phi_T'^2}\,\lesssim \,T,
$$
associated with \eqref{eq:corr-reg} for \emph{generic} coefficients $A$. They are consequences of \emph{ergodicity}. \qed
\end{rem}
\noindent Before we turn to quantitative results we discuss the uniform ellipticity of $A_{T,k}$, which is a rather subtle issue.
\begin{prop}\label{prop:coerc}
Let $A\in L^\infty(\Omega,\Mab)$ be stationary ergodic 
coefficients, let $\xi'=\xi\in \R^d$, $|\xi'|=|\xi|=1$,
and for all $k\in \N$ and $T>0$, let $\phi_{T,k},\phi_{T,k}'$ and $A_{T,k}$ be as in Definitions~\ref{def:extra} and~\ref{def:extra-hom}.
Then, 
\begin{multline}\label{eq:coer-AT1}
\xi \cdot A_{T}\xi\,=\, \frac{1}{2}\expec{(\xi+\nabla \phi_T)\cdot A(\xi+\nabla \phi_T)}
+\frac{1}{2}\expec{(\xi+\nabla \phi_T')\cdot A(\xi+\nabla \phi_T')}\\
+\frac{T^{-1}}{2}\expec{(\phi_T-\phi_T')^2},
\end{multline}
and for all $k\geq 1$,
\begin{multline}\label{eq:coer-ATk}
\xi \cdot A_{T,k+1}\xi\,=\, \frac{1}{2}\expec{(\xi+\nabla \phi_{T,k+1})\cdot A(\xi+\nabla \phi_{T,k+1})}
+\frac{1}{2}\expec{(\xi+\nabla \phi_{T,k+1}')\cdot A(\xi+\nabla \phi_{T,k+1}')}\\
+(2^{k+1}T)^{-1}\expec{(\phi_{T,k+1}-\phi'_{T,k+1})(\phi_{T,k+1}-\phi'_{T,k+1}-\phi_{T,k}+\phi'_{T,k})}.
\end{multline}
In particular, when $A$ is symmetric, $A_{T,k}$ is uniformly coercive for all $k\in \N$ and $T>0$, whereas
when $A$ is non-symmetric, $A_{T}$ is coercive for all $T>0$ but $A_{T,k}$ may have an ellipticity defect for $k>1$.
\qed
\end{prop}
\begin{rem}
We can bound the possibly negative correction in \eqref{eq:coer-ATk} as follows
\begin{multline*}
(2^{k+1}T)^{-1}\expec{(\phi_{T,k+1}-\phi'_{T,k+1})(\phi_{T,k+1}-\phi'_{T,k+1}-\phi_{T,k}+\phi'_{T,k})} \\ \lesssim \,
(2^{k+1} T)^{-1} \Big(\sup_{2^k T\leq \tau \leq 2^{k+1}T} \expec{\phi_\tau^2}+\sup_{2^k T\leq \tau \leq 2^{k+1}T} \expec{{\phi'_\tau}^2}\Big),
\end{multline*}
which vanishes both in the limits $T\uparrow \infty$ and $k\uparrow \infty$.
As soon as we have a good control of the $L^2$-norms of the regularized correctors wrt $T$, this crude estimate yields a useful control of 
the ellipticity defect (say, when $\expec{\phi_T^2}$ is bounded).
When correctors are not well-behaved, extrapolation is not expected to improve the convergence rates so that the approximation $A_T$ is sufficient (and unconditionally uniformly elliptic).
\qed
\end{rem}
\begin{proof}[Proof of Proposition~\ref{prop:coerc}]
We start with the proof of \eqref{eq:coer-AT1} and then turn to \eqref{eq:coer-ATk}.

\medskip
\step{1} Proof of \eqref{eq:coer-AT1}.

\noindent By the regularized corrector equation for $\phi_T$,
\begin{multline*}
\expec{(\xi+\nabla \phi_T')\cdot A(\xi+\nabla \phi_T)}\,=\,
\expec{(\xi+\nabla \phi_T)\cdot A(\xi+\nabla \phi_T)}+\expec{(\nabla \phi_T'-\nabla \phi_T)\cdot A(\xi+\nabla \phi_T)}
\\=\,
\expec{(\xi+\nabla \phi_T)\cdot A(\xi+\nabla \phi_T)}-T^{-1}\expec{(\phi_T'-\phi_T)\phi_T}.
\end{multline*}
Likewise, by the regularized corrector equation for $\phi_T'$,
\begin{multline*}
\expec{(\xi+\nabla \phi_T')\cdot A(\xi+\nabla \phi_T)}\,=\,
\expec{(\xi+\nabla \phi_T')\cdot A(\xi+\nabla \phi_T)}+\expec{(\nabla \phi_T-\nabla \phi_T')\cdot A^*(\xi+\nabla \phi_T')}
\\=\,
\expec{(\xi+\nabla \phi_T')\cdot A(\xi+\nabla \phi_T')}-T^{-1}\expec{(\phi_T-\phi_T')\phi_T'}.
\end{multline*}
The half sum of these identities yields \eqref{eq:coer-AT1}.

\medskip
\step{2} Proof of \eqref{eq:coer-ATk}.

\noindent We proceed in a similar way for \eqref{eq:coer-ATk}. 
As we shall show by induction in the proof of Theorem~\ref{th:quant-estim-Perio} below (cf. \eqref{eq:induc-phik}),
\begin{equation*}
T^{-1} \phi_{T,k+1}-\nabla \cdot A(\xi+\nabla \phi_{T,k+1})\,=\,T^{-1}  \phi_{2T,k},
\end{equation*}
so that 
\begin{multline*}
\expec{(\xi+\nabla \phi_{T,k+1}')\cdot A(\xi+\nabla \phi_{T,k+1})}\\
=\,
\expec{(\xi+\nabla \phi_{T,k+1})\cdot A(\xi+\nabla \phi_{T,k+1})}+\expec{(\nabla \phi_{T,k+1}'-\nabla \phi_{T,k+1})\cdot A(\xi+\nabla \phi_{T,k+1})}
\\=\,
\expec{(\xi+\nabla \phi_{T,k+1})\cdot A(\xi+\nabla \phi_{T,k+1})}-T^{-1}\expec{(\phi_{T,k+1}'-\phi_{T,k+1})(\phi_{T,k+1}-\phi_{2T,k})}.
\end{multline*}
Likewise,
\begin{multline*}
\expec{(\xi+\nabla \phi_{T,k+1}')\cdot A(\xi+\nabla \phi_{T,k+1})}\\
=\,
\expec{(\xi+\nabla \phi_{T,k+1}')\cdot A(\xi+\nabla \phi_{T,k+1}')}+\expec{(\nabla \phi_{T,k+1}-\nabla \phi'_{T,k+1})\cdot A^*(\xi+\nabla \phi'_{T,k+1})}
\\=\,
\expec{(\xi+\nabla \phi'_{T,k+1})\cdot A(\xi+\nabla \phi'_{T,k+1})}-T^{-1}\expec{(\phi_{T,k+1}-\phi'_{T,k+1})(\phi'_{T,k+1}-\phi'_{2T,k})},
\end{multline*}
so that the half sum of these identities yields
\begin{multline*}
\expec{(\xi+\nabla \phi_{T,k+1}')\cdot A(\xi+\nabla \phi_{T,k+1})}\\
=\,
\frac{1}{2}\expec{(\xi+\nabla \phi_{T,k+1})\cdot A(\xi+\nabla \phi_{T,k+1})}+\frac{1}{2} \expec{(\xi+\nabla \phi'_{T,k+1})\cdot A(\xi+\nabla \phi'_{T,k+1})} \\
+\frac{1}{2}T^{-1}\expec{(\phi_{T,k+1}-\phi'_{T,k+1})(\phi_{T,k+1}-\phi'_{T,k+1}-\phi_{2T,k}+\phi_{2T,k}')},
\end{multline*}
which we rewrite as \eqref{eq:coer-ATk} using the definition of the Richardson extrapolation.
\end{proof}

\subsection{Poincar\'e-type inequalities}

In this paragraph we introduce two types of Poincar\'e's inequality in the probability space $L^2(\Omega)$ that will allow us to turn Theorem~\ref{th:spec} quantitative in the following two paragraphs.
Consider a function $X$ of the field $A$ seen itself as a map from $\R^d$ to $\Mab$. 
We call ``horizontal" the derivative of $X$ with respect to changes of $A$ obtained by shifting the map along the directions of $\R^d$. We call ``vertical" the derivative of $X$ with respect to changes of the values of $A$ in $\Mab$ in some spatial region of $\R^d$. We shall use Poincar\'e's inequalities for both types of derivatives.

\medskip

\noindent We start with Poincar\'e-type inequalities for the horizontal derivatives  in $L^2(\Omega)$, which are nothing but $\{\DD_j\}_{j=1,\dots,d}$.
As usual, we say that $\DD$ satisfies a Poincar\'e inequality if for all $X\in \calH(\Omega)$ such that $\expec{X}=0$, we have
\begin{equation}\label{eq:Poinca-hor}
\|X\|_{L^2(\Omega)}=\expec{X^2}\,\lesssim\,\expec{|\DD X|^2}=\|\DD X\|_{L^2(\Omega)}.
\end{equation}
This inequality holds for periodic coefficients (in which case it reduces to the Poincar\'e-Wirtinger inequality on the torus).
We can weaken this inequality by strengthening the norm of the RHS, and ask that for some $s_0>0$,
\begin{equation}\label{eq:Poinca-hor-weak}
\|X\|_{L^2(\Omega)}\,\lesssim\,\|\DD X\|_{H^{s_0}(\Omega)},
\end{equation}
where
$$
\|\DD X\|_{H^{s_0}(\Omega)}\,:=\,\expec{X (-\triangle)^{s_0} X}^{\frac{1}{2}},
$$
$-\triangle\,=\,-\DD\cdot \DD$ is the Laplacian on $L^2(\Omega)$ (whose fractional  powers can be defined by the spectral theorem). Although it is not yet clear, this inequality holds true for a subclass of almost periodic coefficients (in which case $\Omega$ is a high-dimensional torus), see \cite{Kozlov-78}.

\medskip

\noindent We turn now to the Poincar\'e-type inequality for the vertical derivatives. 
We say that the coefficients satisfy a Poincar\'e inequality (or a specral gap) for the Glauber dynamics if 
there exist $0<\rho,\ell<\infty$ such that
\begin{equation}\label{eq:Poinca-vert}
\expec{X^2}\,\leq \,\frac{1}{\rho} \int_{\R^d} \expec{\Big(\osc{X}{A|_{B_\ell(z)}}\Big)^2}dz ,
\end{equation}
where
$$
\osc{X}{A|_{B_\ell(z)}}:=\sup_{\tilde A: \tilde A|_{\R^d\setminus B_\ell(z)}}X(\tilde A)-\inf_{\tilde A: \tilde A|_{\R^d\setminus B_\ell(z)}}X(\tilde A).
$$
Such an inequality holds for Poisson random inclusions, see \cite{Gloria-Otto-10b}.

\medskip

\noindent It is not surprising that these Poincar\'e inequalities allow us to obtain quantitative estimates 
for our approximation methods. Indeed, they also allow one to quantify the error in the first two terms of the two-scale expansion, cf. \cite{bensoussan-lions-papanicolaou-78} and \eqref{eq:Poinca-hor} for periodic coefficients, \cite{Kozlov-78} and \eqref{eq:Poinca-hor-weak} for a class of almost periodic coefficients, and \cite{Gloria-Neukamm-Otto-11} and \eqref{eq:Poinca-vert} for random coefficients (in the context of discrete elliptic equations).

\medskip
\noindent The last two paragraphs of this section are dedicated to the proof of quantitative results for the approximation of the homogenized coefficients and of the correctors by regularization and extrapolation.
We start with an approach based on spectral theory, which makes more intuitive the roles of the Poincar\'e inequalities \eqref{eq:Poinca-hor} and \eqref{eq:Poinca-vert}. This approach is however limited to symmetric coefficients. We then present a second approach which solely relies on PDE analysis and allows one to cover the case of non-symmetric coefficients under any of the three assumptions  \eqref{eq:Poinca-hor},  \eqref{eq:Poinca-hor-weak}, and \eqref{eq:Poinca-vert}.

\subsection{Quantitative results via spectral theory}\label{sec:quant-spec}

Spectral theory has played an important role in the theory of stochastic homogenization of self-dual operators since the seminal contribution
by Papanicolaou and Varadhan \cite{Papanicolaou-Varadhan-79}.
In this paragraph, we assume that $A$ is a stationary ergodic field of symmetric coefficients, and shall prove a spectral representation for the quantities in Theorem~\ref{th:spec}.
This will allow us to turn Theorem~\ref{th:spec} quantitative for coefficients satisfying
\eqref{eq:Poinca-hor} or \eqref{eq:Poinca-vert}.
To this end, let us recall the spectral theory introduced in \cite{Papanicolaou-Varadhan-79}.

\medskip

\noindent Let $\calL=-\DD \cdot A(0) \DD$ be the operator defined on
$\calH(\Omega)$ as a quadratic form.
We denote by $\calL$ its Friedrichs extension on $L^2(\Omega)$.
This operator is a nonnegative self-dual operator, so that 
by the spectral theorem it admits a spectral resolution 
\begin{equation*}
\calL \,=\, \int_{\R^+} \lambda G(d\lambda).
\end{equation*}
\begin{prop}\label{prop:spec}
In the context of Theorem~\ref{th:spec}, if $A\in L^\infty(\Omega,\Mabs)$,
we have
\begin{eqnarray}
\expec{|\nabla \phi_{T,k}-\nabla \phi|^2}&\sim & \int_{\R^+} \frac{T^{-2k} }{\lambda(T^{-1}+\lambda)^{2k}}de_{\mathfrak{d}}(\lambda),\label{eq:prop-spec1}\\
|A_{T,k}-A_\ho|&\sim& \int_{\R^+} \frac{T^{-2k} }{\lambda(T^{-1}+\lambda)^{2k}}de_{\mathfrak{d}}(\lambda),\label{eq:prop-spec2}
\end{eqnarray}
where $e_{\mathfrak{d}}$ denotes the projection of the spectral resolution $G$ of $\calL=-\DD \cdot A(0) \DD$ onto the
local drift defined by $\mathfrak{d}\,:=\,\DD\cdot A(0)\xi \in \calH(\Omega)'$, 
that is, $e_{\mathfrak{d}}([0,\nu])=\int_0^\nu \expec{\mathfrak{d}G(d\lambda)\mathfrak{d}}$.
\qed
\end{prop}
\noindent Combined with the integrability property $\int_{\R^+}\frac{1}{\lambda}de_{\mathfrak{d}}(\lambda)<\infty$ proved in \cite{Papanicolaou-Varadhan-79} (this is also a direct consequence of Remark~\ref{rem:prop-spec3.1} and of the monotone convergence theorem), Proposition~\ref{prop:spec}, combined with the Lebesgue dominated convergence theorem, yields an alternative
proof of Theorem~\ref{th:spec} for symmetric coefficients.

\begin{proof}[Proof of Proposition~\ref{prop:spec}]
We split the proof into three steps.

\medskip

\step{1} Reformulation of \eqref{eq:prop-spec1}.

\noindent
Let $\tilde T \geq T>0$.
By ellipticity and the spectral theorem, 
\begin{eqnarray*}
\expec{  |\nabla \phi_{T}-\nabla \phi_{\tilde T}|^2 } &\lesssim &
\expec{  (\nabla \phi_{T}-\nabla \phi_{\tilde T})\cdot A (\nabla \phi_{T}-\nabla \phi_{\tilde T})  }\\ &=&
\int_{\R^+} \lambda \Big( \frac{1}{T^{-1}+\lambda}-\frac{1}{\tilde T^{-1}+\lambda}\Big)^2de_{\mathfrak{d}}(\lambda) \\
&=& \int_{\R^+}\frac{(T^{-1}-\tilde T^{-1})^2\lambda}{(T^{-1}+\lambda)^2(\tilde T^{-1}+\lambda)^2}de_{\mathfrak{d}}(\lambda).
\end{eqnarray*}
Since, as proved in \cite{Papanicolaou-Varadhan-79}, 
\begin{equation}\label{eq:prop-spec3}
\int_{\R^+}\frac{1}{\lambda}de_{\mathfrak{d}}(\lambda)<\infty,
\end{equation}
this implies by the Lebesgue dominated convergence theorem that $\nabla \phi_T(0)$ is a Cauchy sequence in $L^2(\Omega)$, which therefore converges to $\nabla \phi(0)$ strongly in $L^2(\Omega)$ (and not only weakly).
Let now $\psi:\R^+\to \R^+$ be a continuous function such that
$\psi(\lambda)\lesssim \frac{1}{\lambda}$ and set $\Psi:=\psi(\calL)\mathfrak{d}$. 
Then, by the spectral theorem, on the one hand,
$$
\expec{|\nabla \Psi|^2} \,\lesssim \,\expec{\nabla \Psi \cdot A \nabla \Psi}
\,=\,\int_{\R^+} \lambda \psi(\lambda)^2 de_{\mathfrak{d}}(\lambda) <\infty,
$$
and on the other hand,
$$
\expec{(\nabla \Psi-\nabla \phi_{\tilde T})\cdot A(\nabla \Psi-\nabla \phi_{\tilde T})}
\,=\,\int_{\R^+} \lambda \Big(\psi(\lambda)-\frac{1}{\tilde T^{-1}+\lambda}\Big)^2de_{\mathfrak{d}}(\lambda).
$$
Combined with the convergence of $\nabla \phi_{\tilde T}(0)$ to $\nabla \phi(0)$ in $L^2(\Omega)$ as $\tilde T\uparrow \infty$ and the Lebesgue dominated convergence theorem, this turns into:
\begin{equation}\label{eq:abstract-spec-psi}
\expec{(\nabla \Psi-\nabla \phi)\cdot A(\nabla \Psi-\nabla \phi)}
\,=\,\int_{\R^+} \lambda \Big(\psi(\lambda)-\frac{1}{\lambda}\Big)^2de_{\mathfrak{d}}(\lambda).
\end{equation}
For all $T>0$, set 
$$\psi_{T,1}:\R^+\to \R^+, \lambda\mapsto \psi_{T,1}(\lambda)=\frac{1}{T^{-1}+\lambda},$$
and define by induction for all $k\in \N$, 
$$\psi_{T,k+1}:\R^+\to \R^+, \lambda\mapsto \psi_{T,k+1}(\lambda)=\frac{1}{2^k-1}(2^k \psi_{2T,k}(\lambda)-\psi_{T,k}(\lambda)).$$
These functions are continuous and satisfy $\psi_{T,k}(\lambda) \lesssim \frac{1}{\lambda}$ (where the multiplicative constant depends on $T$ and $k$).
By definition of the Richardson extrapolation of the regularized corrector (cf. Definition~\ref{def:extra}), we also have $\phi_{T,k}=\psi_{T,k}(\calL)\mathfrak{d}$ for all $T>0$ and $k\in \N$.
Finally, by ellipticity, 
$$
\expec{ |\nabla \phi_{T,k}-\nabla \phi|^2 } \,\lesssim \,\expec{ (\nabla \phi_{T,k}-\nabla \phi)\cdot A (\nabla \phi_{T,k}-\nabla \phi) } \,\lesssim \, \expec{ |\nabla \phi_{T,k}-\nabla \phi|^2 } ,
$$
so that it enough to estimate the Dirichlet form to prove \eqref{eq:prop-spec1}.
We shall prove that for all $T>0$, $k\in \N$, and $\lambda>0$,
\begin{equation}\label{eq:induc-psi}
\frac{1}{\lambda}-\psi_{T,k}(\lambda)\,=\,2^{-\frac{1}{2}k(k-1)}T^{-k} \frac{1}{ \lambda \prod_{i=0}^{k-1} ((2^iT)^{-1}+\lambda)}.
\end{equation}
The desired estimate \eqref{eq:prop-spec1} will then follow from the combination of \eqref{eq:abstract-spec-psi}, \eqref{eq:induc-psi}, and \eqref{eq:prop-spec3} --- the latter to prove the equivalence in terms of scaling and not only the upper bound.

\medskip

\step{2} Proof of \eqref{eq:induc-psi}.

\noindent We proceed by induction. For $k=1$, \eqref{eq:induc-psi} reduces to 
$$
\frac{1}{\lambda}-\psi_{T,1}(\lambda)\,=\,\frac{1}{T^{-1}+\lambda}-\frac{1}{\lambda}\,=\,T^{-1} \frac{1}{\lambda(T^{-1}+\lambda)}.
$$
Assume that \eqref{eq:induc-psi} holds at step $k\in \N$.
Then, 
\begin{eqnarray}
\frac{1}{\lambda}-\psi_{2T,k}(\lambda)&=&2^{-\frac{1}{2}k(k-1)}(2T)^{-k} \frac{1}{ \lambda \prod_{i=0}^{k-1} ((2^{i+1}T)^{-1}+\lambda)}\nonumber\\
&=& 2^{-k} \frac{T^{-1}+\lambda}{(2^{k}T)^{-1}+\lambda}   2^{-\frac{1}{2}k(k-1)}T^{-k} \frac{1}{ \lambda \prod_{i=0}^{k-1} ((2^iT)^{-1}+\lambda)} \nonumber\\
&=&2^{-k} \frac{T^{-1}+\lambda}{(2^{k}T)^{-1}+\lambda}  \Big(\frac{1}{\lambda}-\psi_{T,k}(\lambda)\Big)
\label{eq:prop-spec-indus-int}
\end{eqnarray}
Hence the induction rule for $\psi_{T,k+1}$ and \eqref{eq:induc-psi} at step $k\in \N$ yield
\begin{eqnarray*}
\frac{1}{\lambda}-\psi_{T,k+1}(\lambda)&=&
\frac{2^k}{2^k-1}\Big(\frac{1}{\lambda}-\psi_{2T,k+1}(\lambda)\Big)-\frac{1}{2^k-1}\Big(\frac{1}{\lambda}-\psi_{T,k+1}(\lambda)\Big)\\
&\stackrel{\eqref{eq:prop-spec-indus-int}}{=}&\frac{1}{2^k-1}\Big(\frac{1}{\lambda}-\psi_{T,k}(\lambda)\Big) \Big(\frac{T^{-1}+\lambda}{(2^{k}T)^{-1}+\lambda}-1 \Big) 
\\
&=&2^{-k}T^{-1}\Big(\frac{1}{\lambda}-\psi_{T,k}(\lambda)\Big)\frac{1}{(2^{k}T)^{-1}+\lambda}\\
&\stackrel{\eqref{eq:induc-psi} \text{ at step }k }{=}& 2^{-\frac{1}{2}k(k+1)}T^{-(k+1)} \frac{1}{ \lambda \prod_{i=0}^{k} ((2^iT)^{-1}+\lambda)},
\end{eqnarray*}
as desired.

\medskip

\step{3} Proof of \eqref{eq:prop-spec2}.

\noindent Recall the weak form of the corrector equation for $\phi$:  For all $\Psi \in \calH(\Omega)$,
\begin{equation}\label{eq:weak-form-corr-spec}
\expec{\DD \Psi\cdot A(0)(\xi+\nabla \phi(0))}\,=\,0, 
\end{equation}
and since $\nabla \phi_T(0) \to \nabla \phi(0)$ strongly in $L^2(\Omega,\R^d)$, we also have
\begin{equation}\label{eq:weak-form-corr-spec2}
\quad \expec{\nabla \phi\cdot A(\xi+\nabla \phi)}\,=\,0.
\end{equation}
By symmetry of $A$ and \eqref{eq:weak-form-corr-spec} \& \eqref{eq:weak-form-corr-spec2} we then have for all $\Psi \in \calH(\Omega)$,
\begin{eqnarray*}
\lefteqn{\expec{(\xi+\DD \Psi)\cdot A(0)(\xi+\DD \Psi)}-\expec{(\xi+\nabla \phi)\cdot A(\xi+\nabla \phi)}}\nonumber \\
&=&\expec{(\DD \Psi-\nabla \phi(0))\cdot A(0)(\xi+\DD \Psi)}+\expec{(\xi+\nabla \phi(0))\cdot A(0)(\DD \Psi-\nabla \phi(0))}\nonumber\\
&\stackrel{\mbox{symmetry}}{=}&\expec{(\DD \Psi-\nabla \phi(0))\cdot A(0)(\xi+\DD \Psi)}+\expec{(\DD \Psi-\nabla \phi(0))\cdot A(0) (\xi+\nabla \phi(0))}\nonumber\\
&\stackrel{\eqref{eq:weak-form-corr-spec}\& \eqref{eq:weak-form-corr-spec2}}{=}&\expec{(\xi+\nabla \phi(0))\cdot A(0)(\DD \Psi-\nabla \phi(0))}-\expec{(\DD \Psi-\nabla \phi(0))\cdot A(0) (\xi+\nabla \phi(0))}\nonumber\\
&=&\expec{ (\DD \Psi-\nabla \phi(0))\cdot A(0) (\DD \Psi-\nabla \phi(0)) }.
\end{eqnarray*}
Estimate \eqref{eq:prop-spec2} then follows from the definition of $A_{T,k}$ and \eqref{eq:prop-spec1} 
by taking $\Psi=\phi_{T,k}(0)$ in this identity.
\end{proof}
\noindent Let us now show how Proposition~\ref{prop:spec} allows one to turn the consistency result
of Theorem~\ref{th:spec} quantitative if we assume \eqref{eq:Poinca-hor} or \eqref{eq:Poinca-vert}.

\medskip

\noindent In the case when \eqref{eq:Poinca-hor} holds (namely for periodic coefficients), the elliptic operator $\calL$ is not degenerate and has a spectral gap: there exists $\mu>0$ such that the non-negative measure $e_{\mathfrak{d}}$ satisfies $e_{\mathfrak{d}}((0,\mu))=0$.
Hence, for all $k\in \N$
$$
\int_{\R^+} \frac{1}{\lambda^k}d_{\mathfrak{d}}(\lambda) \,=\,\int_{\mu}^\infty \frac{1}{\lambda^k}d_{\mathfrak{d}}(\lambda)\,\stackrel{\eqref{eq:prop-spec3}}{\lesssim} \, \mu^{-k}+1 <\infty,
$$
which makes the estimates of Proposition~\ref{prop:spec} explicit in $T$ (see Theorem~\ref{th:quant-estim-Perio} below for the precise statement for general non-necessarily symmetric periodic coefficients).

\medskip

\noindent In the case when the coefficients satisfy \eqref{eq:Poinca-vert}, the operator $\calL$ is a degenerate elliptic operator which does not have a spectral gap. Obtaining quantitative estimates is then much more subtle.
In~\cite{Gloria-Otto-10b}, Otto and the first author obtained the following bounds on the associated spectral measure. With the notation of Theorem~\ref{th:spec}, we say that the spectral exponents are \emph{at least} $(\gamma_1,\gamma_2)\in \R^+\times \R^+$ if for all $0<\lambda\leq 1$,
\begin{equation}\label{eq:bott-spec}
e_{\mathfrak{d}}([0,\lambda])\,\lesssim\,
\lambda^{\gamma_1}(\log^{\gamma_2}(\lambda^{-1}) +1).
\end{equation}
By \cite[Corollary~2]{Gloria-Otto-10b}, \eqref{eq:Poinca-vert} implies that
\begin{equation}\label{eq:spec-exponents2}
\begin{array}{rll}
2\le d<6:&\gamma_1=\frac{d}{2}+1,& \gamma_2= 0,\\
d=6:&\gamma_1=4, &\gamma_2=1,\\
d>6:&\gamma_1=4,& \gamma_2=0.
\end{array}
\end{equation}
From these estimates of the spectral exponents, one may deduce the following error estimates for the regularization and extrapolation methods.
\begin{theo}\label{th:quant-estim-Poisson}
Let the symmetric coefficients $A\in L^\infty(\Omega,\Mabs)$ satisfy \eqref{eq:Poinca-vert}, and let $A_\ho$
be the associated homogenized coefficients.
Let $\xi \in \R^d$ be a fixed unit vector and let $\phi$ be the corrector in direction $\xi$. For all $T>0$ and $k\in \N$, let $\phi_{T,k}$ be the extrapolation of the regularized corrector of Definition~\ref{def:extra}, and $A_{T,k}$ be the approximation of the homogenized coefficients of Definition~\ref{def:extra-hom}.
Let $k(d):=[\frac{d}{4}]$ (where $[\cdot]$ denotes here the smallest integer larger or equal to). Then  for all $d\geq 2$ and for all $k'\geq \min\{k(d),2\}$, we have
\begin{equation}\label{eq:syst-err-sto-cont11}
\expec{|\nabla \phi_{T,k'}-\nabla \phi|^2}\,\lesssim\, 
\left|
\begin{array}{rcl}
2\le d<6&:&T^{-\frac{d}{2}},\\
d=6&:&T^{-3}\log T,\\
d>6&:&T^{-3}.
\end{array}
\right.
\end{equation}
Likewise, for all $k'\geq  \min\{k(d),2\}$
\begin{equation}\label{eq:syst-err-sto-cont21}
|A_{T,k'}-A_\ho|\,\lesssim\, 
\left|
\begin{array}{rcl}
2\leq d<6&:&T^{-\frac{d}{2}},\\
d=6&:&T^{-3}\log T,\\
d>6&:&T^{-3}.
\end{array}
\right.
\end{equation}
In the estimates above, the multiplicative constant depends on $k'$, next to $\alpha,\beta$, and $d$.
\qed
\end{theo}
\begin{rem}\label{rem:sto-disc}
In the case of discrete linear elliptic equations on $\Z^d$ with coefficients
that satisfy a discrete version of \eqref{eq:Poinca-vert} (e.g. independent and identically distributed coefficients), Neukamm, Otto, and the first author proved in \cite{Gloria-Neukamm-Otto-14} the following \emph{optimal} values of the spectral exponents: $(\frac{d}{2}+1,0)$ for all $d\geq 2$. 
In particular, this can be used to prove the following \emph{optimal} estimates for the regularization and extrapolation method in this discrete setting:
For all $k'\geq k(d)$, 
\begin{equation*}
\expec{|\nabla \phi_{T,k'}-\nabla \phi|^2},|A_{T,k'}-A_\ho| \,\lesssim\, 
 T^{-\frac{d}{2}}.
\end{equation*}
We believe that these estimates hold as well in the case of continuum equations for coefficients satisfying \eqref{eq:Poinca-vert}.
\qed
\end{rem}
\begin{rem}\label{rk:optimality-k=1}
In view of Remark~\ref{rem:sto-disc}, Theorem~\ref{th:quant-estim-Poisson} is optimal in terms of scaling for $d=2,3,4,5,6$ up to a logarithmic correction in dimension $d=6$. In particular, this implies that for dimensions of interest in practice (say, $d=2$ and $d=3$), $k'=1$ is enough to reach the optimal scaling and Richardson extrapolation does not generally reduce the error further for random coefficients that satisfy \eqref{eq:Poinca-vert} (as opposed to \eqref{eq:Poinca-hor}). \qed
\end{rem}
\begin{proof}[Proof of Theorem~\ref{th:quant-estim-Poisson}]
Theorem~\ref{th:quant-estim-Poisson} is a direct consequence of Proposition~\ref{th:spec} and 
of \cite{Gloria-Otto-10b}.
The fundamental theorem of calculus and Fubini's theorem imply that for all $f\in C^1((0,1])$,
\begin{eqnarray}
\int_0^{1} f(\lambda)d e_{\mathfrak{d}}(\lambda) &=& -\int_{\lambda=0}^1\int_{\hat \lambda=\lambda}^{1}f'(\hat \lambda) d\hat \lambda d e_{\mathfrak{d}}(\lambda)+ f(1) \int_{\lambda=0}^1 d e_{\mathfrak{d}}(\lambda)\nonumber\\
&= & -\int_{\hat \lambda=0}^1f'(\hat \lambda)e_{\mathfrak{d}}([0,\hat \lambda ])\, d\hat \lambda  + f(1) e_{\mathfrak{d}} ([0,1]).\label{eq:fond-th-calc}
\end{eqnarray}
Used with $f(\lambda)=\frac{1 }{\lambda(T^{-1}+\lambda)^{2k'}}$ and combined with Proposition~\ref{th:spec}, this yields for all $k'\in \N$
\begin{eqnarray*}
\expec{|\nabla \phi_{T,k'}-\nabla \phi|^2}&\lesssim & \int_{\R^+} \frac{T^{-2k'} }{\lambda(T^{-1}+\lambda)^{2k'}}de_{\mathfrak{d}}(\lambda)\\
&\leq & \int_0^1 \frac{T^{-2k'} }{\lambda(T^{-1}+\lambda)^{2k'}}de_{\mathfrak{d}}(\lambda)+\int_{\R^+} \frac{T^{-2k'} }{\lambda}de_{\mathfrak{d}}(\lambda) \\
&\stackrel{\eqref{eq:fond-th-calc} \& \eqref{eq:prop-spec3}}{\lesssim}&
 \int_0^1 \frac{T^{-2k'} }{\lambda^2(T^{-1}+\lambda)^{2k'}}e_{\mathfrak{d}}([0, \lambda ])d\lambda+T^{-2k'}.
\end{eqnarray*}
We then appeal to \cite[Corollary~2]{Gloria-Otto-10b} in the form of \eqref{eq:spec-exponents2},
and assume that $2k'+1\geq \gamma_1$ (which yields the definition of $k(d)$).
According to \eqref{eq:spec-exponents2}, $\gamma_1\geq 2$, so that we have in that case
\begin{eqnarray*}
\expec{|\nabla \phi_{T,k'}-\nabla \phi|^2}&\stackrel{\eqref{eq:bott-spec} \& \eqref{eq:spec-exponents2}}{\lesssim} &
T^{-2k'}  \int_0^1 \frac{\log^{\gamma_2}(\lambda^{-1}) +1}{(T^{-1}+\lambda)^{2k'+2-\gamma_1}}d\lambda+T^{-2k'} \\
&{\lesssim}& T^{-\gamma_1+1} \log^{\gamma_2}T,
\end{eqnarray*}
which yields the claim.
\end{proof}

\subsection{Quantitative results via PDE analysis}

In this paragraph we address the quantitative analysis for non-necessarily symmetric coefficients satisfying \eqref{eq:Poinca-hor}, \eqref{eq:Poinca-hor-weak}, or \eqref{eq:Poinca-vert}.
We start with \eqref{eq:Poinca-hor}.
\begin{theo}\label{th:quant-estim-Perio}
Let $A\in L^\infty(\R^d,\Mab)$ be periodic non-necessarily symmetric coefficients (so that they satisfy \eqref{eq:Poinca-hor}), and let $A_{\ho}$ be the associated homogenized coefficients. 
Let $\xi,\xi'\in \R^d$ be fixed unit vectors, and denote by $\phi,\phi'\in H^1_\per(Q)$ the
periodic corrector and dual corrector in directions $\xi,\xi'$, respectively.
For all $T>0$ and $k\in \N$, let $\phi_{T,k}$ and $\phi_{T,k}'$ be the associated regularized correctors of Definition~\ref{def:extra}, and ${A}_{T,k}$ the approximations of the homogenized  coefficients 
of Definition~\ref{def:extra-hom}.
Then, we have
\begin{eqnarray}\label{apprx3}
\|\nabla(\phi_{T,k}-\phi)\|_{L^2(Q)},\|\nabla(\phi_{T,k}'-\phi')\|_{L^2(Q)} \lesssim T^{-k} ,
\end{eqnarray}
and
\begin{eqnarray}\label{apprx5}
\left|A_{\ho}-{A}_{T,k} \right|\lesssim T^{-2k},
\end{eqnarray}
where the multiplicative constants depend on $k$, next to $\alpha,\beta$, and $d$.
\qed
\end{theo}
\noindent Before we turn to the proof, let us emphasize that the only noteworthy feature of periodic coefficients
in this context is the validity of the Poincar\'e inequality \eqref{eq:Poinca-hor}, and we could have stated Theorem~\ref{th:quant-estim-Perio} under the assumptions that the coefficients be stationary and satisfy \eqref{eq:Poinca-hor} (which essentially reduces to the case of periodic coefficients).
\begin{proof}
We split the proof into two steps.

\medskip
\step{1} Proof of \eqref{apprx3}.

\noindent We only prove the estimate for $\phi_{T,k}$.
We first claim that $\phi_{T,k}$ satisfies for all $k\in \N$ the equation
\begin{equation}\label{eq:induc-phik}
T^{-1} \phi_{T,k+1}-\nabla \cdot A(\xi+\nabla \phi_{T,k+1})\,=\,T^{-1}  \phi_{2T,k}.
\end{equation}
We proceed by induction.
For $k=1$, this follows from a direct calculation. Assume \eqref{eq:induc-phik} holds at step $k$.
Then, by the induction assumption for $\phi_{T,k+1}$ and $2^{k+1}\phi_{2T,k+1}$, we have
\begin{eqnarray*}
T^{-1} \phi_{T,k+1}-\nabla \cdot A(\xi+\nabla \phi_{T,k+1})&=&T^{-1}\phi_{2T,k}, \\
T^{-1}(2^{k} \phi_{2T,k+1})-\nabla \cdot A(2^{k+1}\xi+\nabla 2^{k+1}\phi_{2T,k+1})&=&2^{k}T^{-1}\phi_{4T,k},
\end{eqnarray*}
so that the definition \eqref{eq:def:extrapo} of the Richardson extrapolation at step $k+1$ yields
\begin{multline*}
T^{-1}\Big(\frac{2^k}{2^{k+1}-1} \phi_{2T,k+1}-\frac{1}{2^{k+1}-1}\phi_{T,k+1}\Big)
-\nabla \cdot A (\xi+\nabla \phi_{T,k+2})\\
\,=\,T^{-1} \Big(\frac{2^k}{2^{k+1}-1}\phi_{4T,k}-\frac{1}{2^{k+1}-1}\phi_{2T,k}\Big),
\end{multline*}
that is, using the identity $\frac{2^k}{2^{k+1}-1}\,=\,\frac{2^{k+1}}{2^{k+1}-1}-\frac{2^{k+1}-2^k}{2^{k+1}-1}$,
\begin{eqnarray*}
\lefteqn{T^{-1}\phi_{T,k+2} -\nabla \cdot A (\xi+\nabla \phi_{T,k+2})}\\
&=&\frac{T^{-1}}{2^{k+1}-1} \Big((2^{k+1}-2^k) \phi_{2T,k+1}+2^k\phi_{4T,k}-\phi_{2T,k}\Big) \\
&\stackrel{\eqref{eq:def:extrapo}\text{ at step }k}{=}&\frac{T^{-1}}{2^{k+1}-1} \Big((2^{k+1}-2^k)  \phi_{2T,k+1}+(2^k-1) \phi_{2T,k+1}\Big) \\
&=&T^{-1}\phi_{2T,k+1},
\end{eqnarray*}
as claimed.

\medskip

\noindent From equation~\eqref{eq:induc-phik} we deduce that
\begin{equation}\label{eq:rec-phi_T-phi}
T^{-1} (\phi_{T,k+1}-\phi)-\nabla \cdot A \nabla (\phi_{T,k+1}-\phi)\,=\,T^{-1}  (\phi_{2T,k}-\phi).
\end{equation}
Testing this equation with test-function $\phi_{T,k+1}-\phi$, integrating by parts and using Poincar\'e's inequality on $H^1_\per(Q)$ then yields
$$
\|\nabla (\phi_{T,k+1}-\phi)\|_{L^2(Q)} \,\lesssim \, T^{-1} \|\nabla (\phi_{2T,k}-\phi)\|_{L^2(Q)}.
$$
This proves \eqref{apprx3} by induction, starting with the elementary energy estimate
$$
\|\nabla (\phi_{T,1}-\phi)\|_{L^2(Q)} \,\lesssim \, T^{-1} \|\phi\|_{L^2(Q)} \|\phi_{T,1}-\phi\|_{L^2(Q)},
$$
combined itself with Poincar\'e's inequality.

\medskip

\step{2} Proof of \eqref{apprx5}.

\noindent The proof is similar to Step~3 in the proof of Proposition~\ref{prop:spec}. We
recall the weak form of the corrector and dual corrector equations: For all 
$\Psi \in H^1_\per(Q)$
\begin{equation}\label{eq:weak-form-corr-spec3}
\int_Q \nabla \Psi \cdot A^*(\xi'+\nabla \phi')\,=\,0, \quad \int_Q \nabla \Psi \cdot A(\xi+\nabla \phi)\,=\,0.
\end{equation}
For all $\Psi_1,\Psi_2 \in H^1_\per(Q)$, we then have using \eqref{eq:weak-form-corr-spec3} once
with $\Psi= \Psi_1- \phi$  and $A^*$ and once with $\Psi= \Psi_2-\phi'$ and $A$:
\begin{eqnarray*}
\lefteqn{\fint_Q{(\xi'+\nabla \Psi_2)\cdot A(\xi+\nabla \Psi_1)}-\fint_Q{(\xi'+\nabla \phi')\cdot A(\xi+\nabla \phi)}}\nonumber \\
&=&\fint_Q{(\nabla \Psi_2-\nabla\phi')\cdot A(\xi+\nabla \Psi_1)}+\fint_Q{(\xi'+\nabla \phi')\cdot A(\nabla \Psi_1-\nabla \phi)}\nonumber\\
&{=}&\fint_Q{(\nabla \Psi_2-\nabla\phi')\cdot A(\xi+\nabla \Psi_1)}+\fint_Q{(\nabla \Psi_1-\nabla \phi)\cdot A^* (\xi'+\nabla \phi')}\nonumber\\
&\stackrel{\eqref{eq:weak-form-corr-spec3}}{=}&\fint_Q{(\nabla \Psi_2-\nabla \phi')\cdot A(\xi+\nabla \Psi_1)}-\fint_Q{(\nabla \Psi_2-\nabla \phi')\cdot A(\xi+\nabla \phi)}\nonumber\\
&=&\fint_Q{ (\nabla \Psi_2-\nabla \phi')\cdot A (\nabla \Psi_1-\nabla \phi) }.
\end{eqnarray*}
Hence \eqref{apprx5} follows from \eqref{apprx3} by taking $\Psi_1=\phi_{T,k}$ and $\Psi_2=\phi_{T,k}'$ in the identity above.
\end{proof}

\medskip

\noindent We turn now to the case of coefficients satisfying \eqref{eq:Poinca-hor-weak}, and consider a subclass of almost periodic coefficients introduced by Kozlov in \cite{Kozlov-78}. We start with
a definition.
\begin{defi}[Kozlov class of almost periodic coefficients]\label{def:Kozlov}
Let $N\in \N$ and $\Gamma=\{\gamma^j\}_{1\leq j\leq N}$ be a finite set of vectors $\gamma^j\in \R^d$
such that there exist $s_0>0$ and $C>0$ for which %
\begin{equation}\label{eq:diophan}
\min_{1\leq j\leq N} |\gamma^j\cdot \xi'|\,\geq\,C|\xi'|^{-s_0}
\end{equation}
for all $\xi'\in \Z^d\setminus \{0\}$.
We say that the coefficients $A$ are in the Kozlov $\Gamma$-class of almost periodic coefficients
if all the entries of $A$ are trigonometric polynomials with Fourier exponents in $\Gamma$, that is,
such that $[A(x)]_{kl} \,=\,\sum_{j=1}^N c_{jkl} \exp(i\gamma^j\cdot x)$ for some coefficients $c_{jkl}$ ($k,l\in \{1,\dots,d\}$).
\qed
\end{defi}
\begin{rem}\label{rem:dioph}
A standard Diophantine condition shows that if $\{\gamma^j_k\}_{1\leq j\leq N,1\leq k\leq d}$ are algebraic, then \eqref{eq:diophan} holds for the exponent
$s_0=\frac{1}{\max_{1\leq k\leq d} {N_k}}$ where $N_k$ denotes the cardinal of any rationally independent basis over $\Z$ of $\{\gamma^j_k\}_{1\leq j\leq N}$.
\qed
\end{rem}
\begin{theo}\label{th:quant-estim-quasiper}
Let $A\in L^\infty(\R^d,\Mab)$ be in the class of Kozlov almost periodic coefficients, and let $A_\ho$
be the associated homogenized coefficients.
Let $\xi,\xi' \in \R^d$ be fixed unit vectors, and denote by $\phi,\phi'$ the associated corrector and dual corrector in directions $\xi,\xi'$, respectively.
For all $T>0$ let $\phi_{T},\phi_T'$ be the associated regularized correctors of Definition~\ref{def:extra}, and $A_{T}$ be the approximation of the homogenized coefficients of Definition~\ref{def:extra-hom}.
Then, we have
\begin{eqnarray}\label{apprx3-kozlov}
\calM(|\nabla \phi_{T,k}-\nabla \phi|^2)^{\frac{1}{2}},\calM(|\nabla \phi_{T,k}'-\nabla \phi'|^2)^{\frac{1}{2}} \lesssim T^{-k} ,
\end{eqnarray}
and
\begin{eqnarray}\label{apprx5-kozlov}
\left|A_{\ho}-{A}_{T,k} \right|\lesssim T^{-2k},
\end{eqnarray}
where the multiplicative constants depend on $k$, next to $\alpha,\beta$, the class $\Gamma$, and $d$.
\qed
\end{theo}
\begin{rem}\label{rk:quant-estim-quasiper}
As we shall see in the proof of Theorem~\ref{th:quant-estim-quasiper}, the correctors and regularized correctors are smooth and the estimate~\eqref{apprx3-kozlov} can be strengthened to:
\begin{eqnarray}\label{apprx3-kozlov-reg}
\sup_{\R^d}|\nabla \phi_{T,k}-\nabla \phi|,\sup_{\R^d}|\nabla \phi_{T,k}'-\nabla \phi'| \lesssim T^{-k}.
\end{eqnarray}
\qed
\end{rem}
\begin{proof}[Proof of Theorem~\ref{th:quant-estim-quasiper}]
In order to prove Theorem~\ref{th:quant-estim-quasiper}, we first recall how Kozlov defines correctors in \cite{Kozlov-78} and show how to adapt the proof of Theorem~\ref{th:quant-estim-Perio} in this setting.

\medskip

\step{1} Kozlov's construction of correctors.

\noindent When the coefficients are in some $\Gamma$-Kozlov class of almost periodic coefficients, one can 
take as $\Omega$ a $(M=\sum_{k=1}^d N_k)$-dimensional torus $\Pi_M$, where $N_k$ is as in Remark~\ref{rem:dioph}, and extend $A$ as a map $\overline A: \Pi_M\to \Mab$.
As shown in \cite[Proof of Theorem~1 and Proof of Theorem~4]{Kozlov-78}, any smooth periodic function $\overline\psi$ on $\Pi_M$ defines a smooth almost periodic function $\psi$ on $\R^d$ (in the sense of Besicovitch) by the diagonal restriction
$$
\psi(y_1,\dots,y_d):=\overline \psi(\underbrace{y_1,\dots,y_1}_{N_1\text{ times}},\dots,\underbrace{y_d,\dots,y_d}_{N_d\text{ times}}).
$$
Conversely, the differential operators $\nabla_i$ on $\R^d$ translate by the diagonal restriction into 
differential operators $\DD_i$ on $\Pi_M$. The regularized corrector equation then turns into: find $\overline \phi_T\in H^1_\per(\Pi_M)$ such that
$$
T^{-1}\overline \phi_T-\DD\cdot \overline A(\xi+\DD \overline \phi_T)\,=\,0 \text{ in }\Pi_M.
$$
Note that $H^1_\per(\Pi_M)$ does not coincide with $\calH(\Pi_M)$ since we use standard differential operators $\{\nabla_i\}_{i=1,\dots,M}$ on $\Pi_M$ to define $H^1_\per(\Pi_M)$ whereas we use the differential operators $\{\DD_i\}_{i=1,\dots,d}$ to define
$\calH(\Pi_M)$.
We first solve the corrector equation in the Hilbert space $\calH(\Pi_M)=\{\overline\psi\in L^2(\Pi_M):\int_{\Pi_M}|\DD\overline \psi|^2<\infty\}$. We shall then show a posteriori using elliptic regularity that the regularized corrector $\overline \phi_T$ belongs to $H^1_\per(\Pi_M)$ uniformly in $T$ (and even to $C^\infty(\Pi_M)$).
By construction of $\overline A$ and $\DD$, the coercivity of $A$ yields the following coercivity estimate
$$
\Big(\calL_T \overline\phi_T,\overline\phi_T\Big)_{L^2(\Pi_M)} \,\gtrsim \, T^{-1}\int_{\Pi_M} \overline\phi_T^2+\int_{\Pi_M}|\DD\overline\phi_T|^2.
$$
Using in addition the bound
$$
\Big|\int_{\Pi_M}\DD\overline \psi\cdot \overline A \xi\Big| \,\lesssim \, \Big(\int_{\Pi_M}|\DD\overline \psi|^2\Big)^{\frac{1}{2}},
$$
we deduce from the Lax-Milgram theorem the existence of a unique weak solution in $\calH(\Pi_M)$
of the regularized corrector equation, and the a priori estimate 
\begin{equation}\label{eq:garding-1}
\|\DD\overline \phi_T\|_{L^2(\Pi_M)}^2+T^{-1}\|\overline\phi_T\|_{L^2(\Pi_M)}^2\,\lesssim\,1,
\end{equation}
where the multiplicative constant is independent of $T$.
Now comes the regularity argument. 
On the one hand, G{\aa}rding's inequality yields for all $s>0$ and $\overline\psi\in L^2(\Pi_M)$,
\begin{equation}\label{eq:garding-2}
\Big( (-\triangle)^s \calL_T \overline\psi,\overline\psi\Big)_{L^2(\Pi_M)} \,\geq \,c_1(s) \|\DD \overline\psi\|_{H^s(\Pi_M)}^2-c_2(s)\|\DD \overline\psi\|_{L^2(\Pi_M)}^2,
\end{equation}
where $\triangle$ is the Laplacian in $\R^M$ (and the constants $0<c_1(s),c_2(s)<\infty$ depend on $s$).
With $\overline \psi=\overline \phi_T$, the $L^2(\Pi_M)$-norm of the first term of the LHS is bounded as follows:
\begin{equation}\label{eq:garding-3}
\|(-\triangle)^s \calL_T \overline\phi_T\|_{L^2(\Pi_M)}=\|(-\triangle)^s \DD\cdot \overline A\xi\|_{L^2(\Pi_M)}\,\leq \,\|\overline A\|_{H^{2s+1}(\Pi_M)}\,\lesssim\,1
\end{equation}
since $\overline A$ is smooth. The combination of \eqref{eq:garding-2} and \eqref{eq:garding-3}
then yields by Cauchy-Schwarz' inequality
\begin{equation}\label{eq:garding-4}
\|\overline A\|_{H^{2s+1}(\Pi_M)}\|\overline\phi_T\|_{L^2(\Pi_M)}\,\geq\, c_1(s) \|\DD \overline\phi_T\|_{H^s(\Pi_M)}^2-c_2(s)\|\DD \overline\phi_T\|_{L^2(\Pi_M)}^2.
\end{equation}
On the other hand, since $\Gamma$ is as in Definition~\ref{def:Kozlov}, we have the following weak Poincar\'e inequalities (cf. \cite[Proof of Theorem~4]{Kozlov-78}): for all $s\ge s_0$ and all $\overline \psi\in L^2(\Pi_M)$ such that 
$\fint_{\Pi_M}\psi=0$,
\begin{equation}\label{eq:garding-5}
\|\DD \overline \psi\|_{H^s(\Pi_M)} \,\geq\, c\|\overline \psi\|_{H^{s-s_0}(\Pi_M)},
\end{equation}
which we shall use for $\overline\phi_T$ in the form 
\begin{equation}\label{eq:garding-6}
\|\DD \overline \phi_T\|_{H^{s_0}(\Pi_M)} \,\geq\, c\|\overline \phi_T\|_{L^2(\Pi_M)}.
\end{equation}
Inserting \eqref{eq:garding-6} into \eqref{eq:garding-4} for $s\ge s_0$ and using Young's inequality
to absorb the LHS in the RSH lead to 
\begin{equation*}
\frac{1}{2}\Big(\frac{1}{c\sqrt{c_1(s)}}\Big)^2\|\overline A\|_{H^{2s+1}(\Pi_M)}^2
+c_2(s)\|\DD \overline\phi_T\|_{L^2(\Pi_M)}^2\,\geq \, \frac{c_1(s)}{2} \|\DD \overline\phi_T\|_{H^s(\Pi_M)}^2.
\end{equation*}
Combined with the a priori estimate~\eqref{eq:garding-1}, this yields for all $s\ge s_0$
$$
 \|\DD \overline\phi_T\|_{H^s(\Pi_M)}\,\lesssim\, 1,
$$
and therefore by G{\aa}rding's inequality \eqref{eq:garding-5} again, for all $s\ge 0$,
$$
 \|\overline\phi_T\|_{H^s(\Pi_M)}\,\lesssim\, 1,
$$
where the multiplicative constant depends on $s$ but not on $T$.
This proves in particular that $\phi_T$ is smooth and that the limit $\overline \phi$ of $\overline \phi_T$ as $T\uparrow \infty$ exists and defines a function of class $C^\infty(\Pi_M)$ (uniqueness of correctors is standard). The associated diagonal restrictions $\phi_T$ and $\phi$ 
are therefore smooth almost periodic functions.

\medskip

\step{2} Quantitative estimates.

\noindent The quantitative estimates are now elementary consequences of Step~1 and of the algebra of the proof of Theorem~\ref{th:quant-estim-Perio}.
Indeed, since $\overline \phi_T$ anf $\overline \phi$ are smooth periodic functions, their Richardson extrapolations satisfy the following version of \eqref{eq:rec-phi_T-phi}: for all $k\in \N$ and $T>0$:
\begin{equation}\label{eq:induction-k=0-Garding-1}
T^{-1} (\overline\phi_{T,k+1}-\overline\phi)-\DD \cdot \overline A \DD (\overline\phi_{T,k+1}-\overline\phi)\,=\,T^{-1}  (\overline\phi_{2T,k}-\overline\phi).
\end{equation}
As for Theorem~\ref{th:quant-estim-Perio}, we proceed by induction. We shall first prove that
for all $s\ge 0$,
\begin{equation}\label{eq:induction-k=0-Garding}
\| \overline \phi_T- \overline \phi\|_{H^s(\Pi_M)}\,\lesssim\, T^{-1},
\end{equation}
where the multiplicative constant depends on $s$ but not on $T$.
We start with G{\aa}rding's inequality for $s\ge s_0$, that we combine with
the a priori estimate
$$ 
\int_{\Pi_M}\DD (\overline \phi_T-\overline \phi)\cdot \overline A\,\DD (\overline \phi_T-\overline \phi)
\,\leq \, -T^{-1} \int_{\Pi_M} \overline \phi(\overline \phi_T-\overline \phi)
$$
to get
\begin{eqnarray*}
 \lefteqn{T^{-1}\|\overline \phi\|_{H^{2s}(\Pi_M)} \|\overline \phi_T-\overline \phi\|_{L^2(\Pi_M)}
\,\geq \,
\Big( (-\triangle)^s \calL_T (\overline \phi_T-\overline \phi),\overline \phi_T-\overline \phi\Big)_{L^2(\Pi_M)}}
\\
&\geq& c_1(s) \|\DD (\overline \phi_T-\overline \phi)\|_{H^s(\Pi_M)}^2-c_2(s)\|\DD (\overline \phi_T-\overline \phi)\|_{L^2(\Pi_M)}^2 \\
&\geq & c_1(s) \|\DD (\overline \phi_T-\overline \phi)\|_{H^s(\Pi_M)}^2-c_2(s)T^{-1}\|\overline \phi_T-\overline \phi\|_{L^2(\Pi_M)}\|\overline \phi\|_{L^2(\Pi_M)},
\end{eqnarray*}
up to changing $c_2(s)$ in the last line.
We then appeal to the weak Poincar\'e inequality \eqref{eq:garding-5} and use that $s\ge s_0$ to turn this estimate into
$$
 \|\DD (\overline \phi_T-\overline \phi)\|_{H^s(\Pi_M)} \, \leq \, \frac{1}{c_1(s) c}(T^{-1}\|\overline \phi\|_{H^{2s}(\Pi_M)}+c_2(s)T^{-1}\|\overline \phi\|_{L^2(\Pi_M)}).
$$
Since we have proved in Step~1 that $\overline \phi$ is smooth, this yields \eqref{eq:induction-k=0-Garding} for all $s\ge 0$ by \eqref{eq:garding-5}.

\medskip

\noindent We now turn to the induction argument proper, and assume that at step~$k\in \N$,
for all $s\ge 0$,
\begin{equation}\label{eq:induction-k=0-Garding-2}
\|\overline\phi_{T,k}-\overline\phi\|_{H^s(\Pi_M)}\,\lesssim \, T^{-k},
\end{equation}
where the multiplicative constant depends on $s$ but not on $T$.
Arguing as above and starting from \eqref{eq:induction-k=0-Garding-1},  we end up with the estimate 
$$
 \|\DD (\overline\phi_{T,k+1}-\overline\phi)\|_{H^s(\Pi_M)} \, \leq \, \frac{1}{c_1(s) c}(T^{-1}\|\overline\phi_{2T,k}-\overline\phi\|_{H^{2s}(\Pi_M)}+c_2(s)T^{-1}\|\overline\phi_{2T,k}-\overline\phi\|_{L^2(\Pi_M)}),
$$
from which the induction hypothesis at step $k+1$ follows.

\medskip

\noindent Since \eqref{eq:induction-k=0-Garding-2} holds for all $s\ge 0$, the same estimate holds
for the $C^k(\Pi_M)$-norms by Sobolev embedding. This completes the proof of \eqref{apprx3-kozlov-reg}
by using the diagonal restriction (which is continuous from $C^k(\Pi_M)$ to $C^k(\R^d)$ for all $k\in \R^+$).

\medskip

\noindent The results for the homogenized coefficients then follow from the same calculations as in Theorem~\ref{th:quant-estim-Perio}, based on the approximation of the correctors.
\end{proof}

\medskip

\noindent We conclude with the case of non-necessarily symmetric coefficients that satisfy \eqref{eq:Poinca-vert}. As proved by Otto and the first author in \cite[Theorem~2 \& Proposition~2]{Gloria-Otto-10b}:
\begin{theo}\label{th:quant-estim-Poisson2}
Let $A\in L^\infty(\Omega,\Mab)$ be stationary random coefficients
satisfying~\eqref{eq:Poinca-vert}, and let $A_\ho$
be the associated homogenized coefficients.
Let $\xi,\xi' \in \R^d$ be fixed unit vectors, and denote by $\phi,\phi'$ the associated corrector and dual corrector in directions $\xi,\xi'$, respectively.
For all $T>0$ let $\phi_{T},\phi_T'$ be the associated regularized correctors of Definition~\ref{def:extra}, and $A_{T}$ be the approximation of the homogenized coefficients of Definition~\ref{def:extra-hom}.
Then, for all $d\geq 2$, we have
\begin{equation}\label{eq:syst-err-sto-cont12}
\expec{|\nabla \phi_{T}-\nabla \phi|^2},\expec{|\nabla \phi_{T}'-\nabla \phi'|^2}\,\lesssim\, 
\left|
\begin{array}{rcl}
2\le d<4&:&T^{-\frac{d}{2}},\\
d=4&:&T^{-2}\log T,\\
d>4&:&T^{-2}.
\end{array}
\right.
\end{equation}
Likewise,
\begin{equation}\label{eq:syst-err-sto-cont22}
|A_{T}-A_\ho|\,\lesssim\, 
\left|
\begin{array}{rcl}
2\le d<4&:&T^{-\frac{d}{2}},\\
d=4&:&T^{-2}\log T,\\
d>4&:&T^{-2}.
\end{array}
\right.
\end{equation}
\qed
\end{theo}
\noindent These results are less precise for $d\ge 4$
 than in the case of symmetric coefficients considered in Theorem~\ref{th:quant-estim-Poisson}. We could push the arguments used in \cite{Gloria-Otto-10b} to obtain
the corresponding results for non-symmetric coefficients.
However, in view of Remark~\ref{rk:optimality-k=1}, Theorem~\ref{th:quant-estim-Poisson2} is sufficient for our purposes since higher-order Richardson extrapolations do not reduce further the error in general for coefficients satisfying \eqref{eq:Poinca-vert} in the dimensions $d=2,3$ of practical interest.

%%%%%%%%%%%%%%%%%%

\section{Numerical approximation of homogenized coefficients and correctors} \label{sec:tests}

The aim of this section is to introduce numerical approximations of correctors and homogenized coefficients based on the regularization and extrapolation method introduced in Section~\ref{sec:spectral}. In order to make this method of any practical use one has to make the different quantities at stake computable and control the approximation errors: this is the objective of the first subsection.
The last two subsections are dedicated to a systematic numerical study of the method, which illustrates both the interest of the approach and the sharpness of the analysis.
More precisely, Subsection~\ref{sec:test1} displays results of numerical tests for symmetric periodic and almost periodic coefficients, whereas Subsection~\ref{sec:test2} treats non-symmetric periodic and almost periodic coefficients.
For random coefficients, we refer the reader to the numerical study of the discrete elliptic equations with random conductivities in \cite{Gloria-10} and \cite{EGMN-12}.

\subsection{From abstract to computable approximations}
As recalled in the introduction, the motivation to use the regularization approach is the observation that the solution $\phi_T \in H^1_\loc(\R^d)$ (which exists for any $A\in L^\infty(\R^d,\Mab)$ and is unique in the class of functions that satisfy \eqref{eq:class-1}) of 
$$
T^{-1}\phi_{T}-\nabla \cdot A(\xi+\nabla \phi_T)\,=\,0 \mbox{ in }\R^d
$$
is much easier to approximate on bounded domains than the solution $\phi \in H^1_\loc(\R^d)$ (whose existence is only known to hold under structure assumptions, e.~g. almost-sure existence in the stationary ergodic case) of
$$
-\nabla \cdot A(\xi+\nabla \phi)\,=\,0 \mbox{ in }\R^d.
$$
This fact relies on the exponential decay of the Green function associated with the operator $T^{-1}-\nabla \cdot A \nabla$.

\medskip
\noindent In particular, this observation takes the following general form which holds for \emph{any} coefficient field $A$.
\begin{theo}\label{th:approx-TR}
Let $A\in L^\infty(\R^d,\Mab)$, $\xi \in \R^d$ be a unit vector, and $\phi_{T}$ be the associated regularized corrector 
of Definition~\ref{def:extra}. For all $R$, we let $\phi_{T,R}\in H^1_0(Q_R)$ be the unique weak solution of: for all $\psi\in H^1_0(Q_R)$,
\begin{equation}\label{eq:equation-for-phiTR}
\int_{Q_R}T^{-1} \psi\phi_{T,R}+\nabla \psi\cdot A(\xi+\nabla \phi_{T,R})\,=\,0.
\end{equation}
Then there exists $c>0$ depending only on $\alpha,\beta$ and $d$ such that for all $0<L\leq R$
with $R\sim R-L\gtrsim \sqrt{T}$, we have 
\begin{equation}\label{eq:th-approx-TR}
\int_{Q_L} |\nabla \phi_{T,R}-\nabla \phi_T|^2 \, \lesssim \, R^d T \exp\big(-c \frac{R-L}{\sqrt{T}}\big).
\end{equation}
\qed
\end{theo}
\begin{rem}\label{rem:upgrade-per}
In the case of periodic coefficients, we proved in \cite[Theorem~3.1]{Gloria-09} that \eqref{eq:th-approx-TR} can be upgraded to 
\begin{equation}\label{eq:th-approx-TR-per}
\int_{Q_L} |\nabla \phi_{T,R}-\nabla \phi_T|^2 \, \lesssim \, R^d \sqrt{T} \exp\big(-c \frac{R-L}{\sqrt{T}}\big).
\end{equation}
The same estimate holds for almost periodic coefficients of the Kozlov class.
\qed
\end{rem}
\begin{rem}\label{rem:approx-TRk}
By definition of Richardson extrapolation, in the context of Theorem~\ref{th:approx-TR} and Definition~\ref{def:extra}, we have for all $k\in \N$:
\begin{equation}\label{eq:th-approx-TRk}
\int_{Q_L} |\nabla \phi_{T,k,R}-\nabla \phi_{T,k}|^2 \, \lesssim \, R^d {2^{k-1}T} \exp\big(-c \frac{R-L}{\sqrt{2^{k-1}T}}\big).
\end{equation}
In the case of periodic coefficients or almost periodic coefficients of the Kozlov class, this estimate can be upgraded, as in Remark~\ref{rem:upgrade-per}, to
\begin{equation}\label{eq:th-approx-TRk-per}
\int_{Q_L} |\nabla \phi_{T,k,R}-\nabla \phi_{T,k}|^2 \, \lesssim \, R^d \sqrt{2^{k-1}T} \exp\big(-c \frac{R-L}{\sqrt{2^{k-1}T}}\big).
\end{equation}
Combined with Theorem~\ref{th:quant-estim-Perio} for periodic coefficients and with Theorem~\ref{th:quant-estim-quasiper} for almost periodic coefficients, this yields in particular:
\begin{equation}\label{eq:th-approx-TRk-per-full}
\int_{Q_L} |\nabla \phi_{T,k,R}-\nabla \phi|^2 \, \lesssim \, L^dT^{-2k}+R^d \sqrt{2^{k-1}T} \exp\big(-c \frac{R-L}{\sqrt{2^{k-1}T}}\big).
\end{equation}
\qed
\end{rem}
\noindent Before we turn to proof of Theorem~\ref{th:approx-TR}, let us make some comments on the approximation of homogenized coefficients. Unlike correctors, homogenized coefficients are averaged quantities.
In particular, in order to make the approximations $A_{T,k}$ of Definition~\ref{def:extra-hom} computable, one needs to approximate both the extrapolation $\phi_{T,k},\phi_{T,k}'$ of the regularized correctors and the averaging operator $\calM$ itself.
Since we have already addressed the approximation of the corrector in Theorem~\ref{th:approx-TR}, it only remains to approximate the averaging operator on domains $Q_L$. As recalled in the introduction, a first possibility is to replace $\calM$ by the average on $Q_L$. This yields however a very slow convergence for correlated fields, as the periodic example shows:
If $\mathcal E$ is a periodic (non constant) integrable function, then
\begin{equation}\label{eq:simple-average}
|\calM(\mathcal E)-\frac{1}{|Q_L|}\int_{Q_L}\mathcal E|\sim \frac{1}{L}.
\end{equation}
In order to enhance the convergence rate, we have introduced in \cite{Gloria-09} a filtered average, inspired
by the work by Blanc and Le Bris in \cite{Blanc-LeBris-09}.
\begin{defi}\label{def:mask}
A function $\mu:[-1,1]\to \R^+$ is said to be a filter of order $p \geq 0$ if
$\mu$ is continuous, even, non-increasing on $[0,1]$, and satisfies
\begin{itemize}
\item[(i)] $\mu \in C^{p-1}([-1,1])\cap W^{p,\infty}((-1,1))$,
\item[(ii)] $\int_{-1}^1 \mu(x)dx=1$,
\item[(iii)] $\mu^{(k)}(-1)=\mu^{(k)}(1)=0$ for all $k\in \{0,\dots,p-1\}$, but not for $k=p$.
\end{itemize}
\qed
\end{defi}
\begin{rem}
For $p=0$, the conditions (i) and (iii) are empty. The average on $[-1,1]$ (as used in \eqref{eq:simple-average} in the multidimensional version) is a filter of order $0$.
\qed
\end{rem}
\noindent Then, we may replace the average on $Q_L$ by the filtered average with filter $\mu_L:Q_L\to \R^+$
given by
\begin{equation*}
\mu_L(x)\,:=\,L^{-d}\prod_{i=1}^d\mu(L^{-1}x_i),
\end{equation*}
where $x=(x_1,\dots,x_d)\in \R^d$.
This yields the following formula for the computable approximation of $A_{T,k}$ in directions $\xi,\xi'\in \R^d$ and on
a box $Q_R$ with average on $Q_L$:
\begin{equation}\label{eq:approx-ATKRLp}
\xi' \cdot A'_{T,k,R,L,p}\xi\,:=\,\int_{Q_L} (\xi'+\nabla {\phi}_{T,k,R}'(x))\cdot A(x) (\xi+\nabla \phi_{T,k,R}(x))\mu_L(x)dx.
\end{equation}
We shall also use a variant of this definition, for which we can show the a priori uniform ellipticity for symmetric coefficients (see in particular Step~1 of the proof of Theorem~\ref{th:main-reg} below):
\begin{multline}\label{eq:approx-ATKRLp-v}
\xi' \cdot A_{T,k,R,L,p}\xi\,:=\,\int_{Q_L} \Big(\xi'+\nabla {\phi}_{T,k,R}'(x)-\expec{\nabla {\phi}_{T,k,R}'}_{\mu_L}\Big)
\\
\cdot A(x) \Big(\xi+\nabla \phi_{T,k,R}(x)-\expec{\nabla \phi_{T,k,R}}_{\mu_L}\Big)\mu_L(x)dx,
\end{multline}
where for all $\psi \in L^1(Q_L)$, 
$$
\expec{\psi}_{\mu_L}\,:=\,\int_{Q_L} \psi(x)\mu_L(x)dx.
$$
The following general convergence result holds:
\begin{theo}\label{th:ATkRLp-approx}
Let $d\geq 2$, $A\in L^\infty(\Omega,\Mab)$ be stationary ergodic coefficients, $\xi,\xi'\in \R^d$ be unit vectors, and $\phi_{T,k},\phi_{T,k}'$
be the associated Richardson extrapolations of the regularized corrector and dual corrector for $T>0$ and $k\in \N$, cf. Definition~\ref{def:extra}.
Then for all $p\in \N_0$, 
$$
\lim_{L,R-L\uparrow \infty} \xi' \cdot A_{T,k,R,L,p}\xi\,=\,\lim_{L,R-L\uparrow \infty} \xi' \cdot A'_{T,k,R,L,p}\xi \,=\, \xi' \cdot A_{T,k} \xi,
$$
almost surely, where $A_{T,k,R,L,p}$ and $A'_{T,k,R,L,p}$ are as in \eqref{eq:approx-ATKRLp-v} and \eqref{eq:approx-ATKRLp}, and $A_{T,k}$ as in Definition~\ref{def:extra-hom}.
\qed
\end{theo}
\begin{proof}
We split the proof into two steps.

\medskip

\step{1} Proof that $A'_{T,k,R,L,p} \to A_{T,k}$.

\noindent By Remark~\ref{rem:approx-TRk}, it is enough to show that almost surely
$$
\lim_{L\uparrow \infty} \int_{Q_L} (\xi'+\nabla {\phi}'_{T,k}(x))\cdot A(x) (\xi+\nabla \phi_{T,k}(x))\mu_L(x)dx\,=\,\xi'\cdot A_{T,k} \xi.
$$
We shall prove the claim for any stationary function $\psi\in L^1_\loc(\R^d)$ that satisfies the uniform bound $\sup_{x\in \R^d} \int_{Q_\rho(x)} |\psi(x)|dx<\infty$ for some $\rho>0$ (recall that $\phi_{T,k}$ satisfies
\eqref{eq:class-1}).
For all $t\in \mu_1([0,1])$, set $\mu_1^{-1}(t):=\{x\in Q\,|\,\mu_1(x)\le t\}$ and $L\mu_1^{-1}(t):=\{Lx\,|\,x\in Q \text{ and }\mu_1(x)\le t\}$. By definition of $\mu_1$, this is a connected set. 
Integrating along sub-levelsets of $\mu_L$, we can then rewrite the $\mu_L$-average as
$$
\expec{\psi}_{\mu_L}\,=\,\int_0^{\mu_1(0)} L^{-d} \int_{L\mu_1^{-1}(t)} \psi(x)dxdt.
$$
For all $0<t<\mu(0)$,
the ergodic theorem then yields
$$
\lim_{L\uparrow \infty} L^{-d}\int_{L\mu_1^{-1}(t)} \psi(x)dx \,=\,|\mu_1^{-1}(t)|\calM(\psi)
$$
almost surely, where $|\mu_1^{-1}(t)|$ denotes the $1$-dimensional Lebesgue measure of $\mu_1^{-1}(t)$. Hence, the claim follows from the Lebesgue dominated convergence
theorem and the identity $\int_0^{\mu_1(0)}|\mu_1^{-1}(t)|dt=1$.

\medskip

\step{2} Proof that $A_{T,k,R,L,p} \to A_{T,k}$

\noindent In view of Step~1, it suffices to show that 
$$
\lim_{L,R-L\uparrow \infty} |A_{T,k,R,L,p}-A'_{T,k,R,L,p}|\,=\,0. 
$$
By definition of $A_{T,k,R,L,p}$ and $A'_{T,k,R,L,p}$, the claim follows if
$$
\lim_{L,R-L\uparrow \infty} \expec{\nabla \phi_{T,k,R}}_{\mu_L} \,=\,0,
$$
which, by Remark~\ref{rem:approx-TRk}
$$
\Big| \expec{\nabla \phi_{T,k,R}}_{\mu_L}-\expec{\nabla \phi_{T,k}}_{\mu_L}  \Big|
\,\lesssim \,
\sqrt{T} \big(\frac{R}{L}\big)^d \exp\big(-c\frac{R-L}{\sqrt{T}}\big),
$$
we may prove in the equivalent form of
\begin{equation}\label{eq:phiTmeanfree}
\lim_{L\uparrow \infty} \expec{\nabla \phi_{T,k}}_{\mu_L} \,=\,0.
\end{equation}
But this follows from Step~1 and the fact that $\nabla\phi_{T,k}$ has vanishing expectation:
\begin{equation}\label{eq:phiTmeanfree}
\lim_{L\uparrow \infty} \expec{\nabla \phi_{T,k}}_{\mu_L} \,=\,\expec{\nabla \phi_{T,k}}=0
\end{equation}
almost surely.
\end{proof}

\noindent In the case of periodic coefficients we may turn this qualitative convergence result quantitative:
\begin{theo}\label{th:error-estim-per}
Let $d\geq 2$, $A\in L^\infty(\R^d,\Mab)$ be periodic coefficients, $\mu$ be a filter of order $p\geq 0$, $k\in \N$, and $A_\ho$, $A_{T,k,R,L,p}$ and $A'_{T,k,R,L,p}$  be the
homogenized coefficients and their approximations \eqref{eq:approx-ATKRLp-v} and \eqref{eq:approx-ATKRLp}, respectively. Then in the regime $R^2 \gtrsim T \gtrsim R$, $R\geq L \sim R \sim R-L$,  there exists $c>0$ depending only on $\alpha,\beta$ and $d$ such that we have
\begin{equation}\label{eq:th-error-estim-per}
|A_{T,k,R,L,p}-A_\ho|,|A'_{T,k,R,L,p}-A_\ho|\,\lesssim\, 
L^{-(p+1)}+T^{-2k}+T^{\frac{1}{4}} \exp \left(-c \frac{R-L}{\sqrt{2^{k-1}T}}\right),
\end{equation}
where the multiplicative constant depends on $k$ and $p$ next to $\alpha,\beta,$ and $d$.
\qed
\end{theo}
\begin{proof}
The estimate for $|A'_{T,k,R,L,p}-A_\ho|$ is a direct consequence of the combination of Theorem~\ref{th:quant-estim-Perio}, Remark~\ref{rem:approx-TRk}, 
and \cite[Theorem~3.1]{Gloria-09}. 
It only remains to prove that the estimate for $|A_{T,k,R,L,p}-A_\ho|$ is a consequence of the estimate for 
$|A'_{T,k,R,L,p}-A_\ho|$.
By definition of $A'_{T,k,R,L,p}$ and $A_{T,k,R,L,p}$ and an elementary energy estimate, it is enough to prove that
in the desired regime of $R,L,T$,
$$
\big| \expec{\nabla \phi_{T,k,R}}_{\mu_L} \big| \,\lesssim \, L^{-(p+1)}+T^{\frac{1}{4}} \exp \left(-c \frac{R-L}{\sqrt{2^{k-1}T}}\right).
$$
By Remark~\ref{rem:approx-TRk}, this follows from 
$$
\big| \expec{\nabla \phi_{T,k}}_{\mu_L} \big| \,\lesssim \, L^{-(p+1)},
$$
which is itself a consequence of \cite[Lemma~3.1]{Gloria-09} since $\calM(\nabla \phi_{T,k})=0$.
\end{proof}

\medskip
\noindent A similar result holds for the Kozlov class of almost-periodic coefficients.
\begin{theo}\label{th:error-estim-Kozlov}
Let $d\geq 2$, $A\in L^\infty(\R^d,\Mab)$ be in the Kozlov class of almost periodic coefficients for some set of Fourier modes $\Gamma$. Let $\mu$ be a filter of order $p\geq 0$, $k\in \N$, and $A_\ho$, $A_{T,k,R,L,p}$ and $A'_{T,k,R,L,p}$  be the
homogenized coefficients and their approximations \eqref{eq:approx-ATKRLp-v} and \eqref{eq:approx-ATKRLp}, respectively. Then in the regime $R^2 \gtrsim T \gtrsim R$, $R\geq L \sim R \sim R-L$,  there exists $c>0$ depending only on $\alpha,\beta$ and $d$ such that we have
\begin{equation}\label{eq:th-error-estim-per}
|A_{T,k,R,L,p}-A_\ho|,|A'_{T,k,R,L,p}-A_\ho|\,\lesssim\, 
L^{-(p+1)}+T^{-2k}+T^{\frac{1}{4}} \exp \left(-c \frac{R-L}{\sqrt{2^{k-1}T}}\right),
\end{equation}
where the multiplicative constant depends on $k$ and $p$ next to $\alpha,\beta,\Gamma$ and $d$.
\qed
\end{theo}
\begin{proof}
The proof is similar to the proof for periodic coefficients. The only difference 
comes from the fact that we cannot directly appeal to  \cite[Lemma~3.1]{Gloria-09} to quantify
the convergence of averages.
The argument is however very similar. Since the corrector $\overline \phi_{T,k}$ is periodic on
the high-dimensional torus $\Pi_M$, it can be expanded in Fourier series:
$$
\overline\phi_{T,k}(y_1,\dots,y_M)\,=\,\sum_{l\in \Z^M} c_{k,T}(l)\exp(i\sum_{j=1}^M l_j\omega_jy_j),
$$
where $\Pi_M=\prod_{j=1}^M[0,2\pi \omega_j)$ (note that $c_{k,T}(0)=0$).
Since $\overline \phi_{T,k}$ is smooth uniformly w.~r.~t $T$, we have
\begin{equation}\label{eq:Fourier-exp}
\sum_{l\in \Z^M} |l||c_{k,T}(l)|\,\lesssim\, 1,
\end{equation}
where the multiplicative constant depends on $k$ but not on $T$. 
To make the diagonal restriction explicit we introduce $M$ functions $\nu_j:\R^d\to \R$
which map $x$ to the component $\nu_j(x)$ of $x$ which corresponds to the restriction of $y_j$, so 
that we have for $x\in \R^d$,
$$
\phi_{T,k}(x)=\overline \phi_{T,k}(\nu_1(x),\dots,\nu_M(x)).
$$
We may then turn the Fourier expansion of $\overline{\phi}_{T,k}$ into an expansion for $\phi_{T,k}$:
$$
\phi_{T,k}(x)\,=\,\sum_{l\in \Z^M} c_{k,T}(l)\exp(i\sum_{j=1}^M l_j\omega_j\nu_j(x)),
$$
which is absolutely convergent by \eqref{eq:Fourier-exp}.
Similar results hold for the dual regularized corrector.
We need both to average the energy and the gradient of the corrector and dual corrector.
In terms of Fourier expansion, we therefore have to control terms with coefficients $|l||c_{k,T}(l)|$ for 
the gradient of the correctors and terms with coefficients $|l|^2c_{k,T}^2(l)$ for the energy.
Since both series are summable on $\Z^M$ by \eqref{eq:Fourier-exp}, it is enough to prove that
$$
\sup_{l\in \Z^M}\left|\expec{\exp(i\sum_{j=1}^M l_j\omega_j\nu_j(\cdot))}_{\mu_L}-\calM\Big(\exp(i\sum_{j=1}^M l_j\omega_j\nu_j(\cdot))\Big)\right|
\,\lesssim \, L^{-(p+1)},
$$
which simply follows from $p$ integrations by parts, as in the proof of  \cite[Lemma~3.1]{Gloria-09}.
\end{proof}

\medskip
\noindent In the random setting one cannot expect to have an almost sure \emph{quantitative} convergence result.
The combination of Remark~\ref{rem:approx-TRk} and \cite[Theorem~1 and Remark~2.5]{Gloria-Otto-10b} yields:
\begin{theo}\label{th:error-estim-Poisson}
Let $A\in L^\infty(\Omega,\Mab)$ be stationary random coefficients
satisfying~\eqref{eq:Poinca-vert}, let  $\mu$ be a filter of order $p\geq 1$, and $A_\ho$ and $A'_{T,1,R,L,p}$ be the
homogenized matrix and its approximation \eqref{eq:homTRL} respectively, where
$R^2 \gtrsim T \gtrsim R$ and $R\geq L \sim R \sim R-L$.
Then, there exists $c>0$ and $q>0$ depending only on $\alpha,\beta$ and $d$ such that we have
\begin{multline}\label{eq:th-error-estim-poisson}
\expec{|A'_{T,1,R,L,p}-A_\ho|^2}^{\frac{1}{2}} \,\lesssim\, \sqrt{T} \exp \left(-c \frac{R-L}{\sqrt{T}}\right) 
+ L^{-\frac{d}{2}}\left|
\begin{array}{rcl}
d=2&:& \log L\\
d>2&:&1
\end{array}
\right.\\
+\left|
\begin{array}{rcl}
d=2&:& T^{-1},\\
d=3&:&T^{-\frac{3}{2}},\\
d=4&:&T^{-2}\log T,\\
d>4&:&T^{-2},
\end{array}
\right.
\end{multline}
where the multiplicative constant depends on $p$, next to $\alpha,\beta,$ and $d$.
\qed
\end{theo}
\begin{rem}
Since $k=1$, we drop the subscript $k$ in this remark.
The triangle inequality combined with energy estimates yields for all $1\leq i \leq d$
\begin{multline*}
\expec{(\xi'\cdot(A_{T,R,L,p}-A_\ho)\xi)^2}^{\frac{1}{2}} \lesssim \,\expec{(\xi'\cdot(A'_{T,R,L,p}-A_\ho)\xi)^2}^{\frac{1}{2}} 
\\+\expec{|\expec{\nabla \phi_{T,R}}_{\mu_L}|^2+|\expec{\nabla \phi_{T,R}'}_{\mu_L}|^2+|\expec{\nabla \phi_{T,R}}_{\mu_L}|^4+|\expec{\nabla \phi_{T,R}'}_{\mu_L}|^4}^{\frac{1}{2}} .
\end{multline*}
We only address the corrector (the estimates for the dual corrector are similar).
On the one hand, by the energy estimate $\int_{Q_R} |\nabla \phi_{T,R}|\lesssim R^d$, one has
$$
|\expec{\nabla \phi_{T,R}}_{\mu_L}|^4\,\lesssim \, |\expec{\nabla \phi_{T,R}}_{\mu_L}|^2 \big(\frac{R}{L}\big)^2.
$$
On the other hand, by Remark~\ref{rem:approx-TRk},
$$
\Big|\expec{\nabla \phi_{T,R}-\nabla \phi_{T}}_{\mu_L}\Big|\,\lesssim\,\sqrt{T} \exp \left(-c \frac{R-L}{\sqrt{T}}\right),
$$
and by a simplified version of the string of arguments that leads to \cite[Theorem~1]{Gloria-Otto-10b},
we can prove that for all $1\leq j\leq d$
$$
\var{\xi'\cdot \int_{Q_L} \nabla \phi_{T} \mu_L}\,=\,\var{\xi'\cdot \int_{Q_L} \phi_{T} \nabla \mu_L}\,\lesssim \,
L^{-d}\left|
\begin{array}{rcl}
d=2&:& \log L\\
d>2&:&1
\end{array}
\right.
$$
in the desired regime of $L$ and $T$.
Hence, the same estimate as \eqref{eq:th-error-estim-poisson} holds for $A_{T,R,L,p}$.
\qed
\end{rem}
\begin{rem}
Note that the convergence rate does not depend on the order of the filter provided it is at least of order $1$ (the only fact used in the proof is that 
$\mu$ is smooth and vanishes at $-1$ and $1$). This is due to the fact that the result is not almost-sure.
\qed
\end{rem}
\medskip

\noindent We conclude this subsection with the proof of Theorem~\ref{th:approx-TR}, which is a generalization 
of the corresponding result for symmetric periodic coefficients, see \cite{Gloria-12b}, to any coefficients $A\in L^\infty(\R^d,\Mab)$.
\begin{proof}[Proof of Theorem~\ref{th:approx-TR}]
The argument relies on the exponential decay of the Green function, Caccioppoli's inequality,
and the uniform a priori estimate \eqref{eq:class-2}.

\medskip

\noindent For the main properties of the massive Green function and their proofs, we refer the reader to \cite[Definition~2.4]{Gloria-Otto-10b}.
In particular, for all $A\in L^\infty(\R^d,\Mab)$ and all $T>0$, the Green function $G_T(\cdot,y)$
is for all $y\in \R^d$ the unique distributional solution in $W^{1,1}(\R^d)$, continuous on $\R^d\setminus\{y\}$, of the equation
$$
T^{-1}G_T(x,y)-\nabla_x \cdot A(x)\nabla_xG_T(x,y)\,=\,\delta(x-y).
$$
The following pointwise bound holds uniformly wrt $A$: For all $|x-y|\gtrsim \sqrt{T}$,
\begin{equation}\label{eq:Green-T}
0\leq G_T(x,y) \, \lesssim \, \frac{1}{|x-y|^{d-2}}\exp\left(-c\frac{|x-y|}{\sqrt{T}}\right).
\end{equation}

\medskip
\noindent Arguing by density in \eqref{eq:th-approx-TR}, one may assume by a standard regularization argument that $A$ is smooth, so that $\phi_T$ and $\phi_{T,R}$ are smooth on $Q_R$ by elliptic regularity.
By assumption there exists $\gamma \sim 1$ such that $R-L\geq 2\gamma \sqrt{T}$.
By definition of $\phi_T$ and $\phi_{T,R}$, we have
\begin{equation*}
\left\{
\begin{array}{rcll}
T^{-1}(\phi_T-\phi_{T,R})-\nabla\cdot A(\nabla \phi_T-\nabla \phi_{T,R})&=&0&\text{ in } Q_R ,\\
\phi_T-\phi_{T,R}&=&\phi_T & \text{ on }\partial Q_R.
\end{array}
\right.
\end{equation*}
Set $\phi_1=\chi \phi_T$ where $\chi\in C^\infty(\overline Q_R,\R^+)$ is such that $\chi|_{\partial Q_R}=1$,
$\chi|_{Q_{R-\gamma\sqrt{T}}}=0$, and $|\nabla \chi|\lesssim \sqrt{T}^{-1}$.
Since $\phi_T$ satisfies \eqref{eq:class-2} and $\phi_1$ vanishes on $Q_{R-\gamma\sqrt{T}}$, we have
\begin{equation}\label{eq:bound-phi1}
\|\phi_1\|_{L^2(Q_R)}^2 \,\lesssim \, (R^{d-1} \sqrt{T}) T, \quad \|\nabla \phi_1\|_{L^2(Q_R)}^2
=\|\chi \nabla \phi_T+\phi_T\nabla \chi\|_{L^2(Q_R)}^2\,\lesssim \,R^{d-1} \sqrt{T}.
\end{equation}
The function 
\begin{equation}\label{eq:def-phi2}
\phi_2:=\phi_T-\phi_{T,R}-\phi_1
\end{equation}
is smooth and satisfies the equation
\begin{equation}\label{eq:phi2}
\left\{
\begin{array}{rcll}
T^{-1}\phi_2-\nabla\cdot A\nabla \phi_2&=&-T^{-1}\phi_1+\nabla\cdot A\nabla \phi_1&\text{ in } Q_R, \\
\phi_2&=&0 &\text{ on }\partial Q_R.
\end{array}
\right.
\end{equation}
Let $G_{T,R}:Q_R\times Q_R \to \R^+$ be the Green function associated with the operator $(T^{-1}-\nabla \cdot A \nabla)$
on $Q_R$ with homogeneous Dirichlet boundary conditions. Since the RHS of \eqref{eq:phi2} is smooth, the Green representation formula yields
\begin{equation*}
\phi_2(x)\,=\,\int_{Q_R}(-T^{-1}\phi_1(y)G_{T,R}(x,y)+\nabla G_{T,R}(x,y)\cdot A(y)\nabla \phi_1(y))dy.
\end{equation*}
By Cauchy-Schwarz' inequality, this turns for all $x\in Q_L \subset Q_{R-2\gamma\sqrt{T}}$ into 
\begin{eqnarray}
|\phi_2(x)|&\lesssim & T^{-1} \|\phi_1\|_{L^2(Q_R)} \left(\int_{Q_R\setminus Q_{R-\gamma \sqrt{T}}}G_{T,R}(x,y)^2dy\right)^{\frac{1}{2}}\nonumber
\\
&& \qquad\qquad\qquad +\|\nabla \phi_1\|_{L^2(Q_R)} \left(\int_{Q_R\setminus Q_{R-\gamma \sqrt{T}}}|\nabla G_{T,R}(x,y)|^2dy\right)^{\frac{1}{2}}.\label{eq:avt-caccio}
\end{eqnarray}
To control the first RHS term of \eqref{eq:avt-caccio}, we appeal to \eqref{eq:Green-T} and to the maximum principle in the form of $0\leq G_{T,R} \leq G_{T}$.
For the second RHS term of \eqref{eq:avt-caccio}, we use in addition Caccioppoli's inequality.
Let $\eta:Q_R\to \R^+$ be such that $\eta|_{Q_{R}\setminus Q_{R-\gamma \sqrt{T}}}=1$, $\eta|_{Q_{R-\frac{3}{2}\gamma \sqrt{T}}}=0$, and $|\nabla \eta|\lesssim \sqrt{T}^{-1}$, and test the equation for $G_{T,R}(x,\cdot)$ with $\eta^2 G_{T,R}(x,\cdot)\in H^1_0(Q_R)$.
This yields after integration by parts 
\begin{eqnarray*}
0&=&T^{-1}\int_{Q_R}\eta^2(y)G_{T,R}(x,y)^2dy+\int_{Q_R}\nabla G_{T,R}(x,y)\cdot A(y) \nabla (\eta(y)^2G_{T,R}(x,y))dy\\
&=&T^{-1}\int_{Q_R}\eta^2(y)G_{T,R}(x,y)^2dy+\int_{Q_R}\nabla (\eta(y) G_{T,R}(x,y))\cdot A(y) \nabla (\eta(y) G_{T,R}(x,y))dy
\\
&& \qquad \qquad\qquad \qquad -\int_{Q_R}G_{T,R}(x,y)^2\nabla \eta(y)\cdot A(y)\nabla \eta(y)dy.
\end{eqnarray*}
Hence, by uniform ellipticity of $A$,
\begin{equation*}
\int_{Q_R}|\nabla (\eta(y) G_{T,R}(x,y))|^2 dy
\,\lesssim \,
\int_{Q_R}G_{T,R}(x,y)^2|\nabla \eta(y)|^2dy,
\end{equation*}
that is,
\begin{equation}\label{eq:caccio}
\int_{Q_R\setminus Q_{R-\gamma \sqrt{T}}}|\nabla G_{T,R}(x,y)|^2 dy
\,\lesssim \,
T^{-1} \int_{Q_R\setminus Q_{R-\frac{3}{2}\gamma \sqrt{T}}} G_{T,R}(x,y)^2dy.
\end{equation}
We insert \eqref{eq:caccio} into \eqref{eq:avt-caccio}, appeal to \eqref{eq:Green-T} and \eqref{eq:bound-phi1}, and use that for all $x\in Q_{L+1}$ and $y\in Q_R\setminus Q_{R-\frac{3}{2}\gamma \sqrt{T}}$, 
we have $|x-y|\gtrsim R-L$, so that
\begin{eqnarray}
|\phi_2(x)|&\lesssim & (T^{-1}\|\phi_1\|_{L^2(Q_R)} + \sqrt{T}^{-1}\|\nabla \phi_1\|_{L^2(Q_R)})
\left(\int_{Q_R\setminus Q_{R-\frac{3}{2}\gamma \sqrt{T}}}G_{T,R}(x,y)^2dy\right)^{\frac{1}{2}}\nonumber
\\
&\stackrel{\eqref{eq:bound-phi1}}{\lesssim} & (R^{d-1}\sqrt{T}) \sqrt{T}^{-1} \frac{1}{(R-L)^{d-2}}\exp\big( -c\frac{R-L}{\sqrt{T}}\big)\nonumber
\\
&=& \sqrt{T} \frac{R^{d-1}}{(R-L)^{d-1}}\frac{R-L}{\sqrt{T}} \exp\big( -c\frac{R-L}{\sqrt{T}}\big)\nonumber
\\
&\stackrel{R-L\sim R}{\lesssim}& \sqrt{T} \exp\big( -c\frac{R-L}{\sqrt{T}}\big),
\end{eqnarray}
where $c$ is a positive constant which may change from line to line but does not depend on $T$, $L$, and $R$.
Hence,
\begin{eqnarray*}
\int_{Q_{L+1}}\phi_2(x)^2dx &\lesssim &  R^{d}  {T}\exp\big(-c\frac{R-L}{\sqrt{T}}\big)  .
\end{eqnarray*}
Another use of Caccioppoli's inequality, this time for $\phi_2$ (recall that the RHS of \eqref{eq:phi2} 
vanishes identically in $Q_{R-\gamma \sqrt{T}}$), yields by definition \eqref{eq:def-phi2} of $\phi_2$
and since $\phi_1$ vanishes on $Q_L$:
\begin{multline*}\label{eq:estim-grad-phi2}
\int_{Q_{L}}|\nabla (\phi_T-\phi_{T,R})(x)|^2dx\,=\,\int_{Q_{L}}|\nabla \phi_2(x)|^2dx \,\lesssim \, \int_{Q_{L+1}}\phi_2(x)^2dx\\
\,\lesssim\, R^d T \exp\big(-c\frac{R-L}{\sqrt{T}}\big)
\end{multline*}
as desired.
\end{proof}

\subsection{Numerical tests for symmetric coefficients}\label{sec:test1}

We display the results of three series of numerical tests.
In the first paragraph we treat the case of a discrete elliptic equation on $\Z^d$ with periodic coefficients. The advantage of this example is that there is no additional approximation error besides the machine precision since the problem is already discrete. This allows us to explore the behavior of the method for large $R$, which is not possible for a continuum equation due to necessary discretization of the periodic cell (the total number of degrees of freedom becomes rapidly very large).
We then address two examples of a continuum equation: periodic and almost periodic coefficients.
As expected from the analysis point of view, what we observe empirically for discrete elliptic equations is representative of what is observed for continuum elliptic equations.
The conclusion is as follows: Whereas the first error term $T^{-2k}$ of the RHS of \eqref{eq:th-approx-TRk-per-full} in Remark~\ref{rem:approx-TRk} seems to be dominant in the regime of moderate $R$,  the RHS of \eqref{eq:th-error-estim-per} in Theorem~\ref{th:error-estim-per} is dominated for moderate $R$ by the error term $L^{-(p+1)}$ due to the averaging process, regardless of the error-term in $T^{-2k}$. 
This tends to show that numerical homogenization methods which only rely on  approximations of the corrector (and not on the approximation of homogenized coefficients) may be likely to perform better in terms of the resonance error (at least when the ratio $R=\frac{\rho}{\e}$ is only moderately large).

\medskip

\noindent 
For numerical tests on random coefficients, we refer the reader to \cite{Gloria-10,EGMN-12}, where a systematic study of numerical methods for discrete linear elliptic equations with random coefficients is conducted. These tests confirm the sharpness of (the discrete version of) Theorem~\ref{th:error-estim-Kozlov} (as well as other quantitative results)
and the superiority of our method over the naive approach (for essentially the same computational cost).

\medskip

\noindent Before we turn to the tests proper, let us make precise the filters we use.
The functions $\mu^1,\dots,\mu^4,\mu^\infty$ are filters of order $1,\dots,4,\infty$ on the
interval $[0,1]$, the constants $\kappa_1,\dots,\kappa_4,\kappa_\infty$ are such that the total mass is $1$ on $[0,1]$. The filters on $[-1,1]$ are then obtained by translation and dilation.
\begin{equation*}
\mu^1(t)\,=\kappa_1\,\left\{
\begin{array}{rcl}
t\le \frac{1}{3}&:&0,\\
\frac{1}{3}<t\le \frac{4}{9}&:&3(3t-1) ,\\
\frac{4}{9}<t\le \frac{5}{9}&:&1,\\
\frac{5}{9}<t\le \frac{2}{3}&:&3(2-3t) ,\\ 
\frac{2}{3}<t&:&0.
\end{array}
\right.
\end{equation*}
\begin{equation*}
\mu^2(t)\,=\kappa_2\,\left\{
\begin{array}{rcl}
t\le \frac{1}{3}&:&0,\\
\frac{1}{3}<t\le \frac{4}{9}&:&(3t-1)^2,\\
\frac{4}{9}<t\le \frac{5}{9}&:&\frac{1}{6}-18(t-\frac{1}{2})^2,\\
\frac{5}{9}<t\le \frac{2}{3}&:&(3t-2)^2,\\ 
\frac{2}{3}<t&:&0.
\end{array}
\right.
\end{equation*}
\begin{equation*}
\mu^3(t)\,=\kappa_3\,\left\{
\begin{array}{rcl}
t\le \frac{1}{3}&:&0\\
\frac{1}{3}<t\le \frac{4}{9}&:&(3t-1)^3,\\
\frac{4}{9}<t\le \frac{5}{9}&:&\frac{19}{27}+1359(t-\frac{1}{2})^2+1458(t-\frac{1}{2})^4,\\
\frac{5}{9}<t\le \frac{2}{3}&:&(2-3t)^3 ,\\ 
\frac{2}{3}<t&:&0.
\end{array}
\right.
\end{equation*}
\begin{equation*}
\mu^4(t)\,=\kappa_4\,\left\{
\begin{array}{rcl}
t\le \frac{1}{3}&:&0\\
\frac{1}{3}<t\le \frac{4}{9}&:&(3t-1)^4,\\
\frac{4}{9}<t\le \frac{5}{9}&:&\frac{1}{27}-\frac{27}{2}(t-\frac{1}{2})^2+2268(t-\frac{1}{2})^4-157464 (t-\frac{1}{2})^6,\\
\frac{5}{9}<t\le \frac{2}{3}&:&(3t-2)^4,\\ 
\frac{2}{3}<t&:&0.
\end{array}
\right.
\end{equation*}
\begin{equation*}
\mu^\infty(t)\,=\kappa_\infty\,\left\{
\begin{array}{rcl}
t\le \frac{1}{3}&:&0\\
\frac{1}{3}<t\le \frac{2}{3}&:&\exp(-\frac{1}{(t-\frac{1}{3})(\frac{2}{3}-t)}),\\
\frac{2}{3}<t&:&0.
\end{array}
\right.
\end{equation*}

\subsubsection{Warm-up example: Discrete periodic coefficients}

Consider the discrete elliptic equation
\begin{equation}\label{eq:disc-per}
-\nabla^*\cdot A(\xi+\nabla \phi)=0 \qquad \text{ in }\Z^2,
\end{equation}
where for all $u:\Z^2\to \R$,
\begin{equation*}
\nabla u(x):=\left[  
\begin{array}{l}
u(x+\ee_1)-u(x) \\
u(x+\ee_2)-u(x)
\end{array}
\right],
\
\nabla^* u(x):=\left[  
\begin{array}{l}
u(x)-u(x-\ee_1) \\
u(x)-u(x-\ee_2)
\end{array}
\right],
\end{equation*}
and
\begin{equation}\label{eq:A-disc}
A(x):=\dig{a(x,x+\ee_1),a(x,x+\ee_2)}.
\end{equation}
The matrix $A$ is $[0,4)^2$-periodic, and sketched on a periodic cell on Figure~\ref{fig:discret}. 
In the example considered, $a(x,x+\ee_1)$ and $a(x,x+\ee_2)$ represent 
the conductivities $1$ (light grey) or $100$ (black) of the horizontal edge $[x,x+\ee_1]$ and the vertical edge $[x,x+\ee_2]$ respectively,
cf. Figure~\ref{fig:discret}.
The homogenization theory for such discrete elliptic operators is similar to the one of the continuum setting (see for instance \cite{Vogelius-91} in this two-dimensional case). 
Likewise, the analogue of Theorem~\ref{th:error-estim-per} holds, and our aim is to investigate the
behavior of the error in the regime of $R$ large.
By rotation invariance the homogenized matrix $A_\ho$ is a multiple of the identity.
It can be evaluated numerically (note that we do not make any other error than the machine precision). Its numerical value
is $A_\ho=26.240099009901\dots$.
\begin{figure}
\centering
\includegraphics[scale=.5]{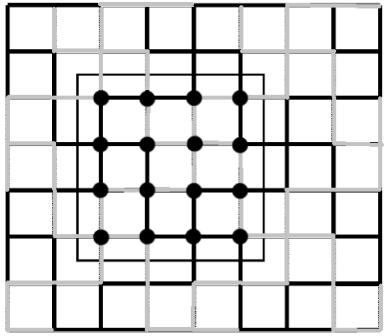}
\caption{Periodic cell in the discrete case}
\label{fig:discret}
\end{figure}

\medskip
\noindent We have considered the first two approximations formulas $A'_{T,1,R,L,\infty}$ and $A'_{T,2,R,L,\infty}$ of $A_\ho$, cf. \eqref{eq:approx-ATKRLp}, 
as well as the naive approximation $A'_{\infty,1,R,L,\infty}$ (where $\phi_{\infty,R}\in H^1_0(Q_R)$ is the solution of \eqref{eq:equation-for-phiTR} for $T=\infty$, that is, without zero-order term).
In all the tests, for $A'_{T,1,R,L,p}$ and $A'_{T,2,R,L,p}$ we have taken the following parameters:
\begin{itemize}
\item zero-order term $T= \frac{R}{10}$,
\item filter of order $p=\infty$ with $L=\frac{R}{3}$.
\end{itemize}
This yields the following theoretical predictions according to (the discrete counterpart of) Theorem~\ref{th:error-estim-per}: 
$$
|A_\ho-A'_{\infty,1,R,L,\infty}|\,\lesssim\, R^{-1}, \quad |A_{\ho}-{A}'_{T,1,R,L,\infty}|\,\lesssim\, R^{-2},\quad |A_{\ho}-{A}'_{T,2,R,L,\infty}|\,\lesssim\,R^{-4}.
$$
The numerical tests have been performed up to $R=100$ periodic cells per dimension (that is $400$ points per dimension, $1.6\,10^5$ degrees of freedom).
They confirm the theoretical predictions, as can be seen of Figure~\ref{fig1}.
Yet, for $k=1$ and even more for $k=2$, the theoretical predictions are only attained for very large $R$. 
Indeed, recall that Theorem~\ref{th:error-estim-per} provides more details on the error: for all $p\in \N_0$,
$$
|A_{\ho}-{A}'_{T,k,R,L,\infty}|\,\leq \, C_pR^{-(p+1)}+C_k R^{-2k}+C_kR^{\frac{1}{4}} \exp(-c\sqrt{2^{-k+1}R}),
$$
where $C_p \uparrow \infty$ as $p\uparrow \infty$.
In the tests, the last error-term is negligible and only the first two error-terms are observed. 
In particular, the error due to averaging is super-algebraic for large $R$ but may be dominant for moderate $R$. This is indeed what we observe on Figure~\ref{fig1}.
For $A_{\infty,1,R,L,\infty}$ the error due to the approximation of the correctors dominate for all $R$ so that the error curve is a straight line of slope $-1$ (this is the naive approach combined with averaging). For $A_{T,1,R,L,\infty}$ and $A_{T,2,R,L,\infty}$  however, for small $R$, both errors coincide and seem to decay super-algebraically --- this is the averaging error --- until they cross a straight line of slopes $-2$ and $-4$, respectively, when the systematic error ``$R^{-2k}$" starts dominating.
This illustrates that, due to averaging, the computational approximation of the homogenized coefficients can be effectively much slower 
than the computational approximation of the correctors for moderate $R$ (which is the regime of interest), although the asymptotic error (for large $R$) is the same.

\medskip

\noindent Last, let us comment on the computational cost of the methods.
The cost is proportional to the extrapolation order $k$ of the method since it requires the approximations of $k$ correctors when the homogenized matrix is a multiple of the identity. Hence the computational cost for the approximation of $A_{T,k,R,L,p}$ is exactly $k$ times 
the computational cost of the naive approximation $A_{\infty,1,R,L,0}$. Likewise for the correctors $\phi_{T,k,R}$ and $\phi_{\infty,1,R}$.
\begin{figure}
\centering
\psfrag{hh}{$\small \log_{10}(R)$}
\psfrag{kk}{$\small \log_{10} \mathcal E (R)$}
\includegraphics[scale=.6]{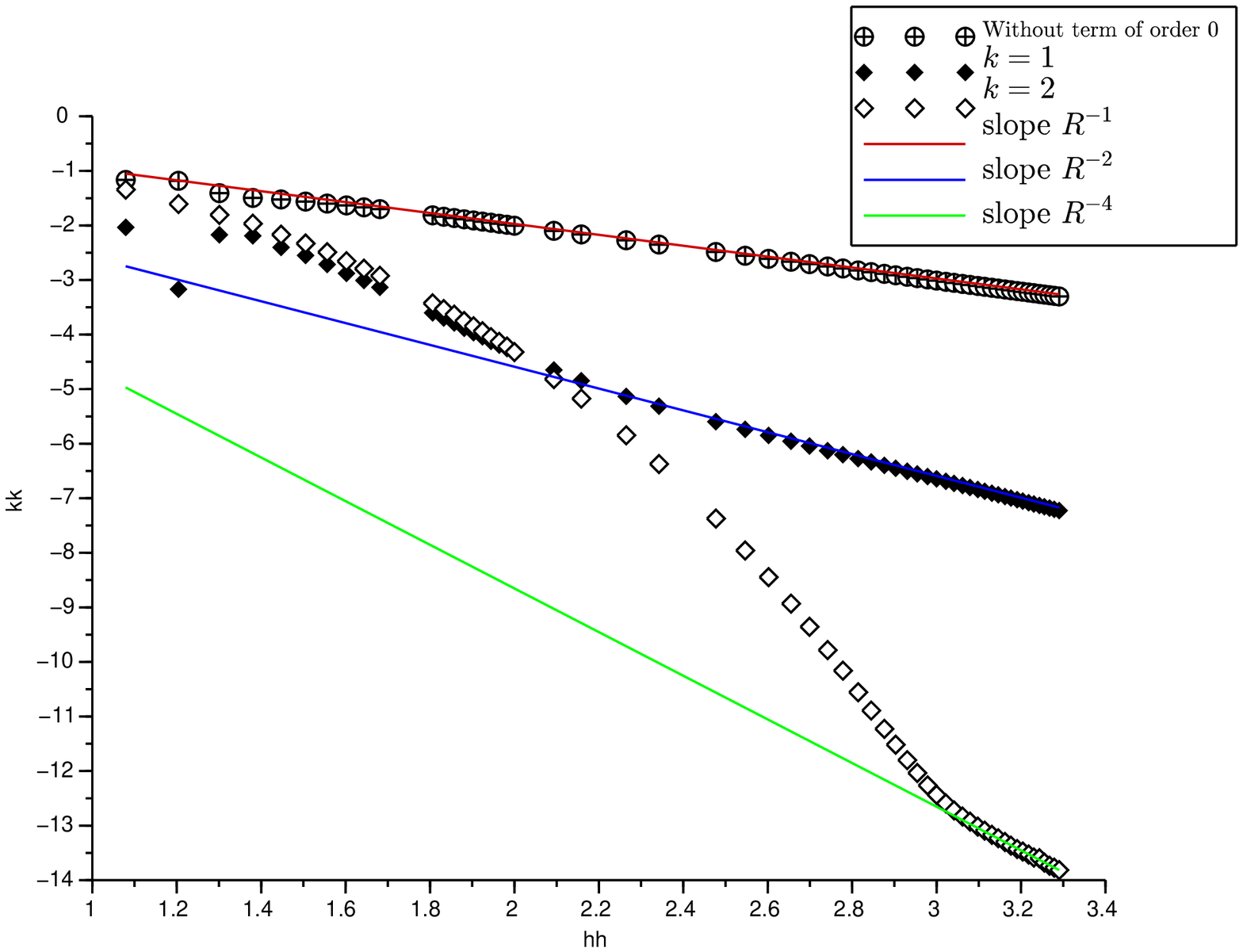}
\caption{Symmetric discrete periodic example \eqref{eq:A-disc}. Bound on $R\mapsto \mathcal{E}(R):=|A_{T,k,R,L,\infty}-A_\ho|$ for $(T,k)=(\infty,1),(\frac{R}{10},1),(\frac{R}{10},2)$, log-log scale (slopes $-1$, $-2$, and $-4$).}
\label{fig1}
\end{figure}

\subsubsection{Symmetric periodic example}

We come back to the continuum setting addressed in this article, and consider the following coeffients $A:\R^2\to \Mabs$,
\begin{equation}\label{eq:mat2}
A(x)\,=\, \left(\frac{2 + 1.8 \sin(2\pi x_1)}{2 + 1.8\cos(2\pi x_2)}+\frac{2 + \sin(2\pi x_2)}{2 + 1.8\cos(2\pi x_1)}\right) \,\mathrm{Id},
\end{equation}
used as benchmark tests in \cite{Hou-04}.
In this case, $\alpha\simeq 0.35$, $\beta\simeq 20.5$, and
$A_\ho\simeq 2.75\,\mathrm{Id}$. The value of $A_\ho$ can be approximated on a single periodic cell with periodic boundary conditions to any order of precision.
We have considered the first two approximation formulas $A_{T,1,R,L,p}$ and $A_{T,2,R,L,p}$ of $A_\ho$,
cf. \eqref{eq:approx-ATKRLp-v}.
In all the tests, we have taken the following parameters:
\begin{itemize}
\item zero-order term $T= \frac{R}{100}$,
\item filters of order $p=3,4$ with $L=\frac{R}{3}$.
\end{itemize}
Theorem~\ref{th:error-estim-per} then yields
\begin{equation}\label{eq:test-per-sym-hom}
|{A}_{\infty,1,R,L,0}-A_\ho|\,\lesssim\, R^{-1},\quad |{A}_{T,1,R,L,3}-A_\ho|\,\lesssim\, R^{-2},\quad |{A}_{T,2,R,L,3}-A_\ho|\,\lesssim\,R^{-4}.
\end{equation}
Likewise, \eqref{eq:th-approx-TRk-per-full} in Remark~\ref{rem:approx-TRk} yields:
\begin{equation}\label{eq:test-per-sym-corr}
\frac{1}{L^d}\int_{Q_{L/2}} |\nabla \phi_{T,k,R}-\nabla \phi|^2 \, \lesssim \, R^{-2k}.
\end{equation}
The numerical tests have been performed with FreeFEM++ \cite{FreeFEM} and $R$ ranges from 3 to 60.
The results for the approximation of the correctors are in perfect agreement with \eqref{eq:test-per-sym-corr},
which describes correctly the decay, even for moderate $R$, see Figure~\ref{fig3}.
Note that there are two different behaviors for the 
naive approach whether oversampling is used or. Without oversampling, the error is given by $\mathcal{E}(R):=\fint_{Q_{R}}|\nabla \phi_{\infty,1,R}-\nabla \phi|^2$ and takes into account the boundary layer: it is of order $R^{-1}$. With oversampling, the error  is given by $\mathcal{E}(R):=\fint_{Q_{R/6}}|\nabla \phi_{\infty,1,R}-\nabla \phi|^2$, and the boundary layer is discarded: it is of order $R^{-2}$.
\begin{figure}
\centering
\psfrag{hh}{$\log_{10}(R)$}
\psfrag{kk}{$\log_{10} \mathcal E (R)$}
\includegraphics[scale=.5]{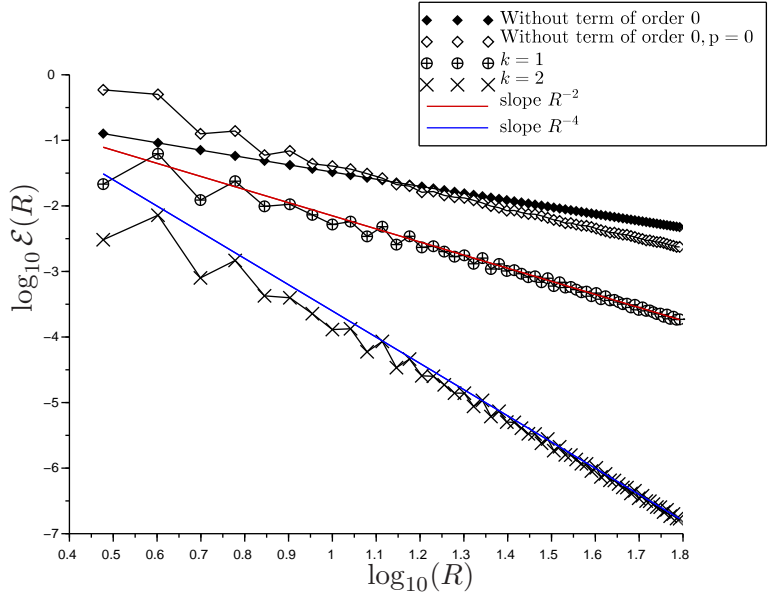}
\caption{Symmetric periodic example \eqref{eq:mat2}. Bound on $R\mapsto \mathcal{E}(R):=\fint_{Q_{R}}|\nabla \phi_{\infty,1,R}-\nabla \phi|^2$, and on $R\mapsto \mathcal{E}(R):=\fint_{Q_{R/6}}|\nabla \phi_{T,k,R}-\nabla \phi|^2$  for $(T,k)=(\infty,1),(\frac{R}{100},1),(\frac{R}{100},2)$, log-log scale (slopes $-1$, $-2$, $-2$, and $-4$).}
\label{fig3}
\end{figure} 
For the homogenized coefficients however, the theoretical predictions of \eqref{eq:test-per-sym-hom}
are only met for larger $R$: for smaller $R$, the averaging error dominates. This can be seen on Figure~\ref{fig5}, which displays the errors $|{A}_{\infty,1,R,L,0}-A_\ho|$ (naive approximation: no zero order term, no filtering but oversampling with $L<R$),
$|{A}_{T,1,R,L,3}-A_\ho|$, and $|{A}_{T,2,R,L,3}-A_\ho|$. 
The naive approximation yields better results for small $R$ because of the large error due to filtering.
As in the discrete case, the errors for $k=1$ and $k=2$ are ``super-algebraic'' and coincide until 
they cross the straight lines of slopes $-2$ and $-4$, respectively, when the systematic error $R^{-2k}$ dominates.
With respect to the tests in the discrete case, there is an additional feature: the overall error displays periodic oscillations, especially for $k=2$. 
Whereas the systematic error $T^{-2k}$ decays monotonically (as illustrated on Figure~\ref{fig3}), the averaging error does not: when $L$ is a multiple of the period, the error gets smaller, whence the periodic modulation in $L$. For $k=1$, the order of magnitude of these oscillations vanish as $R$ gets larger (and the averaging error becomes negligible wrt the systematic error). Indeed, the periodic modulation is still there but the systematic error $T^{-2k}\sim R^{-2}$ dominates the averaging error $L^{-(p+1)}\sim R^{-4}$, so that this effect is not relevant asymptotically.
For $k=2$ and $p=4$, although the averaging error is of order $L^{-5}$ whereas the systematic error is of order $T^{-4}$ in 
Theorem~\ref{th:error-estim-per}, the averaging error for $R\leq 60$ is still a non-negligible part of the overall error (of the same order as the systematic error). 
To further illustrate this fact we have plotted on Figure~\ref{fig5bis} the approximation errors $|{A}_{T,2,R,L,3}-A_\ho|$ and $|{A}_{T,2,R,L,4}-A_\ho|$, the difference of which is only due to the averaging function: for $p=3$, the averaging error is of order $L^{-4}\sim T^{-4}$ and the oscillations due to averaging are of the same order as the systematic error for all $R$ whereas for $p=4$, $L^{-5}\ll T^{-4}$ and the oscillations are smaller wrt to the systematic error as $R$ gets larger.
Again, these tests illustrate the fact that the computational approximation of the homogenized coefficients can be much slower than the computational approximation of the correctors for moderate $R$, due to averaging.

\medskip

\noindent In terms of complexity, the computational cost of the approximation of $A_{T,k,R,L,p}$ is exactly $k$-times higher that the computational cost of the approximation of $A_{\infty,1,R,L,0}$ since
the approximation of $A_{T,k,R,L,p}$ requires the approximation of $kd$ correctors whereas the naive approximation requires  the approximation of $d$ correctors for symmetric coefficients. Likewise for the correctors $\phi_{T,k,R}$ and $\phi_{\infty,1,R}$.
\begin{figure}
\centering
\psfrag{hh}{$\log_{10}(R)$}
\psfrag{kk}{$\log_{10} \mathcal E (R)$}
\includegraphics[scale=.5]{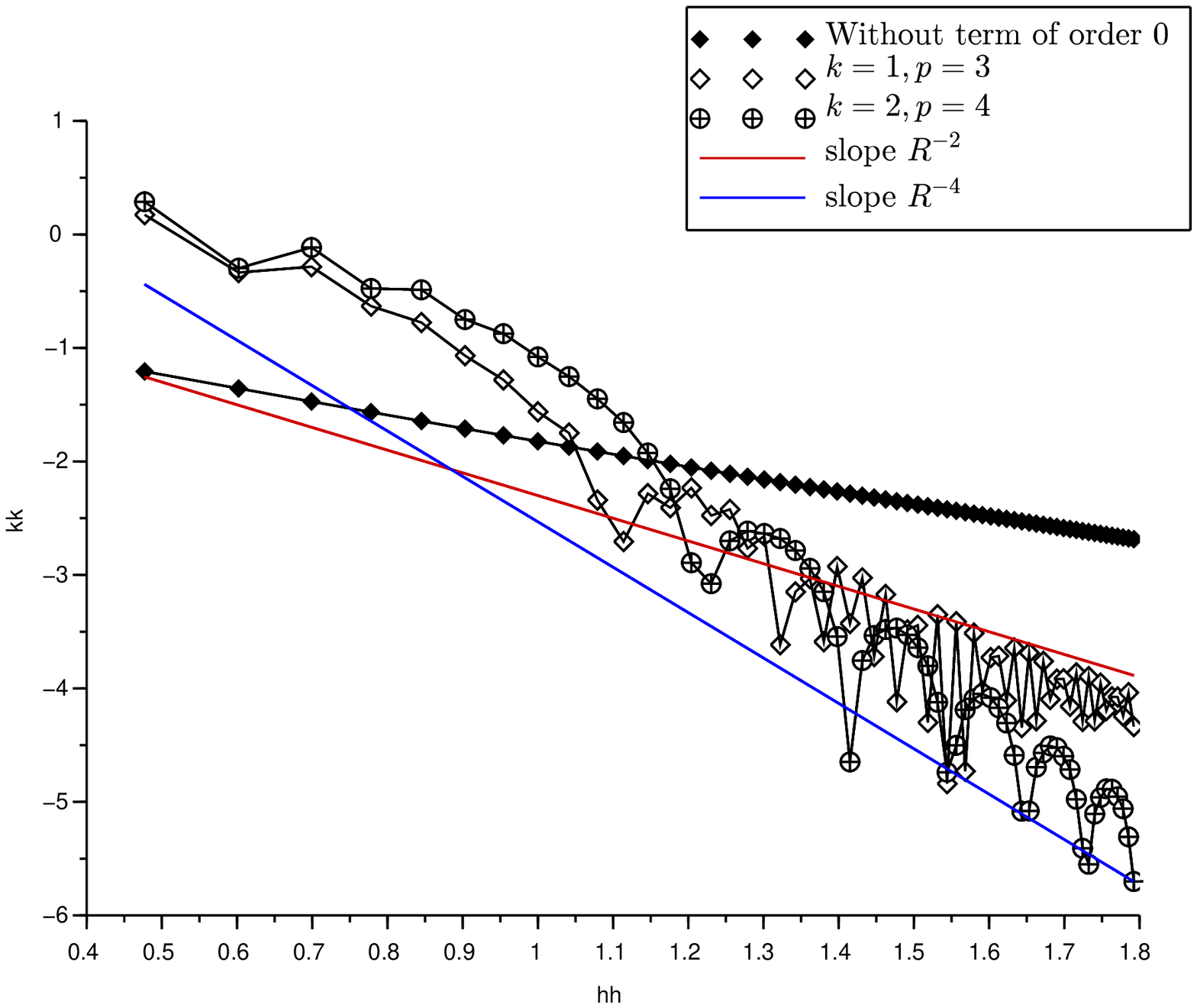}
\caption{Symmetric periodic example \eqref{eq:mat2}. Bound on $R\mapsto \mathcal{E}(R):=|{A}_{T,k,R,L,p}-A_\ho|$ for $(T,k,p)=(\infty,1,0),(\frac{R}{100},1,3),(\frac{R}{100},2,4)$, log-log scale (slopes $-1$, $-2$, and $-4$).}
\label{fig5}
\psfrag{hh}{$\log_{10}(R)$}
\psfrag{kk}{$\log_{10} \mathcal E (R)$}
\includegraphics[scale=.5]{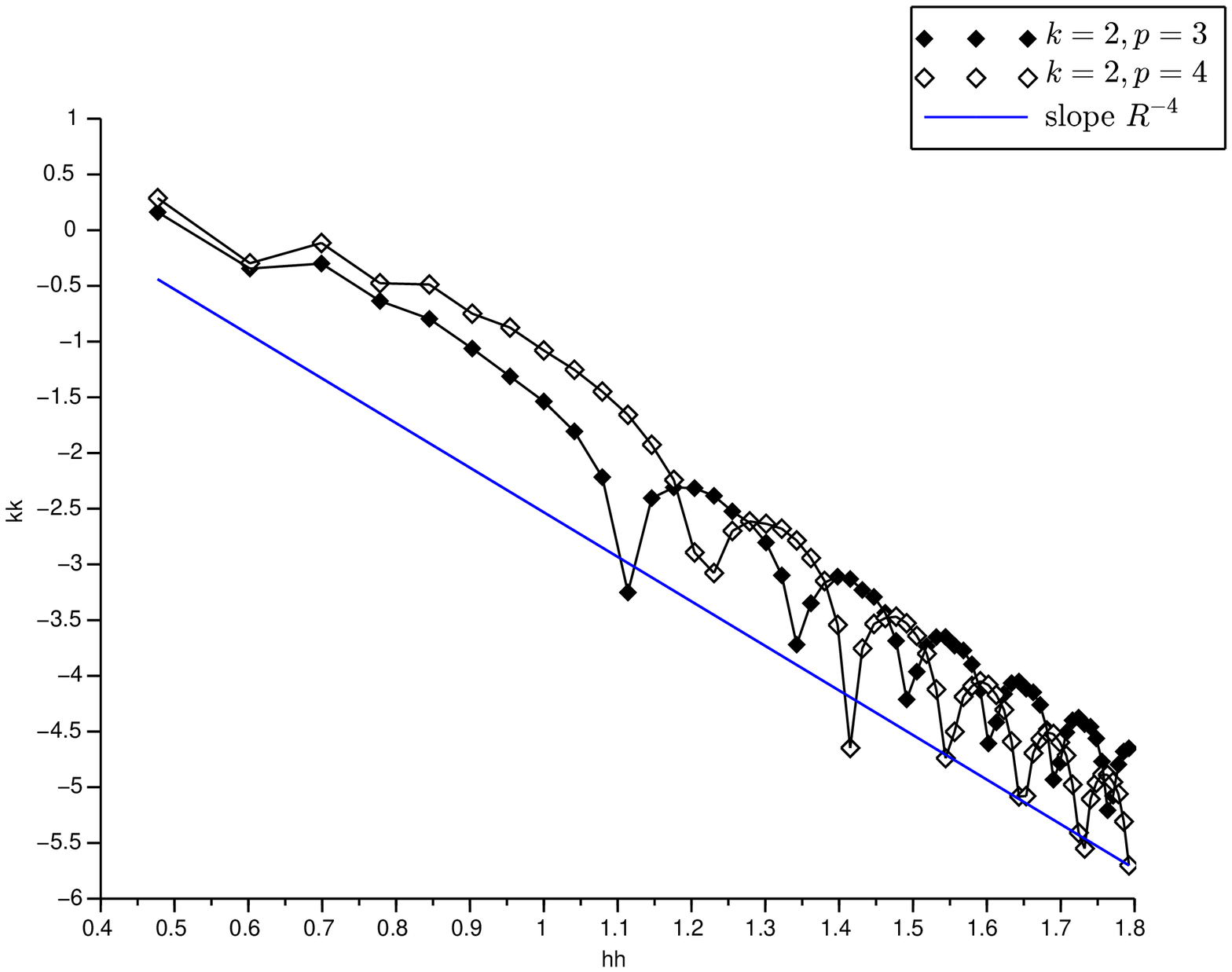}
\caption{Symmetric periodic example \eqref{eq:mat2}. Bound on $R\mapsto \mathcal{E}(R):=|{A}_{T,2,R,L,p}-A_\ho|$ for $(T,p)=(\frac{R}{100},3),(\frac{R}{100},4)$, log-log scale (slopes $-4$).}
\label{fig5bis}
\end{figure} 

\subsubsection{Symmetric almost periodic example}
We turn now to an almost periodic example and consider the following coefficients:
\begin{equation}\label{eq:mat3}
A(x)\,=\,
\begin{pmatrix}
4+\cos(2\pi(x_1+x_2))+\cos(2\pi \sqrt{2}(x_1+x_2)) & 0\\
0 & 6+\sin^2(2\pi x_1)+\sin^2(2\pi \sqrt{2}x_1)\\
\end{pmatrix},
\end{equation}
which are of Kozlov class.
In this case, $\alpha\geq 2$, $\beta\leq 8$, so that the ellipticity ratio $4=\frac{8}{2}$ is
much smaller than in the previous two examples ($100$ and $60$, respectively). On the one hand, this makes the homogenization
procedure easier, and we expect approximation errors to be much smaller.
On the other hand, in the almost periodic setting, we have no easy access to the homogenized coefficients (as opposed to the periodic setting where $A_\ho$ is given by a cell formula), so that one needs a reliable error estimator.
Since the coefficients are of Kozlov class, Theorem~\ref{th:quant-estim-quasiper} and Remark~\ref{rk:quant-estim-quasiper}
imply that for all $k\in \N$,
\begin{eqnarray*}
|A_{T,k}-A_\ho|&\lesssim & T^{-2k}, \\
\sup_{\R^d} |\nabla \phi_{T,k}-\nabla \phi|&\lesssim &T^{-k}.
\end{eqnarray*}
In particular, for all $k'> k$, the triangle inequality then yields:
\begin{eqnarray*}
|A_{T,k}-A_\ho|&\lesssim & T^{-2k'}+|A_{T,k}-A_{T,k'}|, \\
\sup_{\R^d} |\nabla \phi_{T,k}-\nabla \phi|&\lesssim &T^{-k'}+\sup_{\R^d} |\nabla \phi_{T,k}-\nabla \phi_{T,k'}|.
\end{eqnarray*}
Hence, provided $|A_{T,k}-A_{T,k'}|\gtrsim T^{-2k'}$, the quantity $|A_{T,k}-A_{T,k'}|$ is an estimator for 
$|A_{T,k}-A_\ho|$. This allows us not to bias the test with some arbitrary guess of $A_\ho$.
We have considered the first two approximations formulas $A_{T,1,R,L,p}$ and $A_{T,2,R,L,p}$ of $A_\ho$,
besides the naive approximation $A_{\infty,1,R,L,0}$ (no zero-order term, no filtering),
and taken $k'=3$ as a reference for the homogenized coefficients.
In all the tests, we have taken the following parameters:
\begin{itemize}
\item zero-order term $T=\frac{R}{100}$,
\item filter of order $p=4$ with $L=\frac{R}{3}$.
\end{itemize}
We then have by Theorem~\ref{th:error-estim-Kozlov}
\begin{eqnarray*}
|A_{\infty,1,R,L,0}-A_{\ho}|&\lesssim & R^{-5}+R^{-6}+|A_{\infty,1,R,L,0}-A_{T,3,R,L,4}|, \\
|A_{T,1,R,L,4}-A_{\ho}|&\lesssim & R^{-5}+R^{-6}+|A_{T,1,R,L,4}-A_{T,3,R,L,4}|, \\
|A_{T,2,R,L,4}-A_{\ho}|&\lesssim & R^{-5}+R^{-6}+|A_{T,2,R,L,4}-A_{T,3,R,L,4}|,
\end{eqnarray*}
where the error-term $R^{-5}$ is due to averaging (cf. $p=4$) and the second error-term due to
the systematic error (for $k'=3$).
The choice of the filter $p=4$ minimizes the error due to averaging (which is expected to be of higher order than the systematic error for $k=1,2$). Note that the chosen estimator for $|A_{T,k,R,L,4}-A_{\ho}|$
is not very sensitive to the averaging function since the same averaging function is taken 
both in $A_{T,k,R,L,4}$ for $k=1,2$ and in $A_{T,3,R,L,4}$ --- this is not the case for the estimator of $|A_{\infty,1,R,L,0}-A_{\ho}|$.
This explains why  there are not many oscillations on Figure~\ref{fig10} (besides for small 
$R$) for $|A_{T,k,R,L,4}-A_{\ho}|$ (where the errors for $k=1$ and $k=2$ decay at the predicted rates $R^{-2}$ and $R^{-4}$, respectively), whereas they are oscillations for $|A_{\infty,1,R,L,0}-A_{\ho}|$ (which decays at the predicted rate $R^{-1}$ as well).

\medskip
\noindent
For the approximation of correctors, recall that for $k>0$,
$$
\sup_{\R^d} |\nabla \phi_{T,k}-\nabla \phi|\,\lesssim \,T^{-k}.
$$
Since the difference of $\nabla \phi_{T,k}$ and $\nabla \phi_{T,k,R}$ on $Q_{L/2}$ is of infinite order in terms of $\frac{R-L}{\sqrt{T}}$ by Theorem~\ref{th:approx-TR}, this turns with $T= \frac{R}{100}$ and $L=\frac{R}{3}$ into
$$
\fint_{Q_{L/2}}|\nabla \phi_{T,k,R}-\nabla \phi|^2\,\lesssim\, R^{-2k},
$$
which we confront to the results of numerical experiments in the form of 
\begin{equation}\label{eq:quasi-perio-test-corr}
\fint_{Q_{L/2}}|\nabla \phi_{T,k,R}-\nabla \phi_{T,k',R}|^2\,\lesssim\, R^{-2k}
\end{equation}
for $3=k'>k=1,2$. 
Without zero-order term, we still expect
\begin{equation}\label{eq:quasi-perio-test-corr}
\fint_{Q_{L/2}}|\nabla \phi_{\infty,1,R}-\nabla \phi_{T,k',R}|^2\,\lesssim\, R^{-2}
\end{equation}
since oversampling is implicitly taken into accound here (cf. $L/2<R$) and we discard the boundary layer (this estimate is however not proved).
On Figure~\ref{fig10b}, we obtain a monotonic decay with the expected convergence rates
$-1$, $-2$, and $-4$.
In this example, it is not clear that the approximation of correctors is better than the approximation of homogenized coefficients. However, as mentioned above, the estimator $|A_{T,k,R,L,4}-A_{T,3,R,L,4}|$ for $|A_{T,k,R,L,4}-A_{\ho}|$ is not very sensitive to averaging (since they have the same averaging function) and the averaging part of the error is mostly in the term $R^{-5}$, which we did not represent. 

\medskip

\noindent
In terms of computational cost, the approximation of $A_{T,k,R,L,p}$ is again $k$ times more expensive than the approximation of $A_{\infty,1,R,L,0}$. Likewise for the correctors $\phi_{T,k,R}$ and $\phi_{\infty,1,R}$.
\begin{figure}
\centering
\psfrag{hh}{$\log_{10}(R)$}
\psfrag{kk}{$\log_{10} \mathcal E (R)$}
\includegraphics[scale=.5]{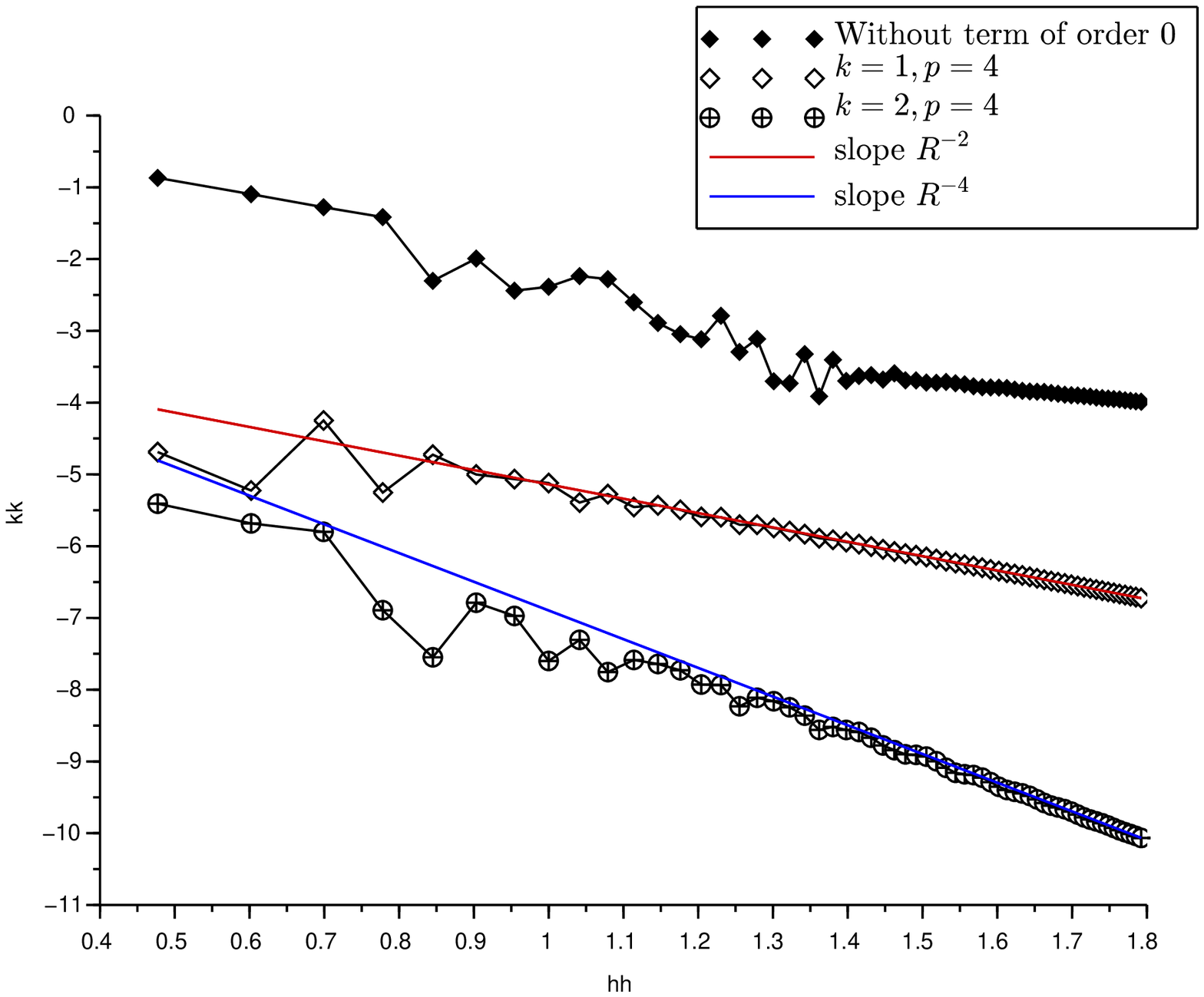}
\caption{Symmetric almost periodic example \eqref{eq:mat3}. Bound on $R\mapsto \mathcal E (R):=|A_{T,k,R,L,p}-A_{\ho}|$ for $(T,k,p)=(\infty,1,0),(\frac{R}{100},1,3),(\frac{R}{100},2,3)$, log-log scale (slopes $-1$, $-2$, and $-4$).}
\label{fig10}
\psfrag{hh}{$\log_{10}(R)$}
\psfrag{kk}{$\log_{10} \mathcal E (R)$}
\includegraphics[scale=.5]{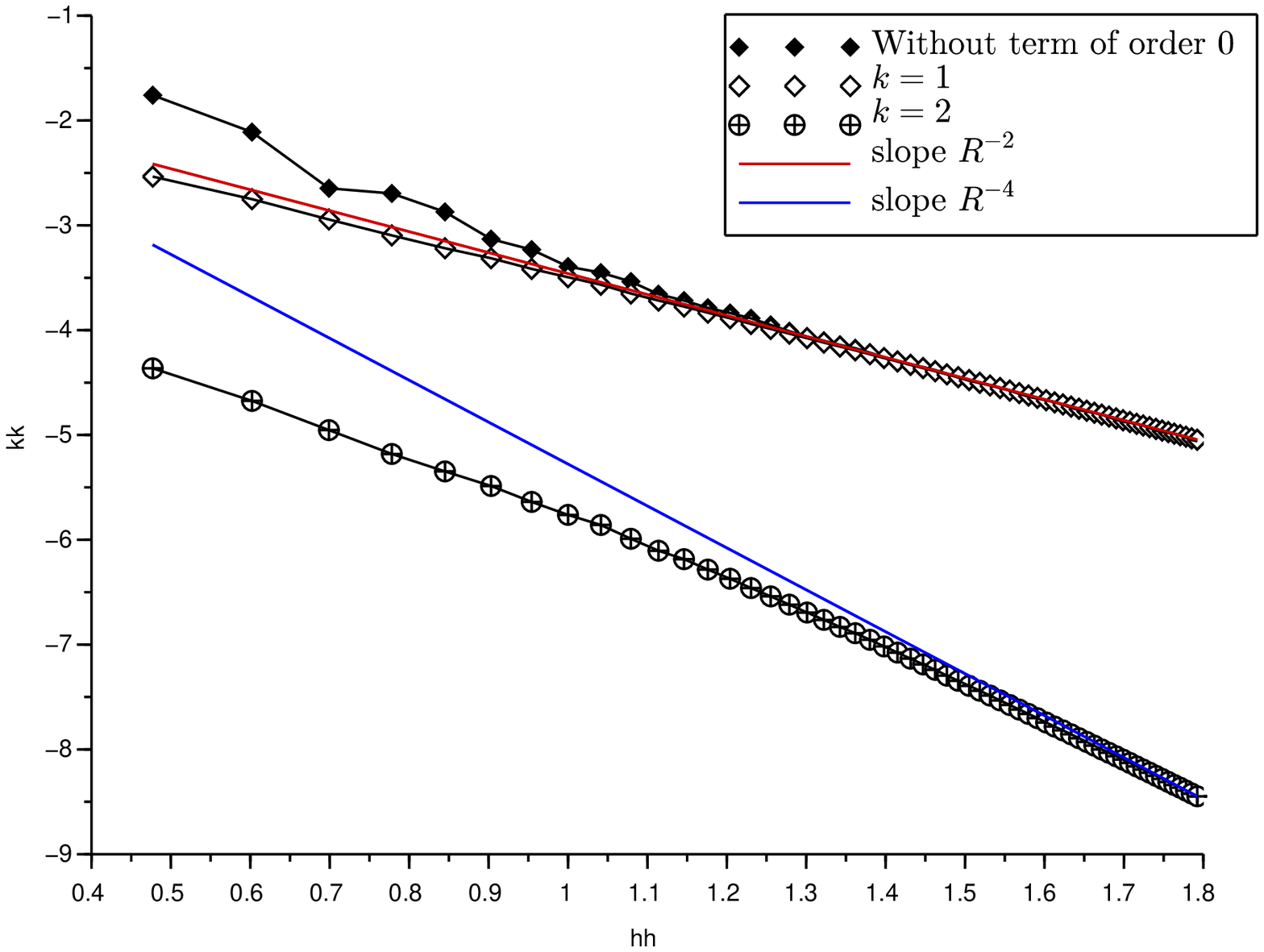}
\caption{Symmetric almost periodic example \eqref{eq:mat3}. Bound on $R\mapsto \mathcal{E}(R):=\fint_{Q_{R/6}}|\nabla \phi_{T,k,R}-\nabla \phi_{T,3,R}|^2$ for $(T,k)=(\infty,1),(\frac{R}{100},1),(\frac{R}{100},2)$, log-log scale (slopes $-2$, $-2$, and $-4$).}
\label{fig10b}
\end{figure} 

\begin{rem}
Although the convergence rates are the same in this almost periodic example and in the periodic examples of the previous paragraphs, the observed errors are smaller for this almost periodic example.
The difference lies in the ellipticity contrast (that is, the ratio of the two best ellipticity constants), which is expected to play a significant role in the prefactors of the estimates since it quantifies the nonlinearity of the homogenization process (the closer to $1$ the ellipticity contrast, the smaller the prefactors). As already emphasized, the ellipticity contrast is less than $4$ in this almost periodic example whereas it is $100$ and $60$ in the two periodic examples.
\end{rem}

\subsection{Numerical tests for non-symmetric coefficients}\label{sec:test2}

The tests, expected errors, and results are similar to those of the symmetric setting.

\subsubsection{Non-symmetric periodic example}

We consider the following coefficients:
\begin{equation}\label{eq:mat4}
A(x)\,=\,
\begin{pmatrix}
\frac{2 + 1.8 \sin(2\pi x_1)}{2 + 1.8\cos(2\pi x_2)}+\frac{2 + \sin(2\pi x_2)}{2 + 1.8\cos(2\pi x_1)} & 2 + \sin(2\pi x_1)\cos(2\pi x_2)\\
-2 + \sin(2\pi x_1)\cos(2\pi x_2) & \frac{2 + 1.8 \sin(2\pi x_1)}{2 + 1.8\cos(2\pi x_2)}+\frac{2 + \sin(2\pi x_2)}{2 + 1.8\cos(2\pi x_1)}\\
\end{pmatrix}.
\end{equation}
In addition to the naive approximation, we take $L=\frac{R}{3}$, $T=\frac{R}{100}$, and filters of orders $p=3,4$. 
As expected, Figures~\ref{fig12} and~\ref{fig14} are in agreement with the theoretical predictions
for $(k,p)=(1,3),(2,4)$:
\begin{equation*}
\fint_{Q_{L/2}} |\nabla \phi_{T,k,R}-\nabla \phi|^2 \, \lesssim \, R^{-2k},\quad
|{A}_{T,k,R,L,p}-A_{\ho}|\,\lesssim\, R^{-2k},
\end{equation*}
and 
\begin{equation*}
\fint_{Q_{R}} |\nabla \phi_{\infty,1,R}-\nabla \phi|^2 \, \lesssim \, R^{-1},\quad
|{A}_{\infty,1,R,L,0}-A_{\ho}|\,\lesssim\, R^{-1}
\end{equation*}
for the naive approximation. When oversampling  is combined with the naive approximation (that is, discarding the boundary layer), we expect a better rate for the approximation of correctors (although there is no rigorous proof)
\begin{equation*}
\fint_{Q_{L/2}} |\nabla \phi_{\infty,1,R}-\nabla \phi|^2 \, \lesssim \, R^{-2}.
\end{equation*}
\begin{figure}
\centering
\psfrag{hh}{$\log_{10}(R)$}
\psfrag{kk}{$\log_{10}\mathcal E(R)$}
\includegraphics[scale=.5]{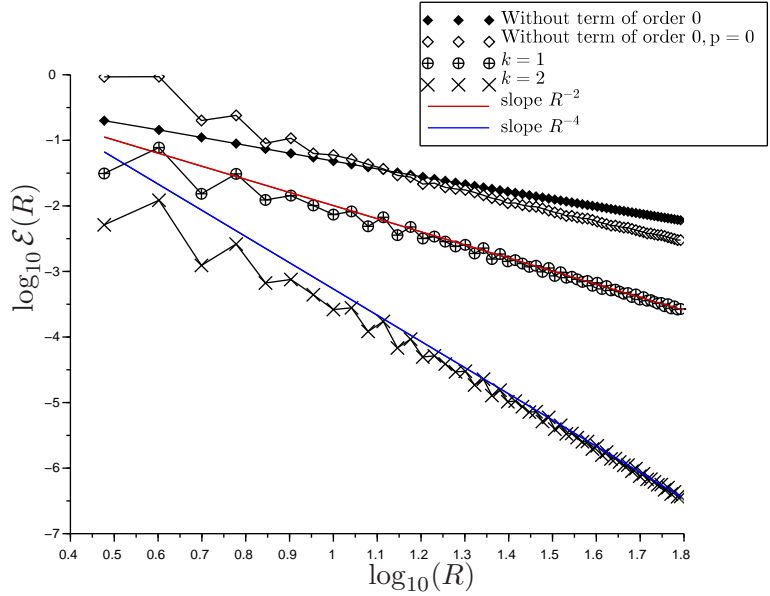}
\caption{Non-symmetric periodic example \eqref{eq:mat4}. Bound on $R\mapsto \mathcal{E}(R):=\fint_{Q_{R}} |\nabla \phi_{\infty,1,R}-\nabla \phi|^2$
and $R\mapsto \mathcal{E}(R):=\fint_{Q_{R/6}} |\nabla \phi_{T,k,R}-\nabla \phi|^2$ for $(T,k)=(\infty,1),(\frac{R}{100},1),(\frac{R}{100},2)$, log-log scale (slopes $-1$, $-2$, $-2$, and $-4$).}
\label{fig12}
\end{figure}
\begin{figure}
\centering
\psfrag{hh}{$\log_{10}(R)$}
\psfrag{kk}{$\log_{10}\mathcal E(R)$}
\includegraphics[scale=.5]{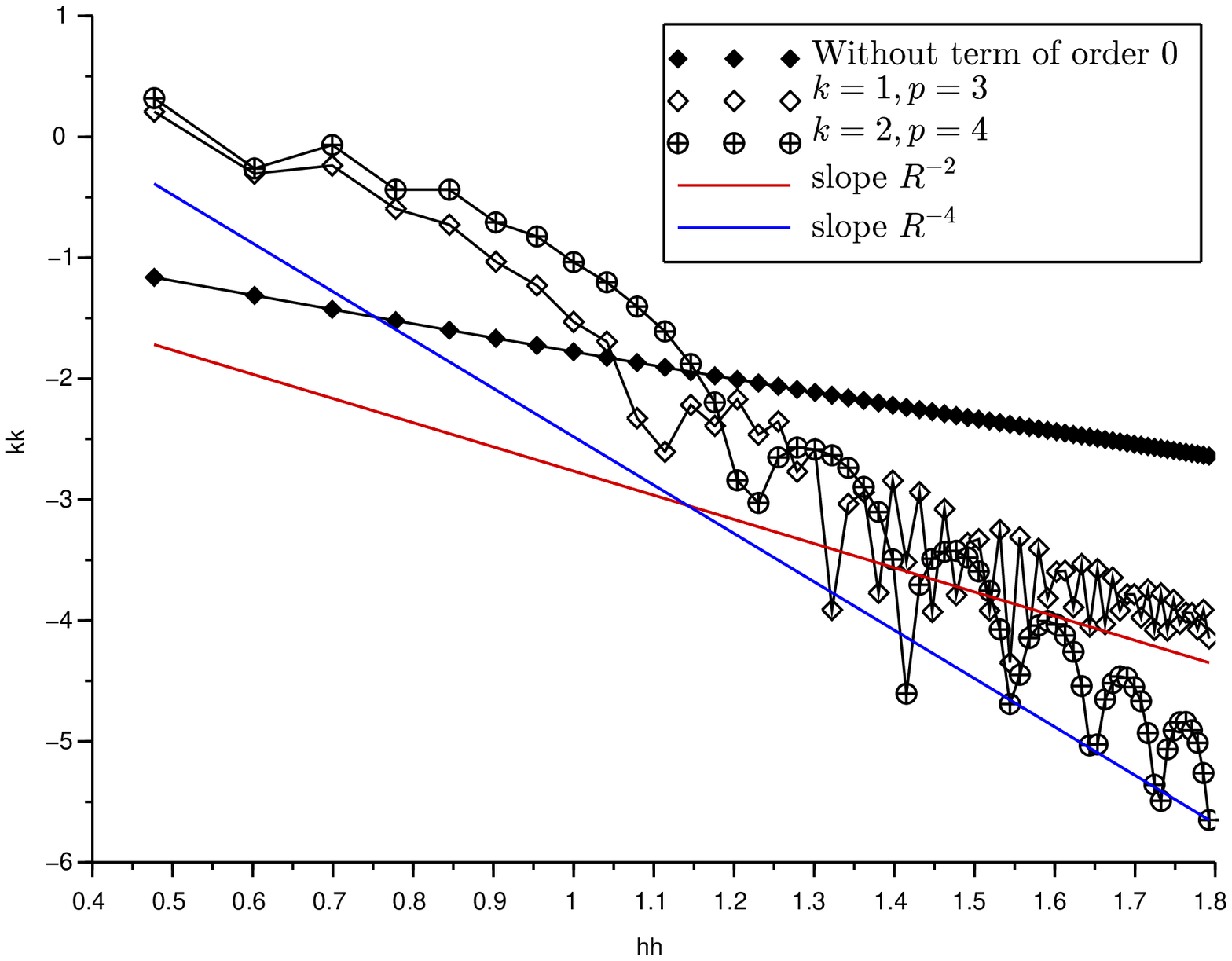}
\caption{Non-symmetric periodic example \eqref{eq:mat4}. Bound on $R\mapsto \mathcal{E}(R):=|A_{T,k,R,L,p}-A_\ho|$ for $(T,k,p)=(\infty,1,0),(\frac{R}{100},1,3),(\frac{R}{100},2,4)$, log-log scale (slopes $-1$, $-2$, and $-4$).}
\label{fig14}
\end{figure}

\subsubsection{Non-symmetric almost periodic example}

We consider the following coefficients:
\begin{equation}\label{eq:mat5}
A(x)\,=\,
\begin{pmatrix}
4+\cos(2\pi(x_1+x_2))+\cos(2\pi \sqrt{2}(x_1+x_2)) & 2 + \sin(2\pi x_1)\cos(2\pi x_2)\\
-2 + \sin(2\pi x_1)\cos(2\pi x_2) & 6+\sin^2(2\pi x_1)+\sin^2(2\pi \sqrt{2}x_1)\\
\end{pmatrix},
\end{equation}
In addition to the naive approximation, we take $L=\frac{R}{3}$, $T=\frac{R}{100}$, and a filter of order $p=4$. 
The results are similar as for the symmetric case and Figures~\ref{fig18}~and~\ref{fig18b} 
show that for $k=1,2$,
\begin{equation*}
\fint_{Q_{L/2}} |\nabla \phi_{T,k,R}-\nabla \phi_{T,3,R}|^2 \, \lesssim \, R^{-2k},\quad
|{A}_{T,k,R,L,3}-{A}_{T,3,R,L,3}|\,\lesssim\, R^{-2k}.
\end{equation*}
and 
\begin{equation*}
\fint_{Q_{L/2}} |\nabla \phi_{\infty,1,R}-\nabla \phi_{T,3,R}|^2 \, \lesssim \, R^{-1},\quad
|{A}_{\infty,1,R,L,0}-{A}_{T,3,R,L,3}|\,\lesssim\, R^{-1}
\end{equation*}
for the naive approximation. 
Since in this comparison, oversampling  is implictly used (since $L/2<R$) for the naive approach, the estimate for the corrector is pessimistic and we rather expect (although there is no rigorous proof)
\begin{equation*}
\fint_{Q_{L/2}} |\nabla \phi_{\infty,1,R}-\nabla \phi_{T,3,R}|^2\, \lesssim \, R^{-2}.
\end{equation*}
\begin{figure}
\centering
\psfrag{hh}{$\log_{10}(R)$}
\psfrag{kk}{$\log_{10}\mathcal E(R)$}
\includegraphics[scale=.5]{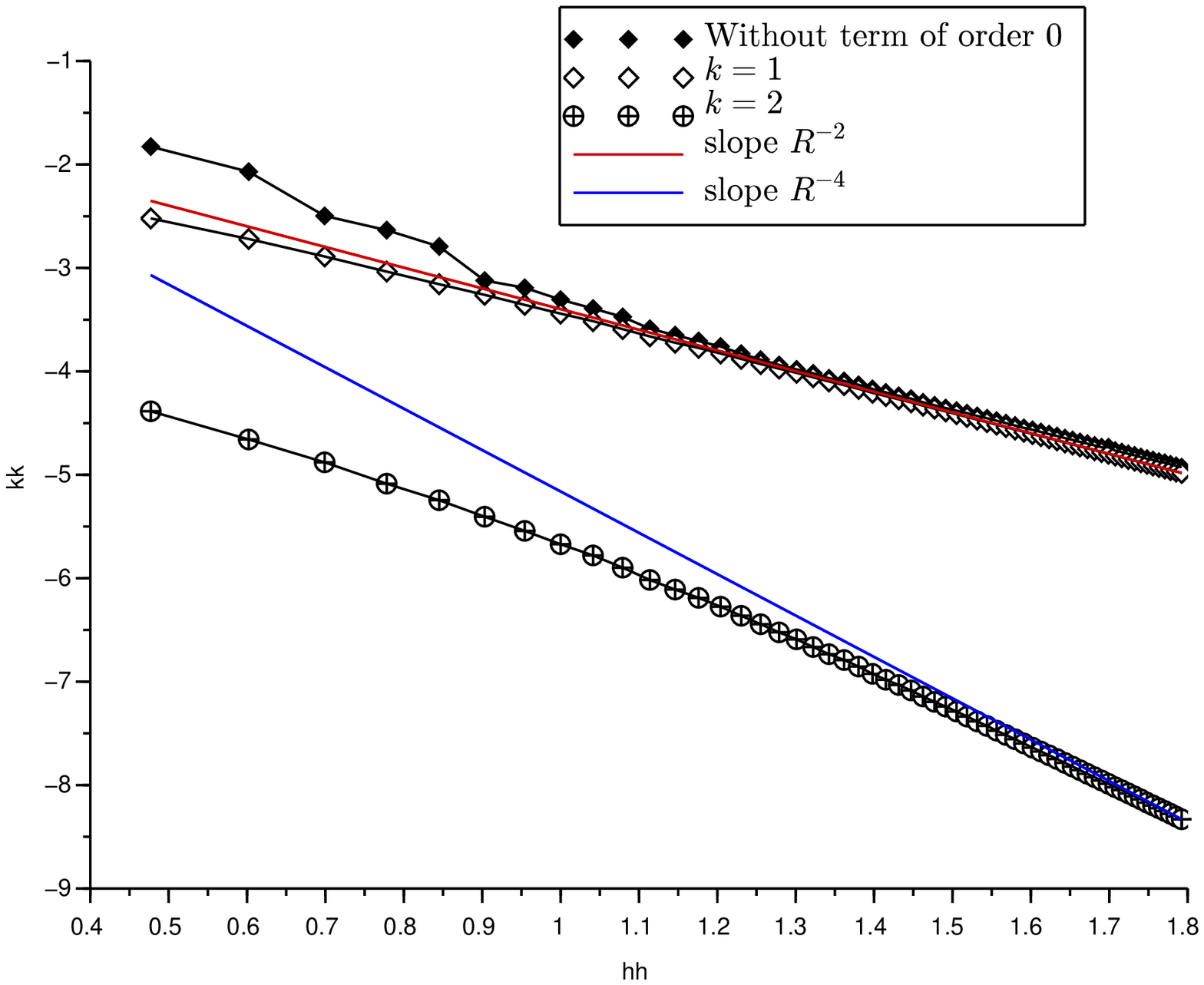}
\caption{Non-symmetric almost periodic example \eqref{eq:mat5}. Bound on $R\mapsto \mathcal{E}(R):=\fint_{Q_{R/6}} |\nabla \phi_{T,k,R}-\nabla \phi_{T,3,R}|^2$ for $(T,k)=(\infty,1),(\frac{R}{100},1),(\frac{R}{100},2)$, log-log scale (slopes $-2$, $-2$, and $-4$).}
\label{fig18}
\psfrag{hh}{$\log_{10}(R)$}
\psfrag{kk}{$\log_{10}\mathcal E(R)$}
\includegraphics[scale=.5]{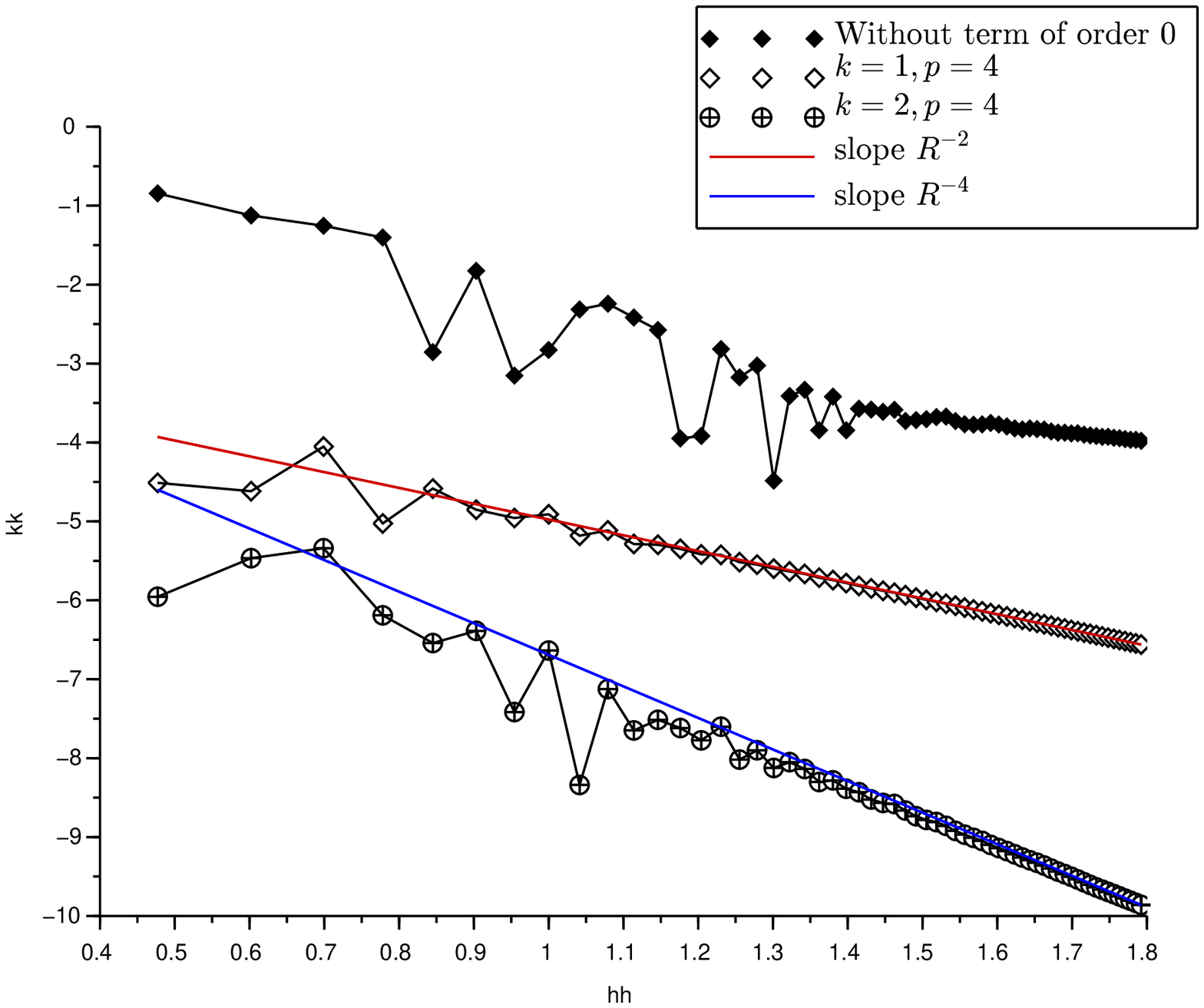}
\caption{Non-symmetric almost periodic example \eqref{eq:mat5}. Bound on $R\mapsto \mathcal{E}(R):=|A_{T,k,R,L,4}-A_{T,3,R,L,4}|$ for $(T,k,p)=(\infty,1,0),(\frac{R}{100},1,3),(\frac{R}{100},2,4)$, log-log scale (slopes $-1$, $-2$, and $-4$).}
\label{fig18b}
\end{figure} 
\medskip

\noindent In terms of complexity, the case non-symmetric coefficients is slightly different than the case of symmetric coefficients.
Indeed, the approximation of $A_{T,k,R,L,p}$ requires the approximation of $2kd$ correctors ($kd$ regularized correctors and $kd$ regularized dual correctors) whereas the naive approximation does not require the approximation of the dual correctors stricto sensu, so that the approximation of $A_{T,k,R,L,p}$ is $2k$ more expensive than the naive approximation.
For the approximation of the correctors $\phi_{T,k,R}$ and $\phi_{\infty,1,R}$, the ratio remains $k$.

%%%%%%%%%%%%%%%%%%%%%%%
%%%%%%%%%%%%%%%%%%%%%%%

\section{Combination with numerical homogenization methods}\label{sec:reduc}

The aim of this last section is to show that the regularization and extrapolation method to approximate homogenized coefficients and correctors can be used to reduce the resonance error in standard numerical homogenization methods.
Following~\cite{Gloria-06b,Gloria-07b} we first introduce a general analytical framework and prove the consistency 
in the case of locally stationary ergodic coefficients. We then recover variants of standard numerical homogenization methods after discretization. 
We conclude the presentation by a quantitative convergence analysis in the case of periodic coefficients.

\subsection{Analytical framework}

Let $A_\e:D\times \Omega \to \Mab$ be a family of random diffusion coefficients parametrized by $\e>0$.
We make the assumption that $A_\e$ is locally stationary and ergodic (and assume some cross-regularity).
\begin{hypo}\label{hypo:reg}
There exists a random Carath\'eodory function (that is continuous in the first variable and measurable in the second variable) 
$\tilde A:D\times \R^d\times \Omega \to \Mab$, and a constant $\kappa>0$ such that
\begin{itemize}
\item For all $x\in D$, the random field $\tilde A(x,\cdot,\cdot)$ is stationary on $\R^d\times \Omega$, and ergodic;
\item \emph{(cross-regularity)} $\tilde A$ is $\kappa$-Lipschitz in the first variable: for all $x,y\in D$, for almost every $z\in \R^d$, and for almost
every realization $\omega \in \Omega$,
$$
|\tilde A(x,z,\omega)-\tilde A(y,z,\omega)|\,\leq\, \kappa|x-y|;
$$
\item For all $x\in D$, for and all $\e>0$, and for almost every  realization $\omega \in \Omega$,
$$
A_\e(x,\omega)\,:=\,\tilde A(x,\frac{x}{\e},\omega)
$$
almost surely.
\end{itemize}
\end{hypo}
\noindent Under Hypothesis~\ref{hypo:reg} $A_\e$ H-converges on $D$ to some deterministic diffusion function $A_\ho:D\to \Mab$ almost surely, which is 
$\kappa$-Lipschitz and characterized for all $x\in D$ and $\xi,\xi' \in \R^d$ by
$$
\xi'\cdot A_\ho(x)\xi\,=\, \expec{(\xi'+\nabla \phi'(x,0,\cdot))\cdot \tilde A(x,0,\cdot)(\xi+\nabla \phi(x,0,\cdot))},
$$
where $\phi(x,\cdot,\cdot),\phi'(x,\cdot,\cdot)$ are the corrector in direction $\xi$ and dual corrector in direction $\xi'$, respectively, associated with the stationary coefficients $\tilde A(x,\cdot,\cdot)$ ($x$ is treated as a parameter).
The proof is standard: by H-compacteness, for almost every realization, $A_\e$ H-converges up to extraction to some limit $A^*$.
Using the locality of H-convergence and the fact that $\tilde A$ is uniformly Lipschitz in the first variable, one concludes that $A^*(x)=A_\ho(x)$ almost surely.
Note that if we weaken the cross-regularity assumption from Lipschitz continuity on $D$ to continuity on $\overline{D}$, 
the result holds true as well --- although the proof of the homogenization result is more subtle,  see 
\cite[Theorem~4.1.1]{Bourgeat-Mikelic-Wright-94}. 

\medskip

\noindent The combination of the regularization and extrapolation method with the analytical framework of \cite{Gloria-06b,Gloria-07b} yields the following local approximation of $A_\ho$.
\begin{defi}\label{def:reg-ana}
Let $\mu:\R^+ \to [0,1]$ have support in $[-1,1]$ and be such that $\int_{-1}^1 \mu=1$.
For all $\rho>0$, let $\mu_\rho:\R^d\to \R^+$ denote the $\rho$-rescaled multidimensional version $y\mapsto \rho^{-d}\prod_{i=1}^d \mu(y_i/\rho)$ of $\mu$ (the support of which is in $Q_\rho$).
For all $\delta >1$, $\rho>0$, $T>0$, $k>0$, and $\e>0$, we denote by $A_{T,k,\rho,\e}^{\delta}:D\to \Md$ the function defined by: For all $\xi,\xi'\in \R^d$ and for $x\in D$,
\begin{multline}\label{eq:reg-approx}
\xi' \cdot A_{T,k,\rho,\e}^\delta(x)\xi :=\int_{Q_{\rho}\cap T_{-x}D} (\xi' +P_{\rho,x}(\nabla_y {v'}_{T,k}^{\delta\rho,\e})(x,y))\\ \cdot A_\e(x+y)(\xi+P_{\rho,x}(\nabla_y v_{T,k}^{\delta\rho,\e})(x,y))\mu_{\rho}(y)dy,
\end{multline}
where $T_{-x}D:=\{y\in \R^d \,|\, y+x\in D\}$, $v_{T,k}^{\delta\rho,\e},{v'}_{T,k}^{\delta\rho,\e}$ are defined by Richardson extrapolations of the unique weak solutions $v_{T,1}^{\delta\rho,\e},{v'}_{T,1}^{\delta\rho,\e}$ in $H^1_0(\Rhodrho)$ of 
\begin{equation}\label{eq:reg-corr-approx}
\begin{array}{rcl}
(T\e^2)^{-1}v_{T,1}^{\delta\rho,\e}(x,y)-\nabla \cdot A_\e(x+y)(\xi+\nabla_y v_{T,1}^{\delta\rho,\e}(x,y))&=&0,\\
(T\e^2)^{-1}{v'}_{T,1}^{\delta\rho,\e}(x,y)-\nabla \cdot A^*_\e(x+y)(\xi'+\nabla_y {v'}_{T,1}^{\delta\rho,\e}(x,y))&=&0,
\end{array}
\end{equation}
and $P_{\rho,x}$ is the projection on $\mu_\rho$-mean free fields of $L^2(Q_{\rho}\cap T_{-x}D)$:
for all $\psi \in L^2(Q_{\rho}\cap T_{-x}D)$,
\begin{equation}\label{def:mean-free-mu}
P_{\rho,x}(\psi)\,:=\,\psi-\frac{1}{\int_{Q_{\rho}\cap T_{-x}D}\mu_\rho(y)dy}\int_{Q_{\rho}\cap T_{-x}D}\psi(y) \mu_\rho(y)dy.
\end{equation}
\end{defi}
\begin{rem}
In \eqref{eq:reg-approx} we use $P_{\rho,x}(\nabla_y v_{T,k}^{\delta\rho,\e})(x,\cdot)$ instead of simply $\nabla_y v_{T,k}^{\delta\rho,\e}(x,\cdot)$ (as presented in \cite{Gloria-12b}). As we shall see, this has the advantage to ensure that 
$A_{T,k,\rho,\e}^\delta$ be uniformly elliptic for symmetric coefficients, while it is still consistent with the homogenization theory of locally stationary ergodic coefficients, and still reduces the resonance error.
\end{rem}

\medskip
\noindent The scaling of the zero-order term in \eqref{eq:reg-corr-approx} may be surprising since it makes the zero-order term diverge as $\e \downarrow 0$. It is however natural, as can be seen on the case of periodic coefficients: if $A_\e(y)=A(y/\e)$ with $A$ periodic, the change of variables $y=\e z$ turns \eqref{eq:reg-corr-approx} into the equation
\begin{equation*}
\begin{array}{rcl}
T^{-1}w(z)-\nabla \cdot A(x+z)(\xi+\nabla w(z))&=&0,\\
T^{-1}w'(z)-\nabla \cdot A^*(x+z)(\xi'+\nabla w'(z))&=&0,
\end{array}
\end{equation*}
which is of the form of \eqref{eq:extra-approx}. 
This shows in particular that it is crucial that $\e$ be the ``actual lengthscale" of the problem, and not
an abstract parameter, which, in this locally random case, is encoded in the assumption $A_\e(x)=\tilde A(x,\frac{x}{\e})$.

\medskip

\noindent Under Hypothesis~\ref{hypo:reg}, the following convergence result holds:
\begin{theo}\label{th:main-reg}
Let $D$ be a domain with a Lipschitz boundary.
For all $\e>0$, let $A_\e$ satisfy Hypothesis~\ref{hypo:reg}, and for all $\rho>0$, $\delta >1$,  $T>0$, and $k\in \N$, let $A_{T,k,\rho,\e}^\delta$ be as in Definition~\ref{def:reg-ana}.
Then, for almost every realization,  there exists a regime of parameters $(T\uparrow \infty,\rho\downarrow 0,\e \downarrow 0)$  for which $A_{T,k,\rho,\e}^\delta$ is uniformly elliptic on $D$, and for all $x\in D$, 
\begin{eqnarray}
\limsup_{T\uparrow \infty}\limsup_{\rho \downarrow 0}\limsup_{\e \downarrow 0} |A_{T,k,\rho,\e}^\delta(x)-A_{\ho}(x)|\,=\,0 .\label{eq:thmain-2-reg}
\end{eqnarray}
The limit in $T$ in \eqref{eq:thmain-2-reg} is uniform in $\rho$ for all $x$ at positive distance from $\partial D$.
In particular, we also have for all $1\leq p<\infty$
\begin{eqnarray}
\limsup_{\rho \downarrow 0,T\uparrow \infty}\limsup_{\e \downarrow 0} \|A_{T,k,\rho,\e}^\delta-A_{\ho}\|_{L^p(D)}\,=\,0 ,\label{eq:thmain-2-reg2}
\end{eqnarray}
where the order of the limits between $\rho$ and $T$ is irrelevant.
In addition, if $A_\e$ is symmetric, then there exists $\gamma>0$ depending only on $D$ such that $A_{T,k,\rho,\e}^\delta\in L^\infty(D,\mathcal{M}_{\alpha'\beta'})$ with $\alpha'=\gamma \alpha$ and $\beta'=2^d(1+\delta^{d})\beta $.
\end{theo}
\begin{rem}\label{rem:coercivity-discrete}
For symmetric coefficients, Theorem~\ref{th:main-reg} ensures that $A_{T,k,\rho,\e}^\delta$ is uniformly elliptic for any values of the parameters. For non-symmetric coefficients however, we have no a priori uniform ellipticity result in general.
Under any of the three quantitative assumptions \eqref{eq:Poinca-hor},  \eqref{eq:Poinca-hor-weak}, and \eqref{eq:Poinca-vert}, one may quantify the difference between $A_{T,k,\rho,\e}^\delta$ and $A_\ho$ and thus deduce the uniform ellipticity of $A_{T,k,\rho,\e}^\delta$  from that of $A_\ho$ under mild conditons on the parameters.
For general coefficients, we can only give some information on the approximation $A_{T,1,\rho,\e}^\delta$ of
$A_\ho$, relying on the uniform ellipticity of $A_{T,1}$ proved in Proposition~\ref{prop:coerc}, by using Theorem~\ref{th:approx-TR} and the Lipschitz continuity of $\tilde A$ in the slow variable.
For the approximations $A_{T,k,\rho,\e}^\delta$ with $k>1$ of $A_\ho$, we are not able to provide any additional information in general.
\end{rem}
\noindent From this theorem, one may directly deduce the convergence of the numerical homogenization method:
\begin{corollary}\label{coro:reg}
Under the assumptions of Theorem~\ref{th:main-reg}, for almost every realization there exists a regime of parameters $(T\uparrow \infty,\rho\downarrow 0,\e \downarrow 0)$ such that $A_{T,k,\rho,\e}^\delta$ is uniformly elliptic, and for all $f\in H^{-1}(D)$
the unique weak solution $u_{T,\rho,\e}^\delta\in H^1_0(D)$ of
\begin{equation}\label{eq:Tkrhoepsdelta}
-\nabla \cdot A_{T,k,\rho,\e}^\delta\nabla u_{T,k,\rho,\e}^\delta=f
\end{equation}
satisfies
\begin{equation}\label{eq:coro-reg}
\limsup_{\rho \downarrow 0,T\uparrow \infty}\limsup_{\e \downarrow 0} \|u_{T,k,\rho,\e}^\delta-u_\ho\|_{H^1(D)}=0
\end{equation}
almost surely, where $u_\ho\in H^1_0(D)$ is the unique weak solution of
\begin{equation*}
-\nabla \cdot A_\ho\nabla u_{\ho}=f.
\end{equation*}
As a consequence of H-convergence we also have that 
\begin{equation}\label{eq:coro-reg2}
\limsup_{\rho \downarrow 0,T\uparrow \infty}\limsup_{\e \downarrow 0} \|u_{T,k,\rho,\e}^\delta-u_\e\|_{L^2(D)}=0
\end{equation}
almost surely, where $u_\e$ is the unique weak solution in $H^1_0(D)$ of 
\begin{equation*}
-\nabla \cdot A_{\e}\nabla u_{\e}\,=\,f.
\end{equation*}
The order of the limits in $\rho$ and $T$ is irrelevant in \eqref{eq:coro-reg} and \eqref{eq:coro-reg2}.
\end{corollary}
\noindent As mentioned in the introduction, by numerical homogenization we mean not only the approximation of the mean-field solution but also the approximation of the local fluctuations. These are the object of the following corrector result.
\begin{defi}\label{def:corr-reg}
Let $H>0$, $I_H\in \N$, and let $\{Q_{H,i}\}_{i\in [\![1,I_H]\!]}$ be a covering of $D$ in disjoint subdomains of diameter of order $H$.
We define a family $(M_H)$ of approximations of the identity on $L^2(D)$ associated with
$Q_{H,i}$: for every $w\in L^2(D)$ and $H>0$,
$$M_{H,i}(w)=\fint_{Q_{H,i}\cap D}w, \quad M_H(w)=\sum_{i=1}^{I_H} M_{H,i}(w)1_{Q_{H,i}\cap D}. $$
Let $\delta>1$ be as in Theorem~\ref{th:main-reg}, and for all $i\in [\![1,I_H]\!]$, set
$$
Q_{H,i}^\delta:=\{x\in \R^d\,|\,d(x,Q_{H,i})<(\delta-1)H\}.
$$
With the notation of Corollary~\ref{coro:reg}, we define the numerical correctors 
$\gamma_{T,k,1,\rho,\e}^{\delta,H,i}$ associated with $u_{T,k,\rho,\e}^\delta$  as
the unique weak solution in $H^1_0(Q_{H,i}^\delta\cap D)$ of
\begin{equation}\label{eq:def-corr-reg}
(T\e^2)^{-1}\gamma_{T,k,1,\rho,\e}^{\delta,H,i} -\nabla \cdot A_\e \bigg( M_{H,i}(\nabla u_{T,k,\rho,\e}^\delta)+\nabla \gamma_{T,k,1,\rho,\e}^{\delta,H,i} \bigg)\,=\,0,
\end{equation}
which generates $\gamma_{T,k,k',\rho,\e}^{\delta,H,i}$ for all $k'>1$ by Richardson extrapolation, cf. Definition~\ref{def:extra}.
We then set for all $k' \in \N$ and for all $1\leq i\leq I_H$,
\begin{equation*}
\nabla u_{T,k,k',\rho,\e}^{\delta,H,i}\,:=\,M_H(\nabla u_{T,k,\rho,\e}^\delta)|_{Q_{H,i}\cap D}+(\nabla \gamma_{T,k,k',\rho,\e}^{\delta,H,i}) |_{Q_{H,i}\cap D} \in L^2(Q_{H,i}\cap D),
\end{equation*}
and finally define the numerical corrector
\begin{equation*}
C^{\delta,H}_{T,k,k',\rho,\e}\,=\,\sum_{i=1}^{I_H}\nabla   u_{T,k,k',\rho,\e}^{\delta,H,i} 1_{Q_{H,i}\cap D}.
\end{equation*}
\end{defi}
\begin{rem}
For the definition of the numerical corrector, the extrapolation parameter $k'$ needs not be the same as the extrapolation parameter $k$ used for $A_{T,k,\rho,\e}^\delta$. Indeed, as shown in Paragraph~\ref{sec:Num.Anal.loc_per} in the case of periodic coefficients, it makes sense to take $k'=2k$ in terms of scaling.
\end{rem}
\noindent The following numerical corrector result holds:
\begin{theo}\label{th:corr-reg}
Under the assumptions of  Corollary~\ref{coro:reg}, the corrector of Definition~\ref{def:corr-reg}
satisfies for all $k'\in \N$,
\begin{equation}\label{eq:corr-conv-reg}
\limsup_{\rho \downarrow 0,T\uparrow \infty,H\downarrow 0} \limsup_{\e \downarrow 0}\Big|\!\Big|\nabla u_\e-C^{\delta,H}_{T,k,k',\rho,\e}\Big|\!\Big|_{L^2(D)}=0
\end{equation}
almost surely, where the order of the limits in $\rho,T$, and $H$ is irrelevant.
\end{theo}
\noindent We first prove Theorem~\ref{th:main-reg}, and then turn to the proof of~Theorem~\ref{th:corr-reg}.

\medskip

\begin{proof}[Proof of Theorem~\ref{th:main-reg}]
We split the proof into three steps, and assume w.~l.~o.~g. that $|\xi|=|\xi'|=1$.
In the first step we prove that $A_{T,k,\rho,\e}^\delta$ is uniformly bounded for general coefficients, uniformly coercive on an open set of the parameters provided \eqref{eq:thmain-2-reg} holds, and uniformly coercive if the coefficients are symmetric. 
Steps~2 and 3 are dedicated to the proof of \eqref{eq:thmain-2-reg}.

\medskip

\step{1} Uniform ellipticity of $A_{T,k,\rho,\e}^\delta$.

\noindent 
We start with the uniform boundedness.
Since $P_{\rho,x}$ is the $L^2(Q_\rho\cap T_{-x}D)$-projection on $\mu_\rho$-mean free functions, for 
all $\psi \in L^2(Q_\rho\cap T_{-x}D)$, 
\begin{equation}\label{eq:L2projection}
\int_{Q_\rho\cap T_{-x} D} (P_{x,\rho}\psi)^2(y)\mu_\rho(y)dy \,\leq \,
\int_{Q_\rho\cap T_{-x} D} \psi^2(y)\mu_\rho(y)dy.
\end{equation}
By uniform ellipticity of $A_\e$ and Young's inequality,
\begin{eqnarray*}
 |\xi'\cdot A_{T,k,\rho,\e}^\delta(x)\xi|
&\leq & \frac{1}{2}\beta  \int_{Q_{\rho}\cap T_{-x}D}|\xi+P_{\rho,x}(\nabla  v^{\delta\rho,\e}_{T,k})(x,y)|^2\mu_\rho(y)dy \\
&&+\frac{1}{2}\beta\int_{Q_{\rho}\cap T_{-x}D}|\xi'+P_{\rho,x}(\nabla  {v'}^{\delta\rho,\e}_{T,k})(x,y)|^2\mu_\rho(y)dy .
\end{eqnarray*}
Both terms of the RHS are similar and we only treat the first one.
Expanding the square, the cross-terms vanish identically by definition of $P_{\rho,x}$, and we 
are left with
\begin{multline*}
\int_{Q_{\rho}\cap T_{-x}D}|\xi+P_{\rho,x}(\nabla  v^{\delta\rho,\e}_{T,k})(x,y)|^2\mu_\rho(y)dy
\\=\, |\xi|^2 \int_{Q_{\rho}\cap T_{-x}D} \mu_\rho(y)dy+\int_{Q_{\rho}\cap T_{-x}D}|P_{\rho,x}(\nabla  v^{\delta\rho,\e}_{T,k})(x,y)|^2\mu_\rho(y)dy.
\end{multline*}
By \eqref{eq:L2projection} and by definition  of $\mu_\rho$, we have by an a priori estimate
\begin{multline}\label{eq:arg-control-proj}
\int_{Q_{\rho}\cap T_{-x}D}|P_{\rho,x}(\nabla  v^{\delta\rho,\e}_{T,k})(x,y)|^2\mu_\rho(y)dy\,=\,
\int_{Q_{\rho}\cap T_{-x}D}|P_{\rho,x}(\xi+\nabla  v^{\delta\rho,\e}_{T,k})(x,y)|^2\mu_\rho(y)dy
\\
\leq\,\int_{Q_{\rho}\cap T_{-x}D}|\xi+\nabla  v^{\delta\rho,\e}_{T,k}(x,y)|^2\mu_\rho(y)dy \,
\lesssim \, \delta^d |\xi|^2=\delta^d.
\end{multline}
This proves that 
$$
|\xi'\cdot A_{T,k,\rho,\e}^\delta(x)\xi| \,\lesssim\, \delta^d.
$$

\medskip

\noindent We turn now to the uniform coercivity for symmetric coefficients.
By regularity of the domain $D$, there exists $\gamma>0$ such that for all $x\in D$,
\begin{equation}\label{eq:def-gamma}
\int_{Q_{\rho}\cap T_{-x}D}\mu_\rho \, \geq\, \gamma.
\end{equation}
By uniform ellipticity of $A_\e$, and then expanding the square, we obtain
\begin{eqnarray*}
 \xi\cdot A_{T,k,\rho,\e}^\delta(x)\xi
&\geq & \alpha \int_{Q_{\rho}\cap T_{-x}D}|\xi+P_{\rho,x}(\nabla  v^{\delta\rho,\e}_{T,k})(x,y)|^2\mu_\rho(y)dy \\
&= & \alpha \int_{Q_{\rho}\cap T_{-x}D}|\xi|^2 \mu_\rho(y)dy + \alpha \int_{Q_{\delta \rho}\cap T_{-x}D}|P_{\rho,x}(\nabla  v^{\delta\rho,\e}_{T,k})(x,y)|^2 \mu_\rho(y)\\
&& +2 \alpha \xi\cdot \underbrace{\int_{Q_{\rho}\cap T_{-x}D} P_{\rho,x}(\nabla  v^{\delta\rho,\e}_{T,k})(x,y)\mu_\rho(y)dy}_{\displaystyle \stackrel{\eqref{def:mean-free-mu}}{=}\,0} \\
&\stackrel{\eqref{eq:def-gamma}}{\geq} & \gamma \alpha |\xi|^2\,=\,\gamma \alpha .
\end{eqnarray*}

\medskip

\noindent We conclude by the uniform coercivity for non-symmetric coefficients provided \eqref{eq:thmain-2-reg} holds. By the regularity of $D$, for all $\delta \ge 1$, there exists a Lipschitz constant $\kappa_\delta$ such that $x\mapsto A_{T,k,\rho,\e}^\delta$ is $\kappa_\delta$-Lipschitz on $D$ uniformly w.~r.~t $T,k,\rho,\e$. Recall that $A_\ho$ is uniformly coercive, say with constant $\alpha'$.
Choose $\bar \rho$ small enough so that $\kappa_\delta \sqrt{d}\bar \rho\leq \frac{1}{4}\alpha'$, and consider a finite set of points $\{x_i\}_{i\in I}$ such that $D \subset \cup_{i\in I} Q_{\bar \rho}(x_i)$.
By \eqref{eq:thmain-2-reg}, for all $i\in I$ there exists a regime of parameters (say a diagonal extraction) 
with $\rho\le \bar \rho$ for which $|A_{T,k,\rho,\e}^\delta(x_i)-A_\ho(x_i)|\leq \frac{1}{4}\alpha'$. Since the set $I$ is finite, one can take the intersection, which proves there exists 
a global regime of parameters for which $A_{T,k,\rho,\e}^\delta$ is $\frac{1}{2}\alpha'$-coercive on $D$.

\medskip

\step{2} From local ergodicity to ergodicity.

\noindent 
In this step we introduce a proxy $\tilde A_{T,k,\rho,\e}^\delta(x)$ for $A^\delta_{T,k,\rho,\e}(x)$ which is uniformly close to $A^\delta_{T,k,\rho,\e}(x)$ in $\rho$, and for which we can use 
the results of Section~\ref{sec:spectral}.
For all $x\in D$, we define $\tilde A^\delta_{T,k,\rho,\e}(x)$ by: 
\begin{multline*}
\xi' \cdot \tilde A_{T,k,\rho,\e}^\delta(x)\xi :=\int_{Q_{\rho}\cap T_{-x}D} (\xi' +P_{\rho,x}(\nabla_y {w'}^{\delta\rho,\e}_{T,k})(x,y))\\
\cdot \tilde A(x,\frac{x+y}{\e})(\xi+P_{\rho,x}(\nabla_y w^{\delta\rho,\e}_{T,k})(x,y))\mu_{\rho}dy,
\end{multline*}
where $w^{\delta\rho,\e}_{T,k}(x,\cdot),{w'}^{\delta\rho,\e}_{T,k}(x,\cdot)$ are (for $k>1$) the Richardson extrapolations of the unique weak solutions $w^{\delta\rho,\e}_{T,1}(x,\cdot),{w'}^{\delta\rho,\e}_{T,1}(x,\cdot)\in H^1_0(\Rhodrho)$ of
\begin{eqnarray*}
(T\e^2)^{-1} w^{\delta\rho,\e}_{T,1}(x,y)-\nabla \cdot \tilde A(x,\frac{x+y}{\e})(\xi+\nabla_y w^{\delta\rho,\e}_{T,1}(x,y))&=&0,\\
(T\e^2)^{-1} {w'}^{\delta\rho,\e}_{T,1}(x,y)-\nabla \cdot \tilde A^*(x,\frac{x+y}{\e})(\xi'+\nabla_y {w'}^{\delta\rho,\e}_{T,1}(x,y))&=&0.
\end{eqnarray*}
We shall prove that 
\begin{equation}\label{eq:local-ergo}
|A_{T,k,\rho,\e}^\delta(x)-\tilde A_{T,k,\rho,\e}^\delta(x)|\,\lesssim \, \rho
\end{equation}
uniformly in $x$, $T$, $\e$, the randomness, and where the multiplicative constant depends on $k$ next
to $\alpha,\beta$ and $d$.
We rewrite the difference as
\begin{eqnarray*}
\lefteqn{\xi'\cdot (A_{T,k,\rho,\e}^\delta(x)-\tilde A_{T,k,\rho,\e}^\delta(x))\xi }\\
&=&\int_{Q_{\rho}\cap T_{-x}D} (\xi' +P_{\rho,x}(\nabla_y {v'}^{\delta\rho,\e}_{T,k})(x,y))\cdot A_\e(x+y)(\xi+P_{\rho,x}(\nabla_y v^{\delta\rho,\e}_{T,k})(x,y))\mu_{\rho}(y)dy \\
&&-\int_{Q_{\rho}\cap T_{-x}D} (\xi' +P_{\rho,x}(\nabla_y {w'}^{\delta\rho,\e}_{T,k})(x,y))\cdot \tilde A(x,\frac{x+y}{\e})(\xi+P_{\rho,x}(\nabla_y w^{\delta\rho,\e}_{T,k})(x,y))\mu_{\rho}(y)dy
\\
&=&\int_{Q_{\rho}\cap T_{-x}D} (\xi' +P_{\rho,x}(\nabla_y {v'}^{\delta\rho,\e}_{T,k})(x,y))\cdot \tilde A(x,\frac{x+y}{\e})(\xi+P_{\rho,x}(\nabla_y v^{\delta\rho,\e}_{T,k})(x,y))\mu_{\rho}(y)dy\\
&&-\int_{Q_{\rho}\cap T_{-x}D} (\xi'+P_{\rho,x}(\nabla_y {w'}^{\delta\rho,\e}_{T,k})(x,y))\cdot \tilde A(x,\frac{x+y}{\e})(\xi+P_{\rho,x}(\nabla_y w^{\delta\rho,\e}_{T,k})(x,y))\mu_{\rho}(y)dy\\
&&+\int_{Q_{\rho}\cap T_{-x}D} (\xi' +P_{\rho,x}(\nabla_y {v'}^{\delta\rho,\e}_{T,k})(x,y)) \\
&& \qquad \qquad \qquad \cdot (A_\e(x+y)-\tilde A(x,\frac{x+y}{\e}))(\xi+P_{\rho,x}(\nabla_y v^{\delta\rho,\e}_{T,k})(x,y))\mu_{\rho}(y)dy.
\end{eqnarray*}
Using that 
\begin{equation}\label{eq:use-Lip}
|A_\e(x+y)-\tilde A(x,\frac{x+y}{\e})|=|\tilde A(x+y,\frac{x+y}{\e})-\tilde A(x,\frac{x+y}{\e})|\leq \kappa |y|,
\end{equation}
that $|\xi'|=|\xi|=1$, and the estimate \eqref{eq:arg-control-proj}, we may bound 
the third term of the RHS:
\begin{multline}
\int_{Q_{\rho}\cap T_{-x}D} (\xi' +P_{\rho,x}(\nabla_y {v'}^{\delta\rho,\e}_{T,k})(x,y)) \\
 \cdot (A_\e(x+y)-\tilde A(x,\frac{x+y}{\e}))(\xi+P_{\rho,x}(\nabla_y v^{\delta\rho,\e}_{T,k})(x,y))\mu_{\rho}(y)dy.
\\ \,\leq \,
\kappa \rho \delta^d
\,\lesssim\, \rho  \label{eq:local-ergo-er1}.
\end{multline}
We now turn to the first two terms, which we write in the form
\begin{eqnarray*}
&&\int_{Q_{\rho}\cap T_{-x}D} (\xi'+P_{\rho,x}(\nabla_y {v'}^{\delta\rho,\e}_{T,k})(x,y))\cdot \tilde A(x,\frac{x+y}{\e})(\xi+P_{\rho,x}(\nabla_y v^{\delta\rho,\e}_{T,k})(x,y))\mu_{\rho}(y)dy\\
&&-\int_{Q_{\rho}\cap T_{-x}D} (\xi' +P_{\rho,x}(\nabla_y {w'}^{\delta\rho,\e}_{T,k})(x,y))\cdot \tilde A(x,\frac{x+y}{\e})(\xi+P_{\rho,x}(\nabla_y w^{\delta\rho,\e}_{T,k})(x,y))\mu_{\rho}(y)dy\\
&=&\int_{Q_{\rho}\cap T_{-x}D} P_{\rho,x}(\nabla_y {v'}^{\delta\rho,\e}_{T,k}-\nabla_y {w'}^{\delta\rho,\e}_{T,k})(x,y)\\
&& \qquad \qquad\qquad \qquad \cdot \tilde A(x,\frac{x+y}{\e})(\xi+P_{\rho,x}(\nabla_y v^{\delta\rho,\e}_{T,k})(x,y))\mu_{\rho}(y)dy \\
&&+\int_{Q_{\rho}\cap T_{-x}D} (\xi' +P_{\rho,x}(\nabla_y {v'}^{\delta\rho,\e}_{T,k})(x,y)) \\
&& \qquad \qquad\qquad \qquad \cdot \tilde A(x,\frac{x+y}{\e})P_{\rho,x}(\nabla_y {v}^{\delta\rho,\e}_{T,k}-\nabla_y w^{\delta\rho,\e}_{T,k})(x,y)\mu_{\rho}(y)dy,
\end{eqnarray*}
so that by the continuity of $P_{\rho,x}$ on $L^2(Q_{\rho}\cap T_{-x}D)$ and using \eqref{eq:arg-control-proj}, it enough to prove that
\begin{equation}\label{eq:T-sansT-corr}
\int_{\Rhodrho} |\nabla_y {v}^{\delta\rho,\e}_{T,k}(x,y)-\nabla_y w^{\delta\rho,\e}_{T,k}(x,y)|^2dy\,\lesssim \, \rho^2
\end{equation}
in order to deduce \eqref{eq:local-ergo} from \eqref{eq:local-ergo-er1} (the argument for the dual correctors being similar).
By definition of the Richardson extrapolation, this estimate directly follows from the Lipschitz bound \eqref{eq:use-Lip}  and an energy estimate on the equation satisfied by  ${v}^{\delta\rho,\e}_{T,1}(x,\cdot)- w^{\delta\rho,\e}_{T,1}(x,\cdot)$:
\begin{multline*}
(T\e^2)^{-1}({v}^{\delta\rho,\e}_{T,1}(x,y)-w^{\delta\rho,\e}_{T,1}(x,y))-\nabla_y \tilde A(x,\frac{x+y}{\e}) \nabla_y ({v}^{\delta\rho,\e}_{T,1}(x,y)-\nabla_y w^{\delta\rho,\e}_{T,1}(x,y)) \\
\,=\,\nabla_y \cdot (A_\e(x+y)-\tilde A(x,\frac{x+y}{\e})) \nabla_y {v}^{\delta\rho,\e}_{T,1}(x,y).
\end{multline*}

\medskip

\step{3} Limits $\rho\downarrow 0$, $T\uparrow \infty$ and $\e \downarrow 0$.

\noindent For all $x\in D$, there exists $\overline \rho(x)>0$ such that $\Rhodrho=Q_{\delta\rho}$
for all $\rho\leq \overline \rho(x)$. Assume that $\Rhodrho=Q_{\delta\rho}$.
In that case, the change of variables $z=\frac{y}{\e}$ allows one to interprete 
$\tilde A_{T,k,\rho,\e}^\delta(x)$ as the approximation 
``$A_{T,k,R,L,p}(x)$'' of $A_\ho(x)$ defined in \eqref{eq:approx-ATKRLp-v} (with $\tilde A(x,\cdot)$ in place of $A(\cdot)$), with the parameters
\begin{equation}\label{eq:equiv-R-eps-L-delta}
R\,=\,\delta \frac{\rho}{\e},\qquad L\,=\,\frac{\rho}{\e}, \qquad \text{ and }p\in \N_0.
\end{equation}
By Theorem~\ref{th:ATkRLp-approx} we then have the almost sure convergence
\begin{equation}\label{eq:lim-eps-reg}
\lim_{\e\downarrow 0} \tilde A_{T,k,\rho,\e}^\delta(x)\,=\, \tilde A_{T,k}(x),
\end{equation}
where $x$ is seen as a parameter, and $\tilde A_{T,k}(x)$ is associated with
the stationary coefficients $y\mapsto \tilde A(x,y)$.
Note that the limit does not depend on $\rho$, as expected by stationarity of $\tilde A(x,y)$ in $y$.

\medskip

\noindent Theorem~\ref{th:spec} allows one to take the limit $T\uparrow \infty$, which yields for all $x\in D$,
\begin{equation}\label{eq:lim-Tk-reg}
\lim_{T\uparrow \infty}\tilde A_{T,k}(x)\,=\,A_\ho(x).
\end{equation}
Since for all $x\in D$, $\overline \rho(x)>0$ (where the function $\overline \rho$ is defined at the beginning of this step), the combination of \eqref{eq:local-ergo}, \eqref{eq:lim-eps-reg}, and \eqref{eq:lim-Tk-reg}
shows that for all $x\in D$
$$
\limsup_{T\uparrow \infty}\limsup_{\rho\downarrow 0}\limsup_{\e \downarrow 0} |A_{T,k,\rho,\e}^\delta(x)-A_\ho(x)|\,=\,0,
$$
and concludes the proof of \eqref{eq:thmain-2-reg}.

\medskip

\noindent 
Note that for any domain $\tilde D$ compactly included in $D$, $\inf_{x \in \tilde D} \overline \rho(x)>0$.
If $\rho \leq \overline \rho(x)$, the limit $\tilde A_{T,k}(x)$ of $\tilde A_{T,k,\rho,\e}^\delta(x)$ as $\e\downarrow 0$ exists almost surely and does not depend on $\rho$.
Hence, for all $x\in \tilde D$, if $\rho\leq \inf_{x \in \tilde D} \overline \rho(x)$ we thus have
using \eqref{eq:lim-eps-reg} and \eqref{eq:local-ergo},
\begin{eqnarray*}
\limsup_{\e \downarrow 0}|A_{T,k,\rho,\e}^\delta(x)-A_\ho(x)|&\leq& \lim_{\e \downarrow 0}|\tilde A_{T,k,\rho,\e}^\delta(x)-A_\ho(x)|\\
&& \qquad +\limsup_{\e \downarrow 0}|A_{T,k,\rho,\e}^\delta(x)-\tilde A_{T,k,\rho,\e}^\delta(x)| 
\\
&\leq &|\tilde A_{T,k}(x)-A_\ho(x)|+\rho,
\end{eqnarray*}
and the uniformity of the convergence in $T$ w.~r.~t. $\rho\leq \inf_{x \in \tilde D} \overline \rho(x)$ follows.

\medskip
\noindent
We prove conclude this step with the proof of \eqref{eq:thmain-2-reg2}.
For all $\rho>0$, set $D_\rho:=\{x\in D\,|\,\rho< \overline{\rho}(x)\}$.
Since $D$ has a Lipschitz boundary, 
\begin{equation}\label{eq:DrhoD}
\lim_{\rho\downarrow 0} |D\setminus D_\rho|\,=\,0.
\end{equation}
Let $1\leq p<\infty$.
By the uniform boundedness of $A_\ho$ and $A_{T,k,\rho,\e}^\delta$ and the definition of $D_\rho$, we have
\begin{eqnarray}
\limsup_{\e\downarrow 0} \|A_{T,k,\rho,\e}^\delta-A_\ho\|_{L^p(D)} &\leq& \limsup_{\e \downarrow 0}\|A_{T,k,\rho,\e}^\delta-A_\ho\|_{L^p(D_\rho)}\nonumber
\\
&& \qquad +\limsup_{\e \downarrow 0}\|A_{T,k,\rho,\e}^\delta-A_\ho\|_{L^p(D\setminus D_\rho)} \nonumber\\
&\lesssim & \|\tilde A_{T,k}-A_\ho\|_{L^p(D)}+\rho+|D\setminus D_\rho|^{\frac{1}{p}},\label{eq:useful-proof-coro}
\end{eqnarray}
and the conclusion follows from \eqref{eq:DrhoD}, \eqref{eq:lim-Tk-reg}, and the Lebesgue dominated convergence theorem. 
\end{proof}

\medskip

\noindent Corollary~\ref{coro:reg} is a consequence of Theorem~\ref{th:main-reg} (using in particular \eqref{eq:useful-proof-coro}) and of the following lemma:
\begin{lemma}\label{6:lemma:ptwise-conv-sol}
Let $(A_{h_1,h_2})_{h_1,h_2>0}$ be uniformly elliptic for all $h_1,h_2$ small enough, $A\in \Mab(D)$, and $(f_{h_1,h_2})_{h_1,h_2>0},f \in H^{-1}(D)$ be such that 
$f_{h_1,h_2} \to f$ in $H^{-1}(D)$ as $h_1,h_2\downarrow 0$ (in any order), and 
\begin{itemize}
\item For all $x\in D$, $\lim_{h_1\downarrow 0}\lim_{h_2\downarrow 0} A_{h_1,h_2}(x)\,=\,A(x)$;
\item There exist two bounded functions $g_1,g_2:\R^+\times D\to \R^+$ such that for all $x\in D$, $\lim_{h\to 0} g_1(x,h)=\lim_{h\to 0}g_2(x,h)=0$, and for all $h_2>0$ there exists a measurable subset $D_{h_2}$ of $D$ such that
$\lim_{h_2\to 0}|D\setminus D_{h_2}|=0$ and such that for all $x\in D_{h_2}$ and all $h_1>0$,
$$
|A_{h_1,h_2}(x)-A(x)|\,\leq \, g_1(x,h_1)+g_2(x,h_2).
$$
\end{itemize}
Then the unique weak solution $u_{h_1,h_2} \in H^1_0(D)$ of
\begin{equation*}
-\nabla \cdot A_{h_1,h_2} \nabla u_{h_1,h_2} =f_{h_1,h_2}
\end{equation*}
converges in $H^1(D)$ as $h_1,h_2\downarrow 0$ to the unique weak solution $u$ in $H^1_0(D)$ of
\begin{equation*}
-\nabla \cdot A \nabla u =f,
\end{equation*}
and the order of convergence of $h_1$ and $h_2$ is irrelevant.
\end{lemma}
\begin{proof}[Proof of Lemma~\ref{6:lemma:ptwise-conv-sol}]
Substract the weak forms of the two equations tested with the admissible test-function $u_{h_1,h_2}-u\in H^1_0(D)$. This yields
\begin{equation*}
\int_D \nabla (u_{h_1,h_2}-u) \cdot (A_{h_1,h_2} \nabla u_{h_1,h_2}-A\nabla u)=\Big( f_{h_1,h_2}-f,u_{h_1,h_2}-u\Big)_{H^{-1}(D),H^1_0(D)},
\end{equation*}
where $\big(\cdot,\cdot\big)_{H^{-1}(D),H^1_0(D)}$ denotes the duality pairing between $H^{-1}(D)$ and $H^1_0(D)$. We rewrite this equation in the form
\begin{multline}\label{eq:ref-pr-HMM-corr-disc}
\int_D \nabla (u_{h_1,h_2}-u) \cdot A_{h_1,h_2} \nabla (u_{h_1,h_2}-u)\\
=-\int_D \nabla(u_{h_1,h_2}-u)\cdot  (A_{h_1,h_2}-A)\nabla u  +\Big( f_{h_1,h_2}-f,u_{h_1,h_2}-u\Big)_{H^{-1}(D),H^1_0(D)}.
\end{multline}
Using the ellipticity of $A_{h_1,h_2}$, Cauchy-Schwarz', Young's, and Poincar\'e's inequalities, this turns into
\begin{equation*}
\|\nabla (u_{h_1,h_2}-u)\|^2_{L^2(D)} 
\,\lesssim \, \int_D |\nabla u \cdot (A_{h_1,h_2}-A)\nabla u|+\|f_{h_1,h_2}-f\|_{H^{-1}(D)}.
\end{equation*}
By assumption, the second term of the RHS goes to zero as $h_1,h_2\downarrow 0$.
It remains to treat the first term. We split the integral into two terms:
\begin{eqnarray*}
\int_D |\nabla u \cdot (A_{h_1,h_2}-A)\nabla u| & =& \int_{D_{h_2}} |\nabla u \cdot (A_{h_1,h_2}-A)\nabla u|+\int_{D\setminus D_{h_2}} |\nabla u \cdot (A_{h_1,h_2}-A)\nabla u| \\
&\leq &\int_{D} |\nabla u|^2 (g_1(x,h_1)+g_2(x,h_2))+2\beta \int_{D\setminus D_{h_2}} |\nabla u|^2.
\end{eqnarray*}
The first term of the RHS converges to zero as $h_1,h_2\downarrow 0$ by the Lebesgue dominated convergence theorem, while the second term converges to zero as $h_2\downarrow 0$ because $|D\setminus D_{h_2}|\downarrow 0$ and $\nabla u \in L^2(D)$. This concludes the proof of the lemma.
\end{proof}

\medskip

\noindent The proof of Theorem~\ref{th:corr-reg}, which is rather long and technical, is somewhat counter-intuitive: Although we expect the regularization and extrapolation method to enhance the convergence, our proof rather shows
that the method does not worsen the convergence of the standard method based on Dirichlet boundary conditions (without regularization and extrapolation) analyzed in \cite{Gloria-06b}.
The proof has essentially two parts: first we prove the numerical corrector result with the standard approximation of the corrector, and we then prove that the difference between that approximation and the one with regularization and extrapolation vanishes in a suitable sense.
For the first part, the main idea is to use Tartar's correctors on each element $Q_{H,i}\cap D$ of the partition of $D$, pass to the limit in $\e$ first, and then in $H$.
The second part of the proof makes crucial use of the ergodic theorem and of Theorem~\ref{th:spec}.
\begin{proof}[Proof of Theorem~\ref{th:corr-reg}]

We divide the proof into four steps.

\medskip

\step{1} Standard corrector result.

\noindent For all $1\leq i\leq I_H$ we define $\gamma_{\e}^{H,i}$ as the
unique weak solution in $H^1_0(Q_{H,i}\cap D)$ of
\begin{equation}\label{eq:corr-rev1}
-\nabla \cdot A_{\e} (M_H(\nabla u_{\ho})+ \nabla \gamma_{\e}^{H,i})\,=\,0,
\end{equation}
that we extend by zero on $D\setminus Q_{H,i}$ as a function of $H^1_0(D)$.
We then define the following variant of Tartar's corrector:
\begin{equation*}
C_{\e}^{H}\,:=\,\sum_{i=1}^{I_H}\nabla u_{\e}^{H,i},
\end{equation*}
where $\nabla u_{\e}^{H,i} \in L^2(D)$ has support in $Q_{H,i}\cap D$ and is given by
\begin{equation}
\nabla u_{\e}^{H,i}\,:=\, \Big( M_H(\nabla u_{\ho})+ \nabla \gamma_{\e}^{H,i} \Big)1_{Q_{H,i}\cap D}\label{eq:corr-rev7}.
\end{equation}
The aim of this step is to prove that almost surely
\begin{equation}\label{eq:standard-corr}
\lim_{H\downarrow 0}\lim_{\e \downarrow 0} \|\nabla u_\e - C_{\e}^{H}\|_{L^2(D)}\,=\,0.
\end{equation}
As we shall see this result is deterministic and follows from H-convergence.

\medskip
\noindent By ellipticity of $A_\e$, 
\begin{multline}\label{eq:dble-prod}
\|\nabla u_\e - C_{\e}^{H}\|_{L^2(D)}^2\,\lesssim\,\int_D (\nabla u_\e-C_{\e}^{H})\cdot A_\e (\nabla u_\e-C_{\e}^{H}) 
\\
=\,\int_D\nabla u_\e \cdot A_\e \nabla u_\e+\int_D C_{\e}^{H}\cdot A_\e \nabla C_{\e}^{H}
\\
-\sum_{i=1}^{I_H}\int_{D} \nabla u_\e\cdot A_\e\nabla u_{\e}^{H,i} -\sum_{i=1}^{I_H}\int_{D} \nabla u_{\e}^{H,i} \cdot A_\e\nabla u_\e.
\end{multline}
We first argue that 
\begin{equation}\label{eq:standard-corr-1}
\lim_{H\downarrow 0}\lim_{\e \downarrow 0} 
\sum_{i=1}^{I_H}\int_{D} \nabla u_{\e}^{H,i} \cdot A_\e\nabla u_\e\,=\, \int_D \nabla u_\ho\cdot A_\ho \nabla u_\ho.
\end{equation}
Since $\gamma_{\e}^{H,i} \in H^1_0(D)$, this term takes the form
$$
\sum_{i=1}^{I_H}\int_{D} \nabla u_{\e}^{H,i} \cdot A_\e\nabla u_\e\,=\,\sum_{i=1}^{I_H} \Big(f,\gamma_{\e}^{H,i}\Big)_{H^{-1}(D),H^1_0(D)}+\int_{D} M_H(\nabla u_{\ho})\cdot A_\e\nabla u_\e
$$
using the weak form of the defining equation for $u_\e$.
On the one hand, since $A_\e\nabla u_\e \rightharpoonup A_\ho \nabla u_\ho$ weakly in $L^2(D,\R^d)$
by homogenization
and $M_H$ converges to the identity on $L^2(D,\R^d)$, 
\begin{equation}\label{eq:standard-corr-2}
\lim_{H\downarrow 0} \lim_{\e \downarrow 0} \int_{D} M_H(\nabla u_{\ho})\cdot A_\e\nabla u_\e \,=\,\int_D \nabla u_\ho\cdot A_\ho \nabla u_\ho.
\end{equation}
On the other hand, homogenization ensures that $\gamma_{\e}^{H,i} \stackrel{\e \downarrow 0}{\to} \gamma_{\ho}^{H,i}$ weakly in $H^1_0(Q_{H,i}\cap D)$ (and therefore weakly in 
$H^1_0(D)$) almost surely, where  $\gamma_{\ho}^{H,i}$ is the unique weak solution in $H^1_0(Q_{H,i}\cap D)$ (extended by zero as
a $H^1_0(D)$ function on $D\setminus Q_{H,i}$) of
$$
-\nabla \cdot A_{\ho} (M_H(\nabla u_{\ho})+ \nabla \gamma_{\ho}^{H,i})\,=\,0.
$$
Since $A_\ho$ is $\kappa$-Lipschitz, for all $x\in Q_{H,i}\cap D$, $|A_\ho(x)-M_{H,i}(A_\ho)|\lesssim \kappa H$.
Hence, an energy estimate yields
\begin{eqnarray*}
\int_{Q_{H,i}\cap D} \nabla \gamma_{\ho}^{H,i}\cdot A_\ho \nabla \gamma_{\ho}^{H,i}
&=&\int_{Q_{H,i}\cap D} \nabla \gamma_{\ho}^{H,i}\cdot A_\ho M_H(\nabla u_{\ho})
\\ 
&=&
\Big(\int_{Q_{H,i}\cap D} \nabla \gamma_{\ho}^{H,i}\Big)\cdot M_{H,i}(A_\ho)M_{H,i}(\nabla u_{\ho}) \\
&&+\int_{Q_{H,i}\cap D} \nabla \gamma_{\ho}^{H,i}\cdot (A_\ho -M_H(A_\ho))M_H(\nabla u_{\ho})  \\
&\lesssim& \kappa H \|\nabla \gamma_{\ho}^{H,i}\|_{L^2(Q_{H,i}\cap D)} \|\nabla u_\ho\|_{L^2(Q_{H,i}\cap D)} ,
\end{eqnarray*}
since $\int_{Q_{H,i}\cap D} \nabla \gamma_{\ho}^{H,i}=0$ and $M_{H,i}$ is a contraction.
Setting $\gamma_\ho^H=\sum_{i=1}^{I_H} \gamma_\ho^{H,i}1_{Q_{H,i}\cap D}\in H^1_0(D)$, this implies by Poincar\'e's inequality
\begin{equation}\label{eq:gamma-to-0}
\|\gamma_{\ho}^H\|_{H^1(D)}\,\lesssim \, H \|\nabla u_\ho\|_{L^2(D)} \stackrel{H\downarrow 0}{\longrightarrow} 0,
\end{equation}
so that
\begin{equation}\label{eq:standard-corr-3}
\lim_{H\downarrow 0} \lim_{\e \downarrow 0} \sum_{i=1}^{I_H} \Big(f,\gamma_{\e}^{H,i}\Big)_{H^{-1}(D),H^1_0(D)} \,=\, \lim_{H\downarrow 0}\Big(f,\gamma_{\ho}^{H}\Big)_{H^{-1}(D),H^1_0(D)} \,=\,0.
\end{equation}
Estimate~\eqref{eq:standard-corr-1} then follows from \eqref{eq:standard-corr-2} and \eqref{eq:standard-corr-3}.

\medskip

\noindent Next, we prove that 
\begin{equation}\label{eq:standard-corr-1bis}
\lim_{H\downarrow 0}\lim_{\e \downarrow 0} 
\sum_{i=1}^{I_H}\int_{D} \nabla u_{\e} \cdot A_\e\nabla u_\e^{H,i}\,=\, \int_D \nabla u_\ho\cdot A_\ho \nabla u_\ho.
\end{equation}
By compensated compactness on each $Q_{H,i}$, $\nabla u_\e\cdot A_\e\nabla u_\e^{H,i}$ converges  in the sense of distributions to 
$\nabla u_\ho\cdot A_\ho(M_{H,i}(\nabla u_\ho)+\nabla \gamma^{H,i}_\ho)$.
To upgrade this convergence into a weak convergence in $L^1(Q_{H,i})$, we need to prove that $\nabla u_\e\cdot A_\e\nabla u_\e^{H,i}$ is equi-integrable, which follows itself from the uniform boundedness of $\nabla u_\e\cdot A_\e\nabla u_\e^{H,i}$ in $L^q(Q_{H,i})$ for some $q>1$. The latter is a consequence of Meyers' estimate, which ensures that $\nabla \gamma^{H,i}_\e$ is uniformly bounded in $L^{\tilde q}(Q_{H,i})$ for some $\tilde q>2$.
Hence,
$$
\lim_{\e \downarrow 0} 
\sum_{i=1}^{I_H}\int_{D} \nabla u_{\e} \cdot A_\e\nabla u_\e^{H,i}\,=\, \int_D \nabla u_\ho\cdot A_\ho (M_H(\nabla u_\ho)+\gamma^H_\ho),
$$
from which \eqref{eq:standard-corr-1bis} follows by convergence of $M_H$ to the identity on $L^2(D)$ and \eqref{eq:gamma-to-0}.

\medskip

\noindent We finally turn to the first two terms of the RHS of \eqref{eq:dble-prod}. Since H-convergence implies the convergence of the energy, we have 
\begin{equation}\label{eq:standard-corr-6}
\lim_{\e \downarrow 0} \int_D \nabla u_\e \cdot A_\e \nabla u_\e \,=\,\int_D \nabla u_\ho \cdot A_\ho \nabla u_\ho. 
\end{equation}
Likewise, for all $1\leq i\leq I_H$,
\begin{multline*}
\lim_{\e \downarrow 0} \int_{Q_{H,i}\cap D}(M_H(\nabla u_\ho)+\nabla \gamma_{\e}^{H,i} )\cdot A_\e (M_H(\nabla u_\ho)+\nabla \gamma_{\e}^{H,i})
\\
\,=\,\int_{Q_{H,i}\cap D}(M_H(\nabla u_\ho)+\nabla \gamma_{\ho}^{H,i} )\cdot A_\ho (M_H(\nabla u_\ho)+\nabla \gamma_{\ho}^{H,i}),
\end{multline*}
so that by summation over $i$ and defining $\gamma_\e^H:=\sum_{i=1}^{I_H} \gamma_\e^{H,i} \in H^1_0(D)$,
\begin{multline*}
\lim_{\e \downarrow 0} \int_{D}(M_H(\nabla u_\ho)+\nabla \gamma_{\e}^{H} )\cdot A_\e (M_H(\nabla u_\ho)+\nabla \gamma_{\e}^{H})
\\
\,=\,\int_{D}(M_H(\nabla u_\ho)+\nabla \gamma_{\ho}^{H} )\cdot A_\ho (M_H(\nabla u_\ho)+\nabla \gamma_{\ho}^{H}).
\end{multline*}
Since $M_H$ converges to the identity on $L^2(D)$ and $\gamma_\ho^H$ converges to zero in $H^1(D)$, this implies in particular
\begin{equation}\label{eq:standard-corr-7}
\lim_{H\downarrow 0} \lim_{\e \downarrow 0} \int_{D}(M_H(\nabla u_\ho)+\nabla \gamma_{\e}^{H} )\cdot A_\e (M_H(\nabla u_\ho)+\nabla \gamma_{\e}^{H})
\,=\,\int_{D}\nabla u_\ho\cdot A_\ho \nabla u_\ho.
\end{equation}
The combination of \eqref{eq:dble-prod}, \eqref{eq:standard-corr-1}, \eqref{eq:standard-corr-6}, and \eqref{eq:standard-corr-7}
yields \eqref{eq:standard-corr}.

\medskip

\step{2} First auxilary numerical corrector and reformulation of \eqref{eq:corr-conv-reg}.

\noindent In view of \eqref{eq:standard-corr}, it enough to prove that almost-surely for all $k'\in \N$,
\begin{equation}\label{eq:Tstandard-corr} 
\limsup_{T\uparrow \infty, H,\rho\downarrow 0} \limsup_{\e \downarrow 0} \|C_{T,k,k',\rho,\e}^{\delta,H}-C_{\e}^{H}\|_{L^2(D)}\,=\,0.
\end{equation}
We introduce the auxiliary numerical corrector $C_{T,k',\e}^{\delta,H}\in L^2(D,\R^d)$ defined as follows.
The function $\gamma_{T,1,\e}^{\delta,H,i}$ is associated with $u_\ho$  as
the unique weak solution in $H^1_0(Q_{H,i}^\delta\cap D)$ of
\begin{equation}\label{eq:def-corr-reg-hom}
(T\e^2)^{-1}\gamma_{T,1,\e}^{\delta,H,i} -\nabla \cdot A_\e \bigg( M_{H,i}(\nabla u_{\ho})+\nabla \gamma_{T,1,\e}^{\delta,H,i} \bigg)\,=\,0,
\end{equation}
and we define $\gamma_{T,k',\e}^{\delta,H,i}$ for all $k'>1$ by Richardson extrapolation, cf. Definition~\ref{def:extra}.
We then set for all $k' \in \N$ and for all $1\leq i\leq I_H$,
\begin{equation*}
\nabla u_{T,k',\e}^{\delta,H,i}\,:=\,M_H(\nabla u_{\ho})|_{Q_{H,i}\cap D}+(\nabla \gamma_{T,k',\e}^{\delta,H,i}) |_{Q_{H,i}\cap D} \in L^2(Q_{H,i}\cap D),
\end{equation*}
and finally define the numerical corrector
\begin{equation*}
C^{\delta,H}_{T,k',\e}\,=\,\sum_{i=1}^{I_H}\nabla   u_{T,k',\e}^{\delta,H,i} 1_{Q_{H,i}}.
\end{equation*}
Substracting the equations \eqref{eq:def-corr-reg} and \eqref{eq:def-corr-reg-hom} for $\gamma_{T,k,1,\rho,\e}^{\delta,H,i}$ and $\gamma_{T,1,\e}^{\delta,H,i}$, respectively, yields the following energy estimate
\begin{equation}
\|\nabla \gamma_{T,k,1,\rho,\e}^{\delta,H,i}-\nabla \gamma_{T,1,\e}^{\delta,H,i}\|_{L^2(Q_{H,i}^\delta\cap D)}^2
\,\lesssim \, |Q_{H,i}^\delta\cap D| |M_{H,i}(\nabla u_{T,k,\rho,\e}^\delta-\nabla u_\ho)|^2,
\end{equation}
so that, using that $M_H$ is a contraction on $L^2(D,\R^d)$, 
\begin{equation*}
\|C_{T,k,1,\rho,\e}^{\delta,H}-C_{T,1,\e}^{\delta,H}\|_{L^2(D)} \,\lesssim\,
\|\nabla u_{T,k,\rho,\e}^\delta-\nabla u_\ho\|_{L^2(D)}.
\end{equation*}
By definition of the Richardson extrapolation this yields for all $k'\in \N$,
\begin{equation*}
\|C_{T,k,k',\rho,\e}^{\delta,H}-C_{T,k',\e}^{\delta,H}\|_{L^2(D)} \,\lesssim\,
\|\nabla u_{T,k',\rho,\e}^\delta-\nabla u_\ho\|_{L^2(D)},
\end{equation*}
and therefore by Corollary~\ref{coro:reg}: 
\begin{equation}\label{eq:Tstandard-corr1} 
\limsup_{T\uparrow \infty,\rho \downarrow 0} \limsup_{\e \downarrow 0} \|C_{T,k,k',\rho,\e}^{\delta,H}-C_{T,k',\e}^{\delta,H}\|_{L^2(D)} \,=\,0,
\end{equation}
uniformly with respect to $H$, and almost surely.
Hence, the claim \eqref{eq:Tstandard-corr} follows from \eqref{eq:Tstandard-corr1}
provided we prove that almost-surely
\begin{equation}\label{eq:Tstandard-corr2} 
\lim_{T\uparrow \infty,H\downarrow 0} \lim_{\e \downarrow 0} \|C_{T,k',\e}^{\delta,H}-C_\e^H \|_{L^2(D)} \,=\,0.
\end{equation}

\medskip

\step{3} Further auxiliary numerical correctors and reformulation of~\eqref{eq:Tstandard-corr2}.

\noindent For all $1\leq i \leq I_H$, fix some arbitrary $x_i \in Q_{H,i}$ and set for all $x\in Q_{H,i}^\delta$, $A_{\e,i}(x):=\tilde A(x_i,\frac{x}{\e})$. The approximation of $A_\e$ by $A_{\e,i}$ on $Q_{H,i}$ (and $Q_{H,i}^\delta$) allows us to appeal to the results we proved in Sections~\ref{sec:spectral} and~\ref{sec:tests} for \emph{stationary} coefficients.
The error we make by replacing $A_\e$ by $A_{\e,i}$ is then controlled using the Lipschitz condition in Hypothesis~\ref{hypo:reg}.

\smallskip

\noindent We introduce the auxiliary numerical correctors $\tilde C_\e^H, \tilde C_{T,k',\e}^{\delta,H}\in L^2(D,\R^d)$ 
associated with $\tilde A_\e$ as follows.
The functions $\tilde \gamma_\e^{H,i}$ and $\tilde \gamma_{T,1,\e}^{\delta,H,i}$ are the unique weak solutions in $H^1_0(Q_{H,i}\cap D)$ and $H^1_0(Q_{H,i}^\delta\cap D)$, respectively, of 
\begin{eqnarray}\label{eq:def-corr-reg-hom-tilde}
-\nabla \cdot A_{\e,i} \bigg( M_{H,i}(\nabla u_{\ho})+\nabla \tilde \gamma_{\e}^{H,i} \bigg)\,=\,0, 
\\
(T\e^2)^{-1}\tilde \gamma_{T,1,\e}^{\delta,H,i} -\nabla \cdot A_{\e,i}  \bigg( M_{H,i}(\nabla u_{\ho})+\nabla \tilde \gamma_{T,1,\e}^{\delta,H,i} \bigg)\,=\,0, \label{eq:def-corr-reg-tilde}
\end{eqnarray}
and $\tilde \gamma_{T,k',\e}^{\delta,H,i}$ is defined for all $k'>1$ by Richardson extrapolation of $\tilde \gamma_{T,1,\e}^{\delta,H,i}$, cf. Definition~\ref{def:extra}.
We then set for all $k' \in \N$ and for all $1\leq i\leq I_H$,
\begin{eqnarray*}
\nabla \tilde u_{\e}^{H,i}\,:=\,M_H(\nabla u_{\ho})|_{Q_{H,i}\cap D}+(\nabla \tilde \gamma_{\e}^{H,i}) |_{Q_{H,i}\cap D} \in L^2(Q_{H,i}\cap D),\\
\nabla \tilde u_{T,k',\e}^{\delta,H,i}\,:=\,M_H(\nabla u_{\ho})|_{Q_{H,i}\cap D}+(\nabla \tilde \gamma_{T,k',\e}^{\delta,H,i}) |_{Q_{H,i}\cap D} \in L^2(Q_{H,i}\cap D),
\end{eqnarray*}
and finally define the numerical correctors
\begin{eqnarray*}
\tilde C^{H}_{\e}\,=\,\sum_{i=1}^{I_H}\nabla  \tilde u_{\e}^{H,i} 1_{Q_{H,i}\cap D},
\qquad
\tilde C^{\delta,H}_{T,k',\e}\,=\,\sum_{i=1}^{I_H}\nabla \tilde u_{T,k',\e}^{\delta,H,i} 1_{Q_{H,i}\cap D}.
\end{eqnarray*}
From \eqref{eq:corr-rev1} and \eqref{eq:def-corr-reg-hom-tilde} on the one hand, 
and \eqref{eq:def-corr-reg-hom} and \eqref{eq:def-corr-reg-tilde} on the other hand, we infer that
\begin{equation*}
-\nabla \cdot A_\e \nabla (\gamma_{\e}^{H,i}-\tilde \gamma_{\e}^{H,i})\,=\, -\nabla \cdot (A_{\e,i}-A_\e) (M_H(\nabla u_\ho)+\nabla \tilde \gamma_{\e}^{H,i}) \quad \mbox{ in }Q_{H,i}\cap D,
\end{equation*}
and 
\begin{multline*}
{(T\e^2)^{-1}(\gamma_{T,1,\e}^{\delta,H,i}-\tilde \gamma_{T,1,\e}^{\delta,H,i}) -\nabla \cdot A_{\e}  \nabla (\gamma_{T,1,\e}^{\delta,H,i}-\tilde \gamma_{T,1,\e}^{\delta,H,i}) }
\\
\,=\,-\nabla \cdot (A_{\e,i}-A_\e) (M_H(\nabla u_\ho)+\nabla \tilde \gamma_{T,1,\e}^{\delta,H,i}) \quad \mbox{ in }Q_{H,i}^\delta\cap D,
\end{multline*}
so that energy estimates combined with the fact that $\tilde A(\cdot,\cdot)$ is $\kappa$-Lipschitz with respect to its first variable yields
\begin{eqnarray*}
\|\nabla \gamma_{\e}^{H,i}-\nabla \tilde \gamma_{\e}^{H,i}\|_{L^2(Q_{H,i}\cap D)}^2 &\lesssim &
(\kappa H)^2 |Q_{H,i}\cap D||M_{H,i}(\nabla u_\ho)|^2,\\
\|\nabla\gamma_{T,1,\e}^{\delta,H,i}-\nabla \tilde \gamma_{T,1,\e}^{\delta,H,i}\|_{L^2(Q_{H,i}\cap D)}^2 &\lesssim &
(\kappa H)^2 |Q_{H,i}^\delta\cap D||M_{H,i}(\nabla u_\ho)|^2,
\end{eqnarray*}
from which we deduce, using that $M_H$ is a contraction,
\begin{eqnarray*}
\|C^{H}_{\e}-\tilde C^{H}_{\e}\|_{L^2(D)} &\lesssim &
H \|\nabla u_\ho\|_{L^2(D)},\\
\|C^{\delta,H}_{T,k',\e}-\tilde C^{\delta,H}_{T,k',\e}\|_{L^2(D)} &\lesssim &
H \|\nabla u_\ho\|_{L^2(D)},
\end{eqnarray*}
uniformly in $\e,T,k'$, and the realization.
To prove \eqref{eq:Tstandard-corr2} it is therefore enough to prove that almost surely 
\begin{equation}\label{eq:Tstandard-corr3} 
\lim_{T\uparrow \infty,H\downarrow 0} \lim_{\e \downarrow 0} \|\tilde C_{T,k',\e}^{\delta,H}-\tilde C_\e^H \|_{L^2(D)} \,=\,0.
\end{equation}

\medskip

\step{4} Proof of \eqref{eq:Tstandard-corr3}.

\noindent Define $J_1^H:=\{1\leq i \leq I_H\,|\, Q_{H,i}^\delta \cap D \neq Q_{H,i}^\delta\}$ and $J_2^H:=\{1\leq i\leq I_H, \, i\notin J_1^H\}$. It suffices to prove that 
\begin{eqnarray}
\lim_{T\uparrow \infty,H\downarrow 0} \lim_{\e \downarrow 0} \|\tilde C_{T,k',\e}^{\delta,H}-\tilde C_\e^H \|_{L^2(D\cap \cup_{i\in J_1^H}Q_{H,i})} &=&0,\label{eq:Tstand-cor3-1}\\
\lim_{T\uparrow \infty,H\downarrow 0} \lim_{\e \downarrow 0} \|\tilde C_{T,k',\e}^{\delta,H}-\tilde C_\e^H \|_{L^2(\cup_{i\in J_2^H}Q_{H,i})} &=&0.\label{eq:Tstand-cor3-2}
\end{eqnarray}
We start with the proof of \eqref{eq:Tstand-cor3-1}. Energy estimates based on \eqref{eq:def-corr-reg-hom-tilde} and \eqref{eq:def-corr-reg-tilde} show that
\begin{multline*}
\|\tilde C_{T,k',\e}^{\delta,H}-\tilde C_\e^H \|_{L^2(D\cap \cup_{i\in J_1^H}Q_{H,i})}  \,\leq \,
\|\tilde C_{T,k',\e}^{\delta,H}\|_{L^2(D\cap \cup_{i\in J_1^H}Q_{H,i})} +\|\tilde C_\e^H \|_{L^2(D\cap \cup_{i\in J_1^H}Q_{H,i})}  \\
\,\lesssim \,
\|M_H(\nabla u_\ho)\|_{L^2(D\cap \cup_{i\in J_1^H}Q_{H,i})} 
\, \leq \, \|\nabla u_\ho\|_{L^2(D\cap \cup_{i\in J_1^H}Q_{H,i})} 
\end{multline*}
uniformly in $T$ and $\e$. Since the measure of the domain of integration $D\cap \cup_{i\in J_1^H}Q_{H,i}$ vanishes as $H\downarrow 0$ and since $\nabla u_\ho \in L^2(D)$, estimate~\eqref{eq:Tstand-cor3-1} follows.

\medskip

\noindent We turn now to \eqref{eq:Tstand-cor3-2}, which, by ellipticity of $\tilde A$, is equivalent to
$$
\lim_{T\uparrow \infty,H\downarrow 0} \lim_{\e \downarrow 0}  \sum_{i\in J_2^H} \int_{Q_{H,i}} (\tilde C_{T,k',\e}^{\delta,H}-\tilde C_\e^H)\cdot A_{\e,i} (\tilde C_{T,k',\e}^{\delta,H}-\tilde C_\e^H)\,=\,0.
$$
By expanding the square it is enough to prove the following four statements for all $i\in J_2^H$ almost surely:
\begin{eqnarray}
\lim_{\e \downarrow 0} \fint_{Q_{H,i}} \tilde C_\e^H\cdot A_{\e,i} \tilde C_\e^H
&=&M_{H,i}(\nabla u_\ho)\cdot A_\ho(x_i) M_{H,i}(\nabla u_\ho),\label{eq:Tstandard-corr21} 
\\
\lim_{\e \downarrow 0} \fint_{Q_{H,i}} \tilde C_{T,k',\e}^{\delta,H}\cdot A_{\e,i}\tilde  C_{\e}^{H}
&=&M_{H,i}(\nabla u_\ho)\cdot A_\ho(x_i) M_{H,i}(\nabla u_\ho),\label{eq:Tstandard-corr22} 
\\
\lim_{T\uparrow \infty}\lim_{\e \downarrow 0} \fint_{Q_{H,i}} \tilde C_{\e}^{H}\cdot A_{\e,i}\tilde  C_{T,k',\e}^{\delta,H}
&=&M_{H,i}(\nabla u_\ho)\cdot A_\ho(x_i) M_{H,i}(\nabla u_\ho),\label{eq:Tstandard-corr22bis} 
\\
\lim_{T\uparrow \infty}\lim_{\e \downarrow 0} \fint_{Q_{H,i}} \tilde C_{T,k',\e}^{\delta,H}\cdot A_{\e,i}\tilde C_{T,k',\e}^{\delta,H}
&=&M_{H,i}(\nabla u_\ho)\cdot A_\ho(x_i) M_{H,i}(\nabla u_\ho).\label{eq:Tstandard-corr23} 
\end{eqnarray}

\smallskip

\noindent Identity \eqref{eq:Tstandard-corr21} follows from the convergence of the energy associated with the homogenization of \eqref{eq:def-corr-reg-hom-tilde} and the fact that the associated homogenized coefficients are constant (with value $A_\ho(x_i)$, since on $Q_{H,i}$, $A_{\e,i}$ is given by the stationary coefficients $x\mapsto \tilde A(x_i,\frac{x}{\e})$). 

\smallskip

\noindent From Theorem~\ref{th:ATkRLp-approx}, we learn that almost surely
$$
\lim_{\e \downarrow 0} \fint_{Q_{H,i}} \tilde C_{T,k',\e}^{\delta,H}\cdot A_{\e,i}\tilde C_{T,k',\e}^{\delta,H}
\,=\, M_{H,i}(\nabla u_\ho)\cdot A_{T,k'} (x_i) M_{H,i}(\nabla u_\ho),
$$
where $A_{T,k'}(x_i)$ is the approximation of the homogenized coefficients associated with the stationary field $x\mapsto \tilde A(x_i,x)$ of Definition~\ref{def:extra-hom}.
Estimate~\eqref{eq:Tstandard-corr23} is now a consequence of Theorem~\ref{th:spec}.

\smallskip

\noindent We conclude with the proofs of \eqref{eq:Tstandard-corr22} and \eqref{eq:Tstandard-corr22bis}. 
By Meyers' estimate, on each $Q_{H,i}$, $\tilde C_{T,k',\e}^{\delta,H}$ is uniformly bounded in $L^{\tilde q}(Q_{H,i})$ for some $\tilde q>2$, so that $\tilde C_{\e}^{H}\cdot A_{\e,i}\tilde  C_{T,k',\e}^{\delta,H}$ and $\tilde C_{T,k',\e}^{\delta,H}\cdot A_{\e,i}\tilde  C_{\e}^{H}$ are equi-integrable in $L^1(Q_{H,i})$. Combined with compensated compactness, it is therefore enough to prove the
following almost sure weak convergences in $L^2(Q_{H,i},\R^d)$:
\begin{eqnarray*}
&&A_{\e,i}\tilde  C_{\e}^{H}\rightharpoonup A_\ho(x_i) M_{H,i}(\nabla u_\ho), \quad \tilde C_{T,k',\e}^{\delta,H}\rightharpoonup M_{H,i}(\nabla u_\ho),\\
&&A_{\e,i}\tilde C_{T,k',\e}^{\delta,H}\rightharpoonup A_\ho(x_i) M_{H,i}(\nabla u_\ho), \quad \tilde  C_{\e}^{H}\rightharpoonup M_{H,i}(\nabla u_\ho).
\end{eqnarray*}
The first and fourth convergences are a consequence of homogenization, which implies the weak convergence of the fluxe and of the gradient of the solution.
For the second and third convergences, recall that on $Q_{H,i}$, $\tilde C_{T,k',\e}^{\delta,H}\,=\,\nabla \tilde \gamma_{T,k',\e}^{\delta,H,i}+M_{H,i}(\nabla u_\ho)$. Denote by $\tilde \gamma_{T,k',\e}^i \in H^1_\loc(\R^d)$ the Richardson extrapolations of the unique solution $\tilde \gamma_{T,1,\e}^i \in H^1_\loc(\R^d)$ of
$$
(T\e^2)^{-1}\tilde \gamma_{T,1,\e}^{i} -\nabla \cdot \tilde A(x_i,\frac{\cdot}{\e})  \big( M_{H,i}(\nabla u_{\ho})+\nabla \tilde \gamma_{T,1,\e}^{\i} \big)\,=\,0 \quad \mbox{ in }\R^d,
$$
satisfying \eqref{eq:class-1}.
By \eqref{eq:th-approx-TRk} in Remark~\ref{rem:approx-TRk}, we then have (after rescaling by $\e$):
$$
\int_{Q_{H,i}} |\nabla \tilde \gamma_{T,k',\e}^{\delta,H,i}-\nabla\tilde \gamma_{T,k',\e}^{i} |^2
\,\lesssim \, T \exp\Big( -c (\delta-1)\frac{H}{\e} \frac{1}{\sqrt{2^{k'-1}T}}\Big).
$$
On the one hand, $\nabla \tilde \gamma_{T,k',\e}^{\delta,H,i}$ weakly converges to zero as $\e \downarrow 0$ if and only if $\nabla\tilde \gamma_{T,k',\e}^{i}$ weakly converges to zero. But the latter is a consequence of the ergodic theorem since
by construction $\nabla\tilde \gamma_{T,k',\e}^{i}$ is the $\e$-rescaled version of the gradient of the corrector,
which is stationary and the expectation of which vanishes. This proves the fourth convergence.
On the other hand, this also implies that
the weak limits as $\e \downarrow 0$ of $A_{\e,i}\tilde C_{T,k',\e}^{\delta,H}$ and $A_{\e,i}(M_{H,i}(\nabla u_\ho)+\nabla\tilde \gamma_{T,k',\e}^{i})$ are the same. By the ergodic theorem, $A_{\e,i}(M_{H,i}(\nabla u_\ho)+\nabla\tilde \gamma_{T,k',\e}^{i})$ weakly converges to $\expec{\tilde A(x_i,\cdot)(M_{H,i}(\nabla u_\ho)+\nabla \tilde \gamma_{T,k,1})}$, so that 
$$
\lim_{\e \downarrow 0} \fint_{Q_{H,i}} \tilde C_{\e}^{H}\cdot A_{\e,i}\tilde  C_{T,k',\e}^{\delta,H}
\,=\,M_{H,i}(\nabla u_\ho)\cdot \expec{\tilde A(x_i,\cdot)(M_{H,i}(\nabla u_\ho)+\nabla \tilde\gamma_{T,k,1})}.
$$
By Theorem~\ref{th:spec}, $\nabla \tilde\gamma_{T,k,1}$ converges strongly to $\nabla \tilde\gamma$ (the gradient of the corrector associated with $\tilde A(x_i,\cdot)$ in direction $M_{H,i}(\nabla u_\ho)$) in $L^2(\Omega)$.
Hence
\begin{multline*}
\lim_{T\uparrow \infty} M_{H,i}(\nabla u_\ho)\cdot \expec{\tilde A(x_i,\cdot)(M_{H,i}(\nabla u_\ho)+\nabla \tilde\gamma_{T,k,1})}\\
=\,M_{H,i}(\nabla u_\ho)\cdot A_\ho(x_i)M_{H,i}(\nabla u_\ho),
\end{multline*}
which shows the third convergence, and concludes the proof.
\end{proof}

\subsection{Discretization of the analytical framework}\label{sec:HMM}

In this paragraph we discuss a possible way to discretize the analytical framework introduced in this section.
We only present the ``direct approach", which allows us to obtain a new variant of the HMM (see \cite{E-05,Abdulle-05,E-07,E-12},\cite{Arbogast-00} for the original method) with regularization and extrapolation. In particular, we shall prove the almost-sure convergence of this numerical approximation under Hypothesis~\ref{hypo:reg}. Note that in \cite{Gloria-06b,Gloria-07b,Gloria-12b} we have also introduced a ``dual approach", which allows one to recover variants of the MsFEM (see \cite{Hou-97,Hou-99,Efendiev-00,Hou-04,Allaire-05,Efendiev-Hou-09}). 
We defer the analysis of the dual approach, which is more technical and requires additional ideas, to a subsequent contribution.

\medskip

\noindent In order to solve \eqref{eq:Tkrhoepsdelta}, one has first to approximate the coefficients $A^\delta_{T,k,\rho,\e}$.
Let $\xi \in \R^d$.
For all $x\in D$, denote by $V_h(x)$ a family of dense subspaces of $H^1_0(Q_{\delta \rho}\cap T_{-x}D)$,
and consider the Galerkin approximation $v_{T,k,h}^{\delta\rho,\e}(x,\cdot),{v'}_{T,k,h}^{\delta\rho,\e}(x,\cdot)$
of $v_{T,k}^{\delta\rho,\e}(x,\cdot),{v'}_{T,k}^{\delta\rho,\e}(x,\cdot)$ in $V_h(x)$ (cf. Definition~\ref{def:reg-ana}), and, following \eqref{eq:reg-approx}, set:
\begin{multline*}\label{eq:reg-approx-h}
\xi'\cdot A_{T,k,\rho,\e}^{\delta,h}(x)\xi :=\int_{Q_{\rho}\cap T_{-x}D} (\xi' +P_{\rho,x}(\nabla_y {v'}_{T,k,h}^{\delta\rho,\e})(x,y))\\ \cdot A_\e(x+y)(\xi+P_{\rho,x}(\nabla_y v_{T,k,h}^{\delta\rho,\e})(x,y))\mu_{\rho}(y)dy,
\end{multline*}
where $P_{\rho,h}$ is as in \eqref{def:mean-free-mu}.
\begin{rem}\label{rem:bdy-layer}
Since $\tilde A$ is Lipschitz-continuous in its first variable, $A_\ho$ is Lipschitz-continuous as well, and we can 
replace $A_{T,k,\rho,\e}^{\delta,h}(x)$ for $x\in D$ such that $Q_{\delta \rho}\cap T_{-x}D\neq Q_{\delta \rho}$
by the value $A_{T,k,\rho,\e}^{\delta,h}(x')$ at a point $x'$ such that $Q_{\delta \rho}\cap T_{-x'}D= Q_{\delta \rho}$
and such that $|x-x'|\lesssim \rho$ (which is possible since $D$ has a Lipschitz boundary). The convergence analysis is similar and this allows one to avoid possible boundary layers on $Q_{\delta}\cap T_{-x}D$, at the expense of the regularity of $A_\ho$.
\end{rem}
\noindent The same argument as in Step~1 of the proof of Theorem~\ref{th:main-reg} shows that the approximation 
$A_{T,k,\rho,\e}^{\delta,h}$ of $A_{T,k,\rho,\e}^{\delta}$ is uniformly elliptic for symmetric coefficients (for non-symmetric coefficients, the discretization yields an additional error term which is uniform).
Let now $V_H$ be a family of dense Galerkin subspaces of $H^1_0(D)$, and denote by $u_{T,k,\rho,\e}^{\delta,h,H}$
the Galerkin approximation in $V_H$ of the weak solution $u_{T,k,\rho,\e}^{\delta,h} \in H^1_0(D)$ of 
$$
-\nabla \cdot A_{T,k,\rho,\e}^{\delta,h}\nabla u_{T,k,\rho,\e}^{\delta,h}\,=\,f,
$$
which exists and is unique by the uniform ellipticity of $A_{T,k,\rho,\e}^{\delta,h}$.
As a first convergence result we have:
\begin{equation}\label{eq:conv-HMM-mod}
\lim_{T\uparrow \infty,\rho,H\downarrow 0} \lim_{\e\downarrow 0} \lim_{h\downarrow 0}
\|u_{T,k,\rho,\e}^{\delta,h,H}-u_\ho\|_{H^1_0(D)}\,=\,0,
\end{equation}
where the order of the limits in $T,\rho,H$ is irrelevant.
Indeed, by the triangle inequality,
\begin{multline*}
\|u_{T,k,\rho,\e}^{\delta,h,H}-u_\ho\|_{H^1_0(D)}\,\leq \, \|u_{T,k,\rho,\e}^{\delta,h,H}-u_{T,k,\rho,\e}^{\delta,h}\|_{H^1_0(D)}+\|u_{T,k,\rho,\e}^{\delta,h}-u_{T,k,\rho,\e}^{\delta}\|_{H^1_0(D)}
\\
+\|u_{T,k,\rho,\e}^{\delta}-u_\ho\|_{H^1_0(D)}.
\end{multline*}
Combined with C\'ea's lemma and the triangle inequality, the first RHS term is estimated as follows:
\begin{eqnarray*}
 \|u_{T,k,\rho,\e}^{\delta,h,H}-u_{T,k,\rho,\e}^{\delta,h}\|_{H^1_0(D)}&\lesssim & \inf_{v_H\in V_H}
\|u_{T,k,\rho,\e}^{\delta,h}-v_H\|_{H^1_0(D)}\\
&\leq & \|u_{T,k,\rho,\e}^{\delta,h}-u_{T,k,\rho,\e}^{\delta}\|_{H^1_0(D)}
+\|u_{T,k,\rho,\e}^{\delta}-u_\ho\|_{H^1_0(D)} \\
&&+\inf_{v_H\in V_H}
\|u_{\ho}-v_H\|_{H^1_0(D)}
\end{eqnarray*}
so that we have
\begin{multline*}
\|u_{T,k,\rho,\e}^{\delta,h,H}-u_\ho\|_{H^1_0(D)}\,\leq \,2\|u_{T,k,\rho,\e}^{\delta,h}-u_{T,k,\rho,\e}^{\delta}\|_{H^1_0(D)}+2\|u_{T,k,\rho,\e}^{\delta}-u_\ho\|_{H^1_0(D)}
\\
+\inf_{v_H\in V_H} \|u_{\ho}-v_H\|_{H^1_0(D)}.
\end{multline*}
The first RHS term vanishes as $h\downarrow 0$ by Galerkin approximation, the second RHS term converges to zero by Theorem~\ref{th:main-reg}, whereas the last RHS term vanishes as $H\downarrow 0$ by Galerkin approximation as well: \eqref{eq:conv-HMM-mod} follows.

\medskip
\noindent We turn now to the numerical corrector\footnote{The number of sub/super-scripts may not ease the reading, and we encourage the reader to simply consider $H_1=H_2=\rho$ everywhere --- although we shall need all the parameters in the proof.}.
For clarity we denote by $H_1$ the discretization parameter of the Galerkin subspaces $V_{H_1}$ of $H^1_0(D)$, and 
denote by $u_{T,k,\rho,\e}^{\delta,h,H_1}$ the Galerkin approximation of $u_{T,k,\rho,\e}^{\delta,h}$ in $V_{H_1}$.
We then follow the strategy of Definition~\ref{def:corr-reg} to introduce the desired numerical corrector.
Let $H_2>0$, $I_{H_2}\in \N$, and let $\{Q_{H_2,i}\}_{i\in [\![1,I_{H_2}]\!]}$ be a covering of $D$ in disjoint subdomains of diameter of order $H_2$.
We denote by $(M_{H_2})$ the family of approximations of the identity on $L^2(D)$ associated with
$Q_{H_2,i}$.
Let $\delta>1$ be as in Theorem~\ref{th:main-reg}, and for all $i\in [\![1,I_{H_2}]\!]$, set
$$
Q_{H_2,i}^\delta:=\{x\in \R^d\,|\,d(x,Q_{H_2,i})<(\delta-1)H_2\}.
$$
For all $i\in [\![1,I_{H_2}]\!]$, let $V_{i,H_2,h}$ be a family of dense subspaces of $H^1_0(Q_{H_2,i}^\delta\cap D)$.
With the notation of Corollary~\ref{coro:reg}, we define the functions $\tilde \gamma_{T,k,1,\rho,\e}^{\delta,h,H_1,H_2,i}$ and $\tilde \gamma_{T,k,1,\rho,\e}^{\delta,H_1,H_2,i}$ as the unique weak 
solutions in $H^1_0(Q_{H_2,i}^\delta\cap D)$ of 
\begin{eqnarray}\label{eq:def-corr-reg-disc}
(T\e^2)^{-1}\tilde \gamma_{T,k,1,\rho,\e}^{\delta,h,H_1,H_2,i}-\nabla \cdot A_\e \bigg( M_{H_2,i}(\nabla u_{T,k,\rho,\e}^{\delta,h,H_1})+\nabla \tilde \gamma_{T,k,1,\rho,\e}^{\delta,h,H_1,H_2,i}\bigg)&=&0,\\
(T\e^2)^{-1}\tilde \gamma_{T,k,1,\rho,\e}^{\delta,H_1,H_2,i}-\nabla \cdot A_\e \bigg( M_{H_2,i}(\nabla u_{T,k,\rho,\e}^{\delta,H_1})+\nabla \tilde \gamma_{T,k,1,\rho,\e}^{\delta,H_1,H_2,i}\bigg)&=&0.
\end{eqnarray}
We then define the numerical corrector $\gamma_{T,k,1,\rho,\e}^{\delta,h,H_1,H_2,i}$ associated with $u_{T,k,\rho,\e}^{\delta,h,H_1}$  as the Galerkin approximation in $V_{i,H_2,h}$ of $\tilde \gamma_{T,k,1,\rho,\e}^{\delta,h,H_1,H_2,i}$ in $H^1_0(Q_{H_2,i}^\delta\cap D)$.
The functions $\gamma_{T,k,k',\rho,\e}^{\delta,h,H_1,H_2,i}$ and $\tilde \gamma_{T,k,k',\rho,\e}^{\delta,H_1,H_2,i}$ are obtained for all $k'\in \N$ by Richardson extrapolation, cf. Definition~\ref{def:extra}.
We also set for all $k' \in \N$ and for all $1\leq i\leq I_{H_2}$,
\begin{eqnarray*}
\nabla \tilde u_{T,k,k',\rho,\e}^{\delta,H_1,H_2,i}&:=&M_{H_2}(\nabla u_{T,k,\rho,\e}^{\delta,H_1})|_{Q_{H_2,i}\cap D}+(\nabla \tilde \gamma_{T,k,k',\rho,\e}^{\delta,H_1,H_2,i}) |_{Q_{H_2,i}\cap D} \in L^2(Q_{H_2,i}\cap D),\\
\nabla u_{T,k,k',\rho,\e}^{\delta,h,H_1,H_2,i}&:=&M_{H_2}(\nabla u_{T,k,\rho,\e}^{\delta,h,H_1})|_{Q_{H_2,i}\cap D}+(\nabla \gamma_{T,k,k',\rho,\e}^{\delta,h,H_1,H_2,i}) |_{Q_{H_2,i}\cap D} \in L^2(Q_{H_2,i}\cap D),
\end{eqnarray*}
and finally define numerical correctors
\begin{equation}\label{eq:def:corrector-HMM-disc}
\tilde C^{\delta,H_1,H_2}_{T,k,k',\rho,\e}\,:=\,\sum_{i=1}^{I_{H_2}}\nabla \tilde u_{T,k,k',\rho,\e}^{\delta,H_1,H_2,i} 1_{Q_{H_2,i}\cap D}, \quad
C^{\delta,h,H_1,H_2}_{T,k,k',\rho,\e}\,:=\,\sum_{i=1}^{I_{H_2}}\nabla u_{T,k,k',\rho,\e}^{\delta,h,H_1,H_2,i} 1_{Q_{H_2,i}\cap D}.
\end{equation}
Our second result is the following almost sure convergence:
\begin{equation}\label{eq:conv-HMM-corr-mod}
\lim_{T\uparrow \infty,\rho,H_1,H_2\downarrow 0} \lim_{\e\downarrow 0} \lim_{h\downarrow 0}
\|C^{\delta,h,H_1,H_2}_{T,k,k',\rho,\e}-\nabla u_\e\|_{L^2(D)}\,=\,0,
\end{equation}
where the order of the limits in $T,\rho,H_1,H_2$ is irrelevant. In particular, one may take $\rho=H_1=H_2$ in practice.

\noindent
By the triangle inequality,
\begin{eqnarray*}
 \|C^{\delta,h,H_1,H_2}_{T,k,k',\rho,\e}-\nabla u_\e\|_{L^2(D)} &\leq &  \|C^{\delta,h,H_1,H_2}_{T,k,k',\rho,\e}-\tilde C^{\delta,H_1,H_2}_{T,k,k',\rho,\e}\|_{L^2(D)} 
\\
&&+\|\tilde C^{\delta,H_1,H_2}_{T,k,k',\rho,\e}-C^{\delta,H_2}_{T,k,k',\rho,\e}\|_{L^2(D)} +
\|C^{\delta,H_2}_{T,k,k',\rho,\e}-\nabla u_\e\|_{L^2(D)},
\end{eqnarray*}
where $C^{\delta,H_2}_{T,k,k',\rho,\e}$ is the corrector of Definition~\ref{def:corr-reg} (with $H_2$ in place of $H$).
We start the analysis with the first RHS term.
Note that $C^{\delta,h,H_1,H_2}_{T,k,k',\rho,\e}$ depends doubly on $h$: once through the Galerkin approximation in $V_{i,H_2,h}$ and once through the approximation $M_{H_2,i}(\nabla u_{T,k,\rho,\e}^{\delta,h,H_1})$
of $M_{H_2,i}(\nabla u_{T,k,\rho,\e}^{\delta,H_1})$.
Since $\lim_{h\downarrow 0} M_{H_2,i}(\nabla u_{T,k,\rho,\e}^{\delta,h,H_1})=M_{H_2,i}(\nabla u_{T,k,\rho,\e}^{\delta,H_1})$ and since the solution of \eqref{eq:def-corr-reg-disc} depends continuously on $M_{H_2,i}(\nabla u_{T,k,\rho,\e}^{\delta,h,H_1})$ in $H^1_0(Q_{H_2,i}^\delta\cap D)$, the Galerkin approximation  $\gamma_{T,k,k',\rho,\e}^{\delta,h,H_1,H_2,i}$ converges to $\tilde \gamma_{T,k,k',\rho,\e}^{\delta,H_1,H_2,i}$ in $H^1_0(Q_{H_2,i}^\delta\cap D)$ as $h\downarrow 0$. Hence,
\begin{equation*}
\lim_{h\downarrow 0} \|C^{\delta,h,H_1,H_2}_{T,k,k',\rho,\e}-\tilde C^{\delta,H_1,H_2}_{T,k,k',\rho,\e}\|_{L^2(D)} \,=\,0.
\end{equation*}
We continue with the second term, which measures the error between the corrector associated with $u_{T,k,\rho,\e}^{\delta}$ and the corrector associated with $u_{T,k,\rho,\e}^{\delta,H_1}$. Since the solution
$\tilde \gamma_{T,k,k',\rho,\e}^{\delta,h,H_1,H_2,i}$ of \eqref{eq:def-corr-reg-disc} depends linearly in $H^1_0(Q_{H_2,i}^\delta\cap D)$ on $M_{H_2}(\nabla u_{T,k,\rho,\e}^{\delta,H_1})$, and $M_{H_2}$ is a continuous linear map on $L^2(D)$, we have
\begin{equation*}
\|\tilde C^{\delta,H_1,H_2}_{T,k,k',\rho,\e}- C^{\delta,H_2}_{T,k,k',\rho,\e}\|_{L^2(D)} \,\lesssim \,
\|\nabla u_{T,k,\rho,\e}^{\delta,H_1}-\nabla u_{T,k,\rho,\e}^{\delta}\|_{L^2(D)}.
\end{equation*}
By Theorem~\ref{th:main-reg}, this second term converges uniformly to zero wrt to $H_2$.
For the third and last RHS terms we finally appeal to Theorem~\ref{th:corr-reg}.
This concludes the qualitative convergence analysis of the HMM with regularization and extrapolation under Hypothesis~\ref{hypo:reg}.

\subsection{Numerical analysis for periodic coefficients}\label{sec:Num.Anal.loc_per}

In this last paragraph we make the qualitative convergence analysis of the HMM with regularization and extrapolation quantitative for periodic coefficients. The choice of the periodic coefficients allows us to combine our quantitative analysis of Sections~\ref{sec:spectral} and~\ref{sec:tests} with quantitative two-scale expansions. 
In particular, in order to complete the analysis in the case of random Poisson inclusions, we would need a quantitative two-scale expansion, which is not yet known to hold (see however the recent progress in that direction for discrete elliptic equations \cite{Gloria-Neukamm-Otto-11}).

\begin{theo}\label{thm:full-discrete}
Let $A\in \Mab$ be periodic coefficients.
Let $f$ and $D$ be smooth enough so that the solution $u_\ho\in H^1_0(D)$ of 
the associated homogenized problem
$$
-\nabla \cdot A_\ho \nabla u_\ho \,=\,f
$$
is in $W^{2,\infty}(D)$. Fix $\delta>1$.
Let $V_H$ be a P1-finite element subspace of $H^1_0(D)$ of dimension $I_H$ associated with a tessellation $\mathcal T_H$ of $D$ of meshsize $H>0$. For every element $T_{H,i}$ of $\calT_H$, let $\calT_{H,i,h}^\delta$ be a tessellation of $T_{H,i}^\delta\cap D$ of meshsize $h>0$ with $T_{H,i}^\delta=\{x\in \R^d\,|\,d(x,T_{H,i})\leq (\delta-1)H\}$, and let $V_{H,h,i}$ be the associated P1-finite element subspace of $H^1_0(T_{H,i}^\delta\cap D)$.
For all $k\in \N$, all filters of order $p\geq 0$, $T>0$ and $\e>0$, we denote by $A_{T,k,H,\e}^{\delta,h}$ the (piecewise constant) approximation of $A_{T,k,\rho,\e}^\delta$ on $\calT_H$ (cf. \eqref{eq:reg-approx}) obtained as follows. 
For all $1\leq i\leq I_H$ such that $T_{H,i}^\delta\subset D$, we define $A_{T,k,H,\e}^{\delta,h}|_{T_{H,i}}$ as the 
approximation \eqref{eq:reg-approx}, with $T_{H,i}^\delta$ in place of $Q_{H,i}^\delta$ and $V_{H,h,i}$ in place of $H^1_0(Q_{H,i}^\delta\cap D)$. If $1\leq i\leq I_H$  is such that $T_{H,i}^\delta\not\subset D$, we follow Remark~\ref{rem:bdy-layer} and set $A_{T,k,H,\e}^{\delta,h}|_{T_{H,i}}:=A_{T,k,H,\e}^{\delta,h}|_{T_{H,j}}$ for some
$1\leq j\leq I_H$ such that $T_{H,j}^\delta\subset D$ and $d(T_{H,i},T_{H,j})\lesssim H$.
For $\delta = \frac{3}{2}$, $T=\frac{H}{\e}$ and $p\ge 2k-1$, we then have 
\begin{equation}\label{eq:quant-HMM-1}
|A_{T,k,H,\e}^{\delta,h}-A_\ho|\,\lesssim \,  \big(\frac{\e}{H}\big)^{2k}+\Big(\frac{h}{\e}\Big)^2,
\end{equation}
Denote by $u_{T,k,H,\e}^{\delta,h}$ the Galerkin approximation in $V_H$ of 
$$
-\nabla \cdot A_{T,k,H,\e}^{\delta,h}\nabla u_{T,k,H,\e}^{\delta,h}\,=\,f.
$$
We have
\begin{equation}\label{eq:quant-HMM-2}
\|u_{T,k,H,\e}^{\delta,h}-u_\ho\|_{H^1(D)}\,\lesssim \, H+\big(\frac{\e}{H}\big)^{2k}+\Big(\frac{h}{\e}\Big)^2.
\end{equation}
Let $C^{\delta,h,H}_{T,k,\e}$ denote the numerical corrector $C^{\delta,h,H_1,H_2}_{T,k,k',\rho,\e}$ defined in \eqref{eq:def:corrector-HMM-disc}, for the covering $\calT_H$ of $D$ and with $\rho=H_1=H_2=H$ and $k=k'$. We have
\begin{equation}\label{eq:quant-HMM-8}
\|\nabla u_\e-C^{\delta,h,H}_{T,k,\e}\|_{L^2(D)} \, \lesssim \, \sqrt{\e}+H+\frac{h}{\e}+\big(\frac{\e}{H}\big)^{\min\{2,k\}}.
\end{equation}
\qed
\end{theo}
\noindent Estimates~\eqref{eq:quant-HMM-1}, \eqref{eq:quant-HMM-2}, and \eqref{eq:quant-HMM-8} are the best estimates available in the literature in terms of the resonance error $\frac{\e}{H}$ for a numerical homogenization method with an oversampling of fixed size $\delta>1$ (taken $\frac{3}{2}$ for the reasoning).
Note that estimate~\eqref{eq:quant-HMM-8} may be overly pessismistic due to our currently poor understanding of the boundary layer.  
\begin{rem}\label{rem:whole-space}
For problems without boundary, say on the whole space with a zero order term of magnitude $1$,
the following improved versions of the estimates of Theorem~\ref{thm:full-discrete} hold for coefficients that are periodic or in the Kozlov class of almost periodic functions:
\begin{eqnarray*}
|A_{T,k,H,\e}^{\delta,h}-A_\ho|&\lesssim & \big(\frac{\e}{H}\big)^{2k}+\Big(\frac{h}{\e}\Big)^2,\\
\|u_{T,k,H,\e}^{\delta,h}-u_\ho\|_{H^1(\R^d)}&\lesssim & H+\big(\frac{\e}{H}\big)^{2k}+\Big(\frac{h}{\e}\Big)^2,\\
\|\nabla u_\e-C^{\delta,h,H}_{T,k,\e}\|_{L^2(\R^d)} & \lesssim & {\e}+H+\frac{h}{\e}+\big(\frac{\e}{H}\big)^{k},
\end{eqnarray*}
the proofs of which are left to the reader (for almost periodic coefficients, the two-scale expansion
on the whole space corresponding to \eqref{eq:quant-HMM-7} with $\e$ on the RHS is proved in \cite{Kozlov-78} 
in this context).
\end{rem}
\begin{proof}
We split the proof into three steps.

\medskip
\step{1} Proof of \eqref{eq:quant-HMM-1}.

\noindent By Theorem~\ref{th:quant-estim-Perio}
and a standard Galerkin approximation result, we have for all $1\leq i \leq I_H$ such that $T_{H,i}^\delta \subset D$, 
\begin{multline*}
|A_{T,k,H,\e}^{\delta,h}-A_\ho|\,\lesssim \, \big(\frac{\e}{H}\big)^{p+1}+\big(\frac{\e}{H}\big)^{2k}+\big(\frac{H}{\e}\big)^{\frac{1}{4}}\exp\Big(-c\frac{1}{\sqrt{2^{k+1}}}\sqrt{\frac{H}{\e}}\Big)+\Big(\frac{h}{\e}\Big)^2.
\end{multline*}
Since $A_\ho$ is constant on $D$ (and not only Lipschitz), this also yields for those $1\leq i \leq I_H$ such that $T_{H,i}^\delta \not \subset D$,
\begin{equation*}
|A_{T,k,H,\e}^{\delta,h}-A_\ho|\,\lesssim \, \big(\frac{\e}{H}\big)^{2k}+\Big(\frac{h}{\e}\Big)^2.
\end{equation*}
(If $A_\ho$ were Lipschitz and non-constant, there would be an additional term of order $H$ in this estimate.)
This proves \eqref{eq:quant-HMM-1}, recalling that $p+1=2k$ and neglecting the exponential term.

\medskip

\step{2} Proof of \eqref{eq:quant-HMM-2}.

\noindent 
Estimate \eqref{eq:quant-HMM-2} is a direct consequence of \eqref{eq:quant-HMM-1} and of the regularity of $u_\ho\in H^2(D)$.
Indeed, let $v_{T,k,H,\e}^{\delta,h}$ be the weak solution in $H^1_0(D)$ of 
$$
-\nabla \cdot A_{T,k,H,\e}^{\delta,h}\nabla v_{T,k,H,\e}^{\delta,h}\,=\,f.
$$
We then appeal to \eqref{eq:ref-pr-HMM-corr-disc} in the proof of Lemma~\ref{6:lemma:ptwise-conv-sol}, which takes the form 
$$
\|\nabla u_\ho-\nabla v_{T,k,H,\e}^{\delta,h}\|_{L^2(D)} \, \lesssim \, \|A_{T,k,H,\e}^{\delta,h}-A_\ho\|_{L^\infty(D)}\|\nabla u_\ho\|_{L^2(D)}.
$$
Combined with \eqref{eq:quant-HMM-1}, this yields
\begin{equation}\label{eq:quant-HMM-3}
\|\nabla u_\ho-\nabla v_{T,k,H,\e}^{\delta,h}\|_{L^2(D)} \, \lesssim \, 
\big(\frac{\e}{H}\big)^{2k}+\Big(\frac{h}{\e}\Big)^2.
\end{equation}
Finally, by C\'ea's lemma and the triangle inequality,
\begin{eqnarray*}
\|u_{T,k,H,\e}^{\delta,h}-u_\ho\|_{H^1(D)}&\leq & \|u_{T,k,H,\e}^{\delta,h}-v_{T,k,H,\e}^{\delta,h}\|_{H^1(D)}
+\|v_{T,k,H,\e}^{\delta,h}-u_\ho\|_{H^1(D)}\\
&\lesssim & \inf_{v_H\in V_H}  \|v_H-v_{T,k,H,\e}^{\delta,h}\|_{H^1(D)}
+\|v_{T,k,H,\e}^{\delta,h}-u_\ho\|_{H^1(D)} \\
&\leq &  \inf_{v_H\in V_H}  \|v_H-u_\ho\|_{H^1(D)}+2\|v_{T,k,H,\e}^{\delta,h}-u_\ho\|_{H^1(D)} ,
\end{eqnarray*}
so that \eqref{eq:quant-HMM-2} follows from \eqref{eq:quant-HMM-3} and a Galerkin approximation result.

\medskip

\step{3} Proof of \eqref{eq:quant-HMM-8}.

\noindent Let denote by $\phi^i$ the periodic corrector associated with $\tilde A$ in direction $e_i$. On the one hand, we have the following quantitative two-scale expansion:
\begin{equation}\label{eq:quant-HMM-7}
\|\nabla u_\e-\nabla u_\ho-\sum_{i=1}^d \nabla_i u_\ho \nabla \phi^i(\e^{-1}\cdot)\|_{L^2(D)} \, \lesssim \, \sqrt{\e}.
\end{equation}
On the other hand, the combination of \eqref{eq:quant-HMM-3} with Remark~\ref{rem:approx-TRk} and a Galerkin approximation result yields for all $1\leq i \leq I_H$ such that $T_{H,i}^\delta\subset D$,
\begin{multline}\label{eq:quant-HMM-5}
\|\nabla u_\ho+\sum_{i=1}^d \nabla_i u_\ho \nabla \phi^i(\e^{-1}\cdot)-C^{\delta,h,H}_{T,k,\e}\|_{L^2(T_{H,i})}^2\,
\\
\lesssim\, |T_{H,i}|\Big(|M_{H,i}(\nabla u_\ho-\nabla u_{T,k,H,\e}^{\delta,h})|+ \frac{h}{\e}+H+\big(\frac{\e}{H}\big)^{k}+\frac{h}{\e} \Big)^2,
\end{multline}
and 
\begin{multline}\label{eq:quant-HMM-52}
\|\nabla u_\ho+\sum_{i=1}^d \nabla_i u_\ho \nabla \phi^i(\e^{-1}\cdot)-C^{\delta,h,H}_{T,k,\e}\|_{L^2(T_{H,i})}^2\,
\\
\lesssim\, |T_{H,i}|\Big(|M_{H,i}(\nabla u_\ho-\nabla u_{T,k,H,\e}^{\delta,h})|+ \frac{h}{\e}+H+\frac{\e}{H}+\frac{h}{\e} \Big)^2
\end{multline}
otherwise (on such elements $T_{H,i}$ the resonance error remains a priori of order $\frac{\e}{H}$).
The combination of  \eqref{eq:quant-HMM-7}, \eqref{eq:quant-HMM-3}, and \eqref{eq:quant-HMM-52} 
on $N_H(\partial D):=\{x\in D\,|\,d(x,\partial D)\leq H\}$ and \eqref{eq:quant-HMM-5} 
on $D\setminus N_H(\partial D)$ then yields the desired estimate \eqref{eq:quant-HMM-8}, and also the interior estimate
\begin{equation*}
\|\nabla u_\e-C^{\delta,h,H}_{T,k,\e}\|_{L^2(D\setminus N_H(\partial D))} \, \lesssim \, \sqrt{\e}+H+\frac{h}{\e}+\big(\frac{\e}{H}\big)^{k}
\end{equation*}
in the spirit of Remark~\ref{rem:whole-space}.
\end{proof}

%%-----------------------------

\section*{Acknowledgements}

The first author acknowledges financial support from the European Research Council under
the European Community's Seventh Framework Programme (FP7/2014-2019 Grant Agreement
QUANTHOM 335410).

\bibliographystyle{plain}

\end{document}